%% file: main.tex
\documentclass{article}
\usepackage{PRIMEarxiv}
\input{preamble}
\title{Mixed Unfitted finite element approximation of band structures in photonic crystals}
\input{authors}
\date{}
\begin{document}
\maketitle
\begin{abstract}
Photonic crystals are periodic dielectric structures whose optical properties are characterized by Maxwell eigenvalue problems with discontinuous material coefficients. For arbitrarily shaped inclusions, their numerical solution remains challenging due to costly mesh generation and remeshing as the geometry changes. In this paper, we address this challenge by developing a mixed unfitted finite element method based on the Floquet–Bloch transform. Our method uses a uniform background mesh independent of the inclusion geometry, thereby avoiding interface-fitted meshing and allowing us to impose the periodic boundary condition through a direct identification of degrees of freedom. A key feature of our construction is that the resulting finite element pairs preserve the de Rham commuting relation for the shifted operator. This structure enables us to establish an inf–sup condition uniformly with respect to the cut position and, in particular, to rule out spurious modes. We further establish optimal convergence of the eigenvalues and eigenspaces, with constants independent of both the mesh size and the position of the interface. Finally, we present extensive numerical experiments in two and three dimensions to validate our theoretical results and demonstrate the robustness of the proposed method.
\end{abstract}

\input{section1}
\input{section2}
\input{section3}
\input{section4}

\input{section5}
\input{section6}

\input{section7}



\section*{Funding}
This work was supported by the Melbourne Research Scholarship
from the University of Melbourne. 

\bibliographystyle{cas-model2-names}
\bibliography{references}
\end{document}

%% file: preamble.tex
\usepackage[utf8]{inputenc}
\usepackage[T1]{fontenc}
\usepackage{graphicx,amsmath,amssymb}
\usepackage{xcolor,etoolbox}
\usepackage{url,booktabs,makecell,multirow,array,colortbl,dcolumn,stfloats}
\usepackage{nicefrac,microtype}
\usepackage{hyperref}
\usepackage[numbers,sort&compress]{natbib}
\usepackage{amsthm}
\usepackage{algorithmic}
\usepackage{empheq,bm}
\usepackage{subcaption}
\usepackage{tikz,tikz-3dplot}
\usetikzlibrary{shapes}
\usepackage{cleveref}
\usepackage[mathlines]{lineno} 
\DeclareGraphicsExtensions{.pdf,.png,.jpg,.jpeg}

\numberwithin{equation}{section}
\newtheorem{theorem}{Theorem}
\newtheorem{lemma}[theorem]{Lemma}

\theoremstyle{definition}
\newtheorem{definition}[theorem]{Definition}
\theoremstyle{remark}
\newtheorem{remark}[theorem]{Remark}

\crefname{hypothesis}{Hypothesis}{Hypotheses}
\crefname{fact}{Fact}{Facts}

\definecolor{unitcellblue}{RGB}{90,89,169}
\definecolor{darksoil}{RGB}{255,255,255}
\definecolor{darksgreen}{RGB}{90,89,169}
\hypersetup{colorlinks=true,linkcolor=blue,citecolor=blue,urlcolor=blue,
 pdftitle={Mixed Unfitted finite element approximation of band structures in photonic crystals},
 pdfauthor={Yuyan Zou, Luke Bennetts, Hailong Guo}}

\allowdisplaybreaks

%% file: authors.tex
\author{
  Yuyan Zou, Luke G.~Bennetts, Hailong Guo\thanks{Corresponding authors: Hailong Guo (hailong.guo@unimelb.edu.au).} \\
  \normalfont School of Mathematics and Statistics, The University of Melbourne \\
  \normalfont Parkville, VIC 3010, Australia \\
  \normalfont\texttt{yuyan.zou@student.unimelb.edu.au} \\
  \normalfont\texttt{luke.bennetts@unimelb.edu.au} \\
  \normalfont\texttt{hailong.guo@unimelb.edu.au}
}

%% file: section1.tex
\section{Introduction}

A photonic crystal (PhC) is a dielectric medium with spatially periodic
permittivity, typically consisting of periodically arranged inclusions embedded
in a background matrix, which makes it possible to manipulate the propagation of
electromagnetic waves~\cite{joannopoulos2008molding}. 
The ability to forbid propagation within a gap, and to
guide or localize light through defects, underlies applications such as low-loss
waveguides around sharp bends~\cite{lin1998experimental}, optical
fibers~\cite{russell2003photonic, mescia2011particle}, near-infrared photonic
structures~\cite{krauss1996two}, and low-threshold
lasers~\cite{painter1999two}. The concept was introduced independently by
Yablonovitch~\cite{yablonovitch1987inhibited} and John~\cite{john1987strong}
in 1987.

The propagation of electromagnetic waves in such a medium is governed by a
time-harmonic Maxwell eigenvalue problem with a periodic
coefficient~\cite{joannopoulos2008molding}. By Floquet--Bloch
theory~\cite{kuchment2012floquet}, this problem decomposes into a family of
eigenvalue problems on a single unit cell, one for each \textit{quasi-momentum}, $\bm k$,
in the \textit{Brillouin zone}.  
The corresponding eigenvalues, \(\omega_n(\bm k)\), form the \textit{band structure}, which characterizes the allowed and forbidden frequencies for electromagnetic wave propagation and is fundamental to photonic crystal design~\cite{joannopoulos2008molding,ammari2018mathematical}.
Therefore, accurate computation
of the band structure is  essential, and is the problem we address in this paper.

The inclusions used in practice are often geometrically complicated, involving nonconvex
cross-sections, rotations, and connected three-dimensional frameworks~\cite{feng2005imaging, serbin2004fabrication, takahashi2009direct}. Moreover, the geometry is not fixed. 
A PhC device is designed by varying
the shape, orientation, and filling fraction of the inclusion until the band
structure meets a specification~\cite{mescia2011particle}, so the band structure
must be recomputed for a sequence of related geometries. Both features affect
the choice of discretization, and they motivate the method developed here.

Since analytical band structures are available only in very special cases, they
must in general be computed numerically. 
However, PhC computation presents three main challenges.
(i)~The dielectric coefficient is discontinuous across material interfaces, which reduces the regularity of the electromagnetic fields and complicates the numerical approximation.
(ii)~The Floquet--Bloch condition,
$\bm u(\bm x+\bm R)=e^{i\bm k\cdot\bm R}\bm u(\bm x)$, for lattice vectors $\bm R$, 
is quasi-periodic rather than periodic, and must be imposed on the unit
cell~\cite{kuchment2012floquet}. 
(iii)~The curl curl operator has a large
kernel, so the divergence-free constraint must be handled correctly~\cite{jiang1996origin, boffi1999computational}.

Several computational methods have been developed, including the plane wave expansion (PWE)
method~\cite{ho1990existence,shi2004plane,li2003photonic,norton2013planewave},
finite-difference time-domain (FDTD)
methods~\cite{de2001finite,wu2000analysis}, finite element methods
(FEM)~\cite{andonegui2013finite,degirmenci2013finite,giani2012adaptive}, the
scattering-matrix approach~\cite{crocco2007generalized}, and boundary integral
equations~\cite{yuan2008modeling}. PWE expands the fields in a truncated Fourier
series and is efficient for smooth structures, but converges
slowly~\cite{sakoda2005optical} at sharp interfaces due to the Gibbs
phenomenon~\cite{inguva1992photonic}. FDTD recovers
the spectrum from a time-domain simulation and represents curved interfaces
through a staircase approximation that degrades the
accuracy~\cite{cangellaris1991analysis}. FEM,
by contrast, handles complex geometries naturally.

FEM can be classified as body-fitted or unfitted, depending on whether the mesh
resolves the material interfaces. Body-fitted approaches include adaptive
$h$-FEM~\cite{giani2012adaptive}, spectral elements~\cite{luo2009spectral}, and
discontinuous Galerkin
methods~\cite{lu2017discontinuous,busch2011discontinuous,chen2020multiphysics}.
They benefit from solid theoretical foundations but suffer from a substantial
mesh-generation cost, as the mesh must conform to every interface and must be
regenerated whenever the geometry is varied. Unfitted methods reduce this cost by
allowing interfaces to cut mesh elements
arbitrarily~\cite{hansbo2002unfitted,wang2021finite}, so that the same background mesh can be used for different geometries. 
Moreover, a uniform background in combination with the Floquet–Bloch transform allows the quasi-periodic boundary conditions to be applied straightforwardly.

Within the unfitted framework, several finite element methods have been developed in recent years for Maxwell interface problems, including adaptive immersed FEMs~\cite{chen2009adaptiveimmersed}, immersed virtual element methods~\cite{cao2021virtual}, high-order unfitted FEMs for time-harmonic Maxwell interface problems~\cite{chen2024highorder}, and interface-penalty and mixed unfitted formulations~\cite{liu2020interface,wang2025mixed}. These studies, however, focus primarily on source problems rather than Maxwell eigenvalue problems. Passing from a source problem to the sort of eigenvalue problems that arise for PhCs requires the discrete kernel to be controlled uniformly and the discrete solution operators to converge in a form that prevents spectral pollution~\cite{monk2003finite,boffi2010finite}, in addition to the standard stability for a prescribed right-hand side. 
For example, the mixed unfitted method
proposed by~\cite{wang2025mixed}  was developed and analyzed for a Maxwell
interface source problem using second-family N\'ed\'elec elements of degree one
together with $P_1$ Lagrange elements. 
This pair does not realize the compatible
sequence used in the present spectral analysis, for which the corresponding
scalar space is piecewise $P_2$. Consequently, although the method
in~\cite{wang2025mixed} is convergent for the source problem, it still suffers from spurious eigenvalues when applied to the Maxwell eigenvalue problem.

For the computation of PhCs governed by Maxwell’s equations with complex inclusion geometries, the only unfitted approach currently available in the literature appears to be the Petrov–Galerkin finite element interface method proposed in~\cite{wang2021finite}. This method discretizes the two-dimensional Maxwell interface eigenvalue problem using Lagrange-type elements and imposes the quasi-periodic boundary condition directly, without employing the Floquet–Bloch transform. More importantly, the work focuses primarily on the numerical formulation and experiments and does not provide a rigorous theoretical analysis.
For crystals governed by $H^1$-elliptic equations, Guo~\cite{guo2021topmaterials,guo2021unfitted} combined an unfitted FEM with the Floquet–Bloch transform and the Babu\v{s}ka–Osborn theory to establish spectral convergence for the resulting eigenvalue problem. Such an analysis, however, does not extend directly to the Maxwell eigenproblem. The latter is posed in $H(\operatorname{curl})$, involves an operator with an infinite-dimensional kernel, and lacks the ellipticity available in the $H^1$ setting. More importantly, a convergent and spurious-free discretization requires compatible finite element spaces that preserve the underlying de~Rham structure through an appropriate commuting relation. To the best of our knowledge, a rigorous spectral convergence theory for unfitted discretizations of the quasi-periodic Maxwell interface eigenvalue problem is still unavailable. In this work, we fill this gap by developing such a theory in both two and three dimensions.

In this paper, we combine the Floquet–Bloch transform~\cite{kuchment2012floquet} with an unfitted finite element framework to develop a mixed discretization of the quasi-periodic Maxwell interface eigenvalue problem. The Floquet–Bloch transform reduces the problem to a periodic unit-cell problem, with the quasi-momentum incorporated into the shifted differential operators through $\nabla_{\bm k}=\nabla+i\bm k$ and, in the presence of a material interface, into the corresponding interface conditions. We discretize the resulting problem on a fixed uniform  mesh that is independent of the inclusion geometry and can therefore be reused as the geometry changes. On this mesh, we impose periodicity by directly identifying the corresponding degrees of freedom, enforce the interface conditions weakly using Nitsche’s method~\cite{nitsche1971variationsprinzip,hansbo2002unfitted}, and employ a face-based ghost penalty~\cite{burman2010ghost,burman2012fictitious,burman2015cutfem} to stabilize arbitrarily small cut elements.

Further, we construct modified unfitted N\'ed\'elec and $P_2$ Lagrange spaces that preserve the de~Rham commuting structure for the shifted operator $\nabla_{\bm k}$. In particular, our construction ensures that $\nabla_{\bm k}q_h$ belongs to the discrete modified N\'ed\'elec space. Exploiting this structure, we establish an inf–sup condition with a constant uniform with respect to the position of the interface relative to the background mesh. This uniform stability allows us to control the discrete kernel and, crucially, to rule out spurious modes.

For each fixed quasi-momentum, we establish the well-posedness and uniform stability of the discrete problem and derive error estimates for the associated source problem. Building on these results and the Babu\v{s}ka–Osborn spectral approximation theory~\cite{babuska1991eigenvalue,osborn1975spectral}, we prove second-order convergence of the eigenvalues and eigenspaces, with constants uniform with respect to both the mesh size and the position of the interface relative to the background mesh. We present numerical experiments in two and three dimensions that confirm the predicted convergence rates and demonstrate the robustness of the proposed method.

The rest of the paper is organized as follows.
Section~\ref{sec:model} formulates the photonic-crystal eigenvalue problem and
its Floquet--Bloch reduction. Section~\ref{chap:chapter3} introduces the mixed
unfitted finite element discretization, and Section~\ref{sec:stability}
establishes its well-posedness. Section~\ref{sec:error-analysis} presents the
error and spectral convergence analysis. Section~\ref{sec:numerical-results}
reports the numerical experiments, and Section~\ref{sec:conclusion} concludes
the paper.

%% file: section2.tex
\section{Mathematical Model}\label{sec:model}
A photonic crystal in $\mathbb{R}^d$ ($d=2,3$) is modeled using a spatially periodic relative permittivity $\varepsilon(\bm{x})>0$. 
We assume that the crystal is linear and non-dispersive and contains no free charges or currents. The propagation of electromagnetic waves is then governed by
\begin{subequations}\label{master}
\begin{align}
\nabla \times \left( \varepsilon^{-1} \nabla \times \mathbf{H} \right)
&= \omega^2 \mathbf{H}
\quad &&\text{in } \mathbb{R}^{d}, \label{mastera} \\[0.5em]
\nabla \cdot \mathbf{H}
&= 0
\quad &&\text{in } \mathbb{R}^{d}. \label{masterb}
\end{align}
\end{subequations}
where $\mathbf{H}(\bm{x})$ is the magnetic field and $\omega \geq 0$ denotes the angular frequency.

We focus on photonic crystals composed of embedded inclusions within a matrix, with
\begin{equation}\label{eq:epsilon}
  \varepsilon(\bm{x}) = 
  \begin{cases} 
    \varepsilon_+ & \text{in the matrix region,} \\[0.3em] 
    \varepsilon_- & \text{in the inclusions,}
  \end{cases}
\end{equation}
where $\varepsilon_+$ and $\varepsilon_-$ are positive constants.
The inclusions are periodically arranged according to the \emph{Bravais lattice}~\cite{kittel2018introduction}
\begin{equation}\label{eq:bravais}
A := \left\{ \bm{R} = \sum_{i=1}^{d} n_i \bm{a}_i \;\middle|\; n_i \in \mathbb{Z} \right\},
\end{equation}
spanned by linearly independent \emph{primitive lattice vectors} $\bm{a}_1, \dots, \bm{a}_d \in \mathbb{R}^d$ , so that $\varepsilon(\bm{x} + \bm{R}) = \varepsilon(\bm{x})$ for all $\bm{R}\in A$ (see Fig.~\ref{fig:lattice2d3d}). 
The associated \emph{primitive unit cell} is
\begin{equation}\label{eq:unitcell}
  \Omega := \left\{ \sum_{i=1}^{d} s_i \bm{a}_i \;\middle|\; 0 \leq s_i < 1 \right\},
\end{equation}
and decomposes as $\Omega = \Omega^+ \cup \Sigma \cup \Omega^-$, with inclusion $\Omega^-$, matrix $\Omega^+ := \Omega \setminus \overline{\Omega^-}$, and interface $\Sigma := \partial \Omega^-$. We assume $\Sigma$ is continuous and denote by $\bm{n}$ its unit normal oriented from $\Omega^-$ to $\Omega^+$.

\begin{figure}[htbp]
\centering
\begin{subfigure}[b]{0.45\textwidth}
    \centering
    \includegraphics[width=\textwidth]{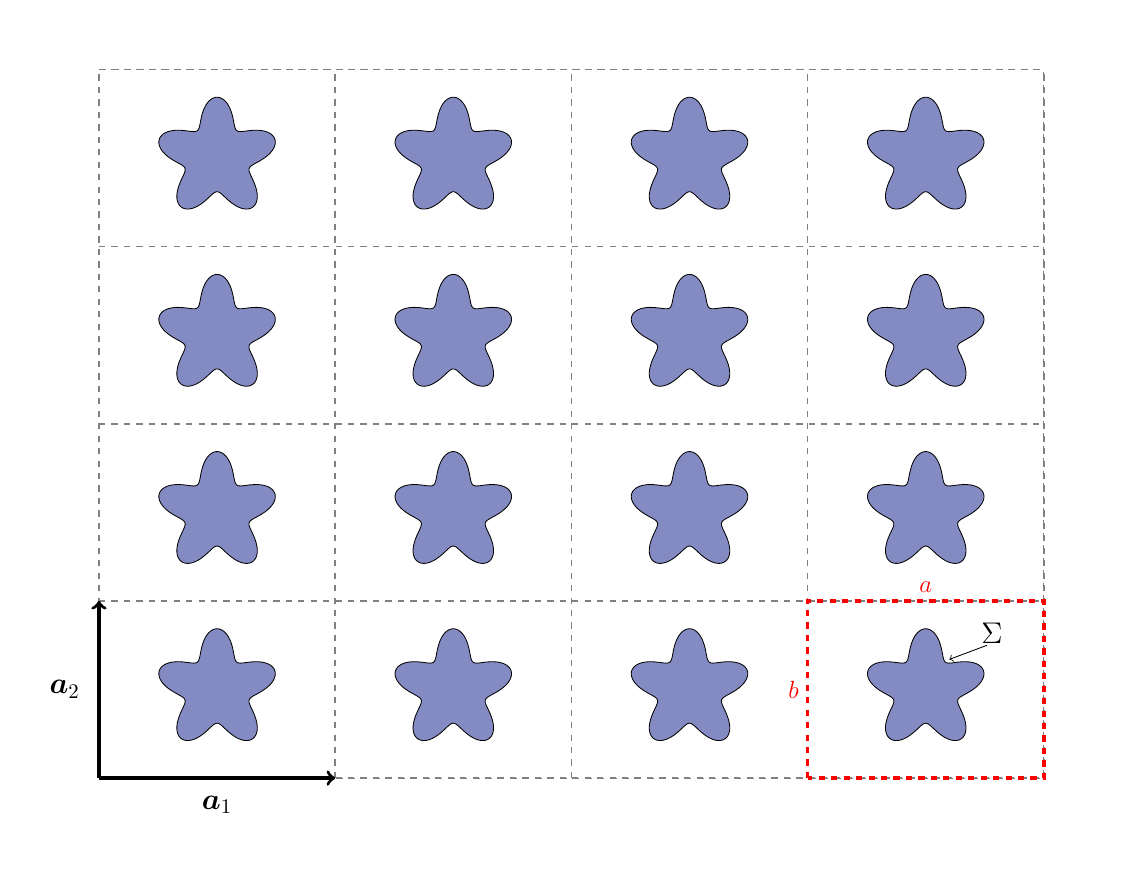}
    \caption{Two-dimensional periodic structure.}
    \label{fig:lattice2D}
\end{subfigure}
\begin{subfigure}[b]{0.46\textwidth}
    \centering
    \includegraphics[width=\textwidth]{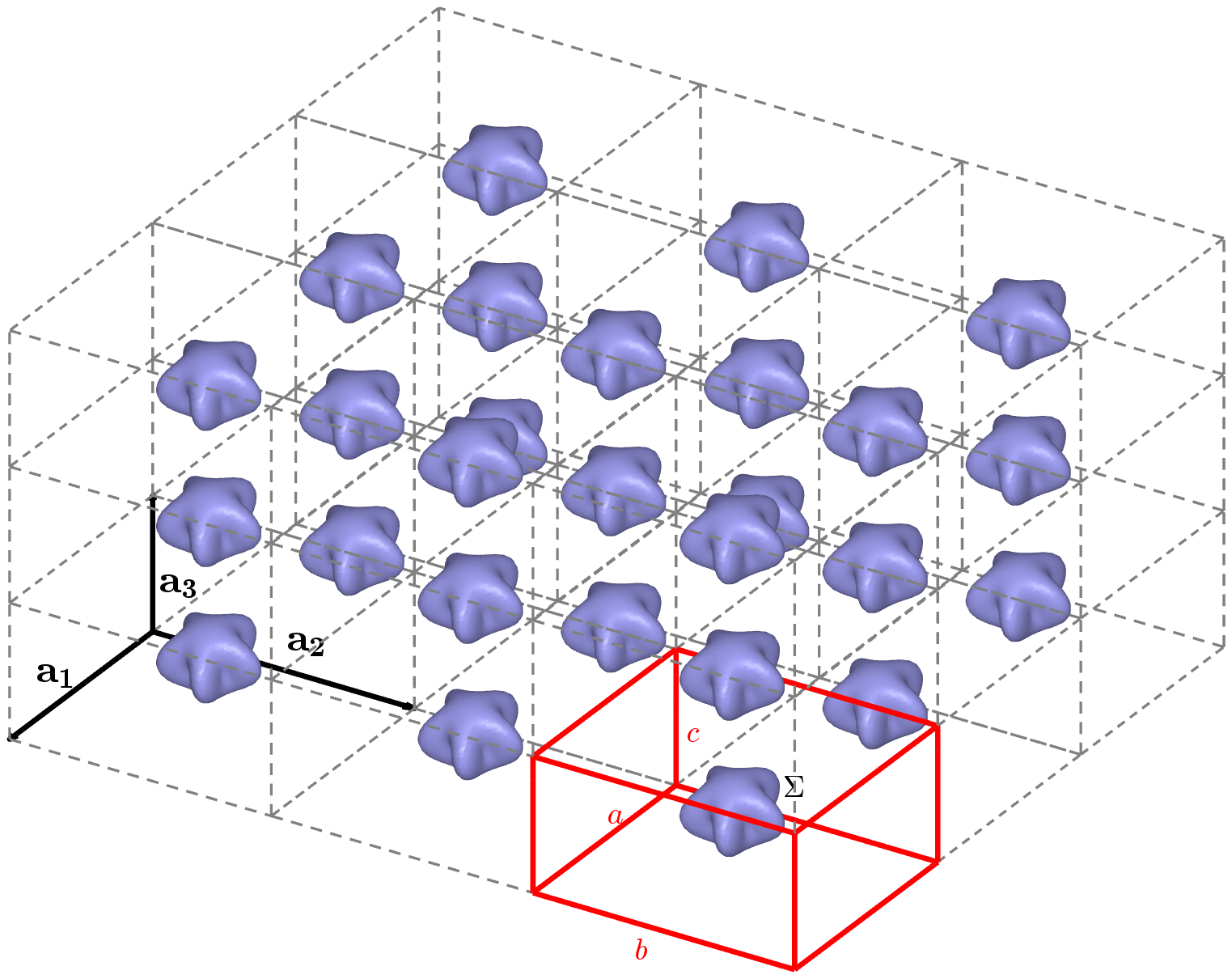}
    \caption{Three-dimensional periodic structure.}
    \label{fig:lattice3D}
\end{subfigure}
\caption{Schematic illustrations of two-dimensional and three-dimensional 
photonic crystal structures with the associated primitive lattice vectors. In Fig.~\ref{fig:lattice2d3d}(\subref{fig:lattice2D}) the lattice vectors are chosen as $\bm{a}_1=(a,0)$ and $\bm{a}_2=(0,b)$. In Fig.~\ref{fig:lattice2d3d}(\subref{fig:lattice3D}) the lattice vectors are $\bm{a}_1=(a,0,0)$, $\bm{a}_2=(0,b,0)$ and $\bm{a}_3=(0,0,c).$}
\label{fig:lattice2d3d}
\end{figure}
Since $\varepsilon$ is discontinuous across $\Sigma$, the equations are supplemented by the electromagnetic jump conditions~\cite[Section~1.5]{jackson1999classical} requiring continuity of the tangential components of $\mathbf{H}$ and $\varepsilon^{-1}\nabla\times\mathbf{H}$. To express these conditions uniformly for $d=2,3$, we introduce the tangential trace operators on $\Sigma$
\begin{equation}\label{eq:tanops}
\gamma_t(\bm{v}):=
\begin{cases}
\bm{v}\cdot \bm{t}, & d=2,\\[0.3em]
\bm{v}\times \bm{n}, & d=3,
\end{cases}
\qquad
\gamma_{\mathcal{T}}(\bm{w}):=
\begin{cases}
\bm{w}, & d=2,\\[0.3em]
\bm{n}\times \bm{w}, & d=3,
\end{cases}
\end{equation}
where, for $d=2$, $\bm{t}:=(-n_2,n_1)^\top$ denotes the unit tangent vector, and $[\![\bm{v}]\!]$ denotes the jump across $\Sigma$
\begin{equation}\label{eq:jump}
[\![ \bm{v} ]\!](\bm{x})
:= \lim_{\delta \to 0^{+}} \bigl[\bm{v}(\bm{x}-\delta \bm{n}) - \bm{v}(\bm{x}+\delta \bm{n})\bigr],
\qquad \bm{x}\in \Sigma.
\end{equation}
The same definition applies to scalar functions.

By the Bloch--Floquet theorem~\cite{joannopoulos1995photonic, floquet1883equations, bloch1929quantenmechanik}, any solution 
of \eqref{master} takes the form
\begin{equation}\label{eq:bloch}
  \mathbf{H}(\bm{x}) = e^{i\bm{k}\cdot\bm{x}}\,\mathbf{u}_{\bm{k}}(\bm{x}),
\end{equation}
where $\mathbf{u}_{\bm{k}}$ is periodic on $A$, that is, $\mathbf{u}_{\bm{k}}(\bm{x}+\bm{R}) = \mathbf{u}_{\bm{k}}(\bm{x})$ for all $\bm{R}\in A$, and the quasi-momentum $\bm{k}$ ranges over the \emph{first Brillouin zone}~$\mathcal{B}$~\cite{kittel2018introduction}.
The \emph{first Brillouin 
zone} $\mathcal{B}$ is defined as the Wigner--Seitz cell of the reciprocal 
lattice
\begin{equation}\label{eq:Brillouinzone}
  A^{*} := \left\{ \sum_{i=1}^{d} m_i \bm{b}_i \;\middle|\; 
  m_i \in \mathbb{Z} \right\},
\end{equation}
whose basis vectors satisfy $\bm{b}_i \cdot \bm{a}_j = 2\pi\delta_{ij}$ for $i,j = 1,\dots,d$.

Using the Bloch representation \eqref{eq:bloch}, problem \eqref{master} on $\mathbb{R}^d$ can be reduced to a family of eigenvalue problems on $\Omega$. The problem then reads: for each $\bm{k}\in\mathcal{B}$, find eigenvalues $\omega^2 \geq 0$ and corresponding eigenfunctions $\mathbf{H}\not\equiv\bm{0}$ such that
\begin{subequations}\label{eq:Model1}
\begin{align}
\nabla \times \left( \varepsilon^{-1} \nabla \times \mathbf{H} \right) &= \omega^2 \mathbf{H}
\quad & \bm{x} \in \Omega \setminus \Sigma, \label{Model1a}\\[0.5em]
\nabla \cdot \mathbf{H} &= 0
\quad & \bm{x} \in \Omega \setminus \Sigma, \label{Model1b}\\[0.5em]
[\![\gamma_t(\mathbf{H})]\!] &= 0 
\quad &\bm{x} \in \Sigma, \label{Model1c}\\[0.5em]
[\![\gamma_\mathcal{T}(\varepsilon^{-1}\nabla \times \mathbf{H})]\!] &= 0
\quad &\bm{x} \in \Sigma, \label{Model1d}\\[0.5em]
\mathbf{H}(\bm{x} + \bm{R}) &= e^{i\bm{k}\cdot\bm{R}}\,\mathbf{H}(\bm{x}) 
\quad &\bm{R}\in A, \quad \bm{x} \in \partial\Omega. \label{Model1e}
\end{align}
\end{subequations}
Conditions \eqref{Model1c}--\eqref{Model1d} give the jump conditions that any eigenfunction must satisfy across the material interface $\Sigma$, and \eqref{Model1e} imposes quasi-periodic boundary conditions on $\mathbf{H}$ across $\partial\Omega$.
The full spectrum of \eqref{master} is recovered by solving \eqref{eq:Model1} over all $\bm{k}\in\mathcal{B}$.
For each $\bm{k}\in\mathcal{B}$, \eqref{eq:Model1} admits a discrete sequence of eigenvalues
\begin{equation}\label{eq:spectrum}
  0 \leq \omega_1(\bm{k})^2 \leq \omega_2(\bm{k})^2 \leq \cdots \to +\infty,
\end{equation}
counted with multiplicity. The maps $\bm{k}\mapsto \omega_n(\bm{k})$, $n=1,2,\dots$, are called the \emph{dispersion relations}, and the set of all admissible frequencies
\begin{equation}\label{eq:bandstructure}
  \sigma = \bigcup_{\bm{k}\in\mathcal{B}} \bigcup_{n=1}^{\infty} \{\omega_n(\bm{k})\}
\end{equation}
forms the \emph{band structure} of the photonic crystal. A \emph{band gap} is a connected interval $(\omega_a,\omega_b)\subset\mathbb{R}^+$ such that $(\omega_a,\omega_b)\cap\sigma = \emptyset$, i.e.\ a range of frequencies at which no electromagnetic wave can propagate through the crystal.


\subsection{Bloch--Floquet Theory}
\label{sec:cell} 

Introducing the shifted operator
\begin{equation}\label{eq:shiftedcurl}
  \nabla_{\bm{k}} := \nabla + i\bm{k},
\end{equation}
and substituting \eqref{eq:bloch} into \eqref{eq:Model1}, we reformulate the quasi-periodic problem equivalently in terms of the lattice-periodic function $\bm{u}_{\bm{k}}$. The problem restricted to a unit cell reads: for fixed $\bm{k}\in\mathcal{B}$, find $\omega^2\in\mathbb{R}^{+}$ and $\bm{u}_{\bm{k}}:\Omega\to\mathbb{C}$, 
$\bm{u}_{\bm{k}} \neq \bm{0}$, such that
\begin{subequations}\label{eq:Model2}
\begin{align}
\nabla_{\bm{k}} \times \left( \varepsilon^{-1} \nabla_{\bm{k}} \times \bm{u}_{\bm{k}} \right) &= \omega^2 \bm{u}_{\bm{k}}
\quad & \bm{x} \in \Omega \setminus \Sigma, \label{Model2a}\\[0.5em]
\nabla_{\bm{k}} \cdot \bm{u}_{\bm{k}} &= 0
\quad & \bm{x} \in \Omega \setminus \Sigma, \label{Model2b}\\[0.5em]
[\![\gamma_t(\bm{u}_{\bm{k}})]\!] &= 0
\quad &\bm{x} \in \Sigma, \label{Model2c}\\[0.5em]
[\![\gamma_\mathcal{T}(\varepsilon^{-1}\nabla_{\bm{k}} \times \bm{u}_{\bm{k}})]\!] &= 0
\quad &\bm{x} \in \Sigma, \label{Model2d}\\[0.5em]
\bm{u}_{\bm{k}}(\bm{x} + \bm{R}) &= \bm{u}_{\bm{k}}(\bm{x})
\quad &\bm{R}\in A, \quad \bm{x} \in \partial\Omega, \label{Model2e}
\end{align}
\end{subequations}
where $\varepsilon$ is the permittivity defined in 
\eqref{eq:epsilon}, and $\nabla_{\bm{k}}$ is the shifted 
operator introduced in \eqref{eq:shiftedcurl}. In contrast to existing formulations, the quasi-momentum $\bm{k}$ appears not only in the governing equations but also in the interface condition \eqref{Model2d} through the shifted curl operator $\nabla_{\bm{k}} \times$.
The unit cells used in the computations are illustrated in Fig.~\ref{fig:unitcell}.
\begin{figure}[htbp]
    \centering
    \begin{subfigure}[b]{0.42\textwidth}
        \centering
        \includegraphics[width=\textwidth]{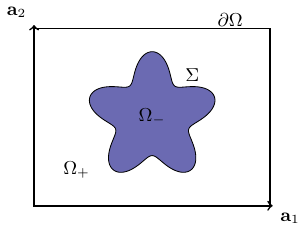}
        \caption{Two-dimensional unit cell.}
        \label{fig2dcell}
    \end{subfigure}
    \begin{subfigure}[b]{0.45\textwidth}
        \centering
        \includegraphics[width=\textwidth]{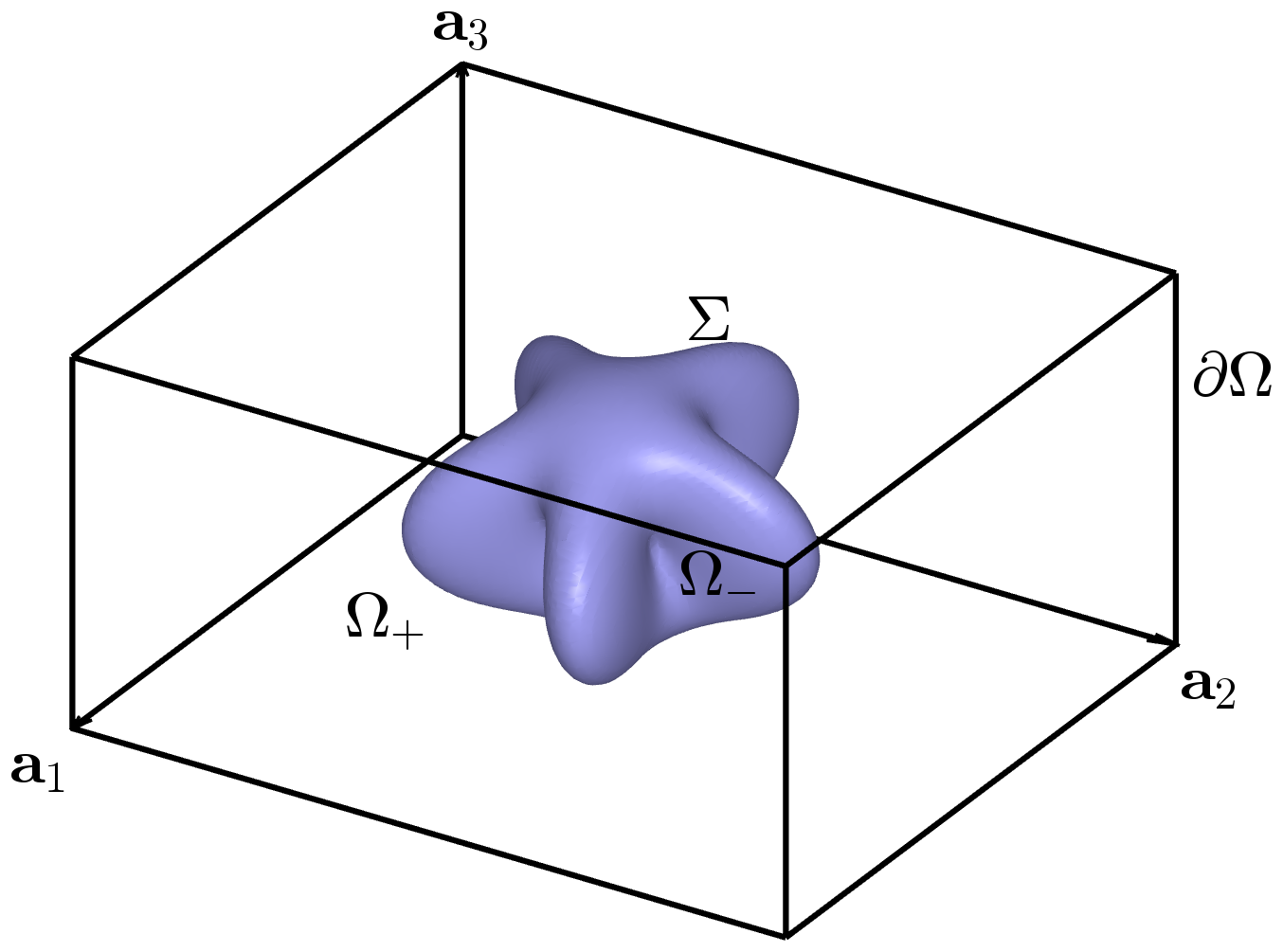}
        \caption{Three-dimensional unit cell.}
        \label{fig3dcell}
    \end{subfigure}
    \caption{Unit cell in two and three dimensions: matrix $\Omega^+$ and inclusion $\Omega^-$, separated by the interface $\Sigma$, with outer boundary $\partial\Omega$.}
    \label{fig:unitcell}
\end{figure}

\subsection{Functional Spaces}
In this subsection we introduce the function spaces in which
problem~\eqref{eq:Model2} is analyzed. For any subdomain $D\subseteq\Omega$, we use the standard 
notation $\mathbf{L}^2(D):=(L^2(D))^d$ with inner product 
$(\bm{u},\bm{v})_D=\int_D\bm{u}\cdot\bar{\bm{v}}\,\mathrm{d}\bm{x}$, 
and recall the Sobolev spaces \cite{evans2022partial}
\[
H^1(D) := \{ v \in L^2(D) : \nabla v \in \mathbf{L}^2(D) \},
\]
\[
\mathbf{H}(curl;D) := \{ \boldsymbol{v} \in \mathbf{L}^2(D) : \nabla \times \boldsymbol{v} \in \mathbf{L}^2(D) \},
\]
\[
\mathbf{H}(div;D) := \{ \boldsymbol{v} \in \mathbf{L}^2(D) :\nabla \cdot \boldsymbol{v} \in L^2(D) \}.
\]
The spaces
$\mathbf{H}(curl;\Omega)$ and $\mathbf{H}(div;\Omega)$ are equipped with inner products:
\[
(\boldsymbol{u},\boldsymbol{v})_{\mathbf{H}(\mathrm{curl};\Omega)}
  =(\boldsymbol{u},\boldsymbol{v})_{\Omega}
   +(\nabla\times\boldsymbol{u},\nabla\times\boldsymbol{v})_{\Omega},
\quad
(\boldsymbol{u},\boldsymbol{v})_{\mathbf{H}(\mathrm{div};\Omega)}
  =(\boldsymbol{u},\boldsymbol{v})_{\Omega}
   +(\nabla\cdot\boldsymbol{u},\nabla\cdot\boldsymbol{v})_{\Omega}.
\]
To account for the periodicity of the Bloch eigenfunctions, let 
$C^\infty_{\mathrm{per}}(\bar{\Omega})$ denote the space of smooth 
functions satisfying $q(\bm{x}+\bm{a})=q(\bm{x})$ for all $\bm{a}\in A$ and 
$\bm{x}\in\partial\bar{\Omega}$, where $A$ denotes the Bravais lattice. The periodic 
Lebesgue space $L^2_{\mathrm{per}}(\Omega)$ is defined as the closure of $C^\infty_{\mathrm{per}}(\bar{\Omega})$ with respect to $\|\cdot\|_{0,D}$ (see, e.g., \cite{lu2017discontinuous} section 3) in $L^2(\Omega)$, and we set $\bm{L}^2_{\mathrm{per}}(\Omega):=[L^2_{\mathrm{per}}(\Omega)]^d$.

\subsubsection{Bloch Map and Modified Spaces}
The electromagnetic field takes the form $\mathbf{H}(\bm{x}) = e^{i\bm{k}\cdot\bm{x}}\bm{u}(\bm{x})$ with $\bm{u}$ periodic. We introduce the Bloch map to pass between the two representations, and we define all working spaces in terms of $\bm{u}$.

\begin{definition}[Bloch map]
\label{def:bloch_map}
For $\bm{k}\in\mathcal{B}$, define the Bloch map
\begin{equation}\label{eq:bloch_map}
  \Phi_{\bm{k}} : \bm{u} \longmapsto \mathbf{H} := e^{i\bm{k}\cdot\bm{x}}\bm{u},
  \qquad
  \Phi_{\bm{k}}^{-1} : \mathbf{H} \longmapsto \bm{u} := e^{-i\bm{k}\cdot\bm{x}}\mathbf{H}.
\end{equation}
The map $\Phi_{\bm{k}}$ intertwines the shifted operator
$\nabla_{\bm{k}} = \nabla + i\bm{k}$ with the standard operator $\nabla$:
\begin{align}
  \nabla_{\bm{k}}\bm{u}
    &= e^{-i\bm{k}\cdot\bm{x}}\nabla\!\left(e^{i\bm{k}\cdot\bm{x}}\bm{u}\right)
     = \Phi_{\bm{k}}^{-1}\nabla\Phi_{\bm{k}}\bm{u},
  \label{eq:bloch_grad}\\
  \nabla_{\bm{k}}\times\bm{u}
    &= \Phi_{\bm{k}}^{-1}(\nabla\times\Phi_{\bm{k}}\bm{u}).
  \label{eq:bloch_curl}
\end{align}
\end{definition}
Since the solutions of~\eqref{eq:Model1} possess piecewise regularity across $\Sigma$, we work with the following piecewise periodic spaces:
\begin{align}\label{eq:V_space}
\tilde{\bm{V}} &:= 
\mathbf{H}_{\mathrm{per}}(\mathrm{curl};\Omega^+\cup\Omega^-)
  := \bigl\{\bm{v}\in\bm{L}^2_{\mathrm{per}}(\Omega):
     \bm{v}|_{\Omega^\pm}\in\mathbf{H}(\mathrm{curl};\Omega^\pm)\bigr\},\\
\label{eq:P_space}
\tilde{P} &:= H^1_{\mathrm{per}}
(\Omega^+\cup\Omega^-)
  := \bigl\{q\in L^2_{\mathrm{per}}(\Omega):
     q|_{\Omega^\pm}\in H^1(\Omega^\pm)\bigr\}.
\end{align}

For a given quasi-momentum $\bm{k}\in\mathcal{B}$, the corresponding \emph{modified piecewise spaces} are defined via $\Phi_{\bm{k}}$ as
\begin{align}
  \bm{V}
    &:= \Phi_{\bm{k}}^{-1}\tilde{\bm{V}}
     = \bigl\{\tilde{\bm{v}}\in\bm{L}^2_{\mathrm{per}}(\Omega):
       e^{i\bm{k}\cdot\bm{x}}\tilde{\bm{v}}\in\tilde{\bm{V}}\bigr\},
  \label{eq:Vk}\\
  P
    &:= \Phi_{\bm{k}}^{-1}\tilde{P}
     = \bigl\{\tilde{q}\in L^2_{\mathrm{per}}(\Omega):
       e^{i\bm{k}\cdot\bm{x}}\tilde{q}\in \tilde{P}\bigr\},
  \label{eq:Pk}
\end{align}
equipped with the norms induced by $\Phi_{\bm{k}}$:
\begin{align}\label{eq:Vk_norm}
  \|\tilde{\bm{v}}\|_{\bm{V}}^2
    := \sum_{m\in\{+,-\}}\!\left(\|\tilde{\bm{v}}\|_{0,\Omega^m}^2
         +\|\nabla_{\bm{k}}\times\tilde{\bm{v}}\|_{0,\Omega^m}^2\right),\\
  \|\tilde{q}\|_{P}^2
    := \sum_{m\in\{+,-\}}(\|\tilde{q}\|_{0,\Omega^m}^2+\|\nabla_{\bm{k}}\tilde{q}\|_{0,\Omega^m}^2).
\end{align}
\begin{remark}
The spaces $\tilde{\bm{V}}$ and $\tilde{P}$ appear only in the definition of $\bm{V}$ and $P$ above.
All subsequent analysis is carried out entirely in $\bm{V}$ and $P$, with the shifted operators $\nabla_{\bm{k}}$ and $\nabla_{\bm{k}}\times$.
\end{remark}
We further define the shifted-curl kernel and the shifted divergence-free
space, respectively:
\begin{align}\label{eq:V0}
\boldsymbol{V}_0
  &:= \bigl\{\boldsymbol{v}\in\boldsymbol{V} :
     \nabla_{\bm{k}}\times\boldsymbol{v}|_{\Omega^\pm} = \bm{0}\bigr\},\\
\label{eq:divfree}
\bm{H}_{\mathrm{per}}(\mathrm{div}^0_{\bm{k}};\Omega^+\cup\Omega^-)
  &:= \bigl\{\bm{v}\in
     \mathbf{H}_{\mathrm{per}}(\mathrm{div};\Omega^+\cup\Omega^-):
     \nabla_{\bm{k}}\cdot\bm{v}|_{\Omega^\pm}=0\bigr\}.
\end{align}

\subsection{Mixed Formulation}
To impose the divergence constraint, we introduce a scalar Lagrange multiplier
$p$. For each $\bm{k}\in\mathcal{B}$, the resulting mixed formulation of
problem~\eqref{eq:Model2} is
\begin{subequations}\label{eqmain}
\begin{align}
\nabla_{\bm{k}} \times \left(\varepsilon^{-1}\nabla_{\bm{k}} \times \bm{u} \right)-\nabla_{\bm{k}}p
&= \omega^2 \bm{u}
\quad & \bm{x} \in \Omega \backslash \Sigma, \label{eqmaina} \\[0.5em]
\nabla_{\bm{k}} \cdot \bm{u}
&= 0
\quad & \bm{x} \in \Omega \backslash \Sigma, \label{eqmainb} \\[0.5em]
[\![p]\!]=0, \quad [\![\gamma_t(\bm{u})]\!]
&= 0
\quad & \bm{x} \in \Sigma, \label{eqmainc} \\[0.5em]
[\![\gamma_\mathcal{T}(\varepsilon^{-1}\nabla_{\bm{k}}\times \bm{u})]\!]
&= 0
\quad & \bm{x} \in \Sigma, \label{eqmaind} \\[0.5em]
\bm{u}(\bm{x} + \bm{R}) 
&= \bm{u}(\bm{x}) 
\quad &\bm{R}\in A, \quad \bm{x} \in \partial\Omega, \label{eqmaine}
\end{align}
\end{subequations}
where $\varepsilon$ is the positive piecewise-constant permittivity of \eqref{eq:epsilon} and
$[\![\cdot]\!]$ denotes the jump defined in \eqref{eq:jump}.

The weak formulation of problem~\eqref{eqmain} is given by: for each $\bm{k} \in \mathcal{B}$,
find $(\omega^2, \bm{u}, p) \in \mathbb{R}^{+} \times \bm{V} \times P$ with $\bm{u} \neq \bm{0}$ such that:
\begin{equation}
\left\{
\begin{aligned}
a(\bm{u}, \bm{v}) + b(p, \bm{v}) &= \omega^2 (\bm{u}, \bm{v})
&& \forall \bm{v} \in \bm{V}, \\
\overline{b(q, \bm{u})}&= 0
&& \forall q \in P,
\end{aligned}
\right.
\end{equation}
with the sesquilinear forms defined by
\begin{align}
a(\bm u,\bm v)
=&\sum_{m=\pm} \int_{\Omega^{m}}(\varepsilon_{m}^{-1} \nabla_{\bm{k}} \times \bm{u}^m) \cdot \overline{(\nabla_{\bm{k}} \times \bm{v}^m)} \, dx ,
\\
(\bm{u}, \bm{v}) =& \sum_{m=\pm}\int_{\Omega^{m}} \bm{u}^m \cdot \overline{\bm{v}^m} \, dx, \\
b(p, \bm{v}) =& -\sum_{m=\pm}\int_{\Omega^{m}} p^m \, \overline{(\nabla_{\bm{k}} \cdot \bm{v}^m)} \, dx.
\end{align}

%% file: section3.tex
\section{Mixed Unfitted Finite Element Method}\label{chap:chapter3}
We formulate a mixed unfitted finite element method (MUFEM) in which the mesh
is constructed independently of the interface $\Sigma$. The method employs
modified finite element spaces and commuting interpolation operators to
preserve the relevant discrete de Rham structure, as expressed by the
commuting diagram \eqref{eq:bloch_comm_diag}.

\subsection{Mesh Discretization}
To impose the periodic boundary conditions, let $\mathcal{T}_h$ be a uniform, shape-regular partition of the computational domain $\Omega \subset \mathbb{R}^d(d=2,3)$, generated without regard to the internal interface $\Sigma$. We denote an arbitrary element of the partition by $K \in \mathcal{T}_h$ and the global mesh size by
$h := \max_{K \in \mathcal{T}_h} \mathrm{diam}(K).$
The set of interface elements is defined by
\begin{equation}
\mathcal{T}_{\Sigma}:=
\{K\in\mathcal T_h:\; K\cap \Sigma\neq \emptyset\}.
\end{equation}
For each $K\in\mathcal{T}_\Sigma$ we write $\Sigma_K := K\cap\Sigma$
and define the piecewise interface
$\Sigma_h := \bigcup_{K\in\mathcal{T}_\Sigma}\Sigma_K$.
The sets of interior faces and ghost-penalty faces are defined by
\begin{align}
  \mathcal{F}_h^{\mathrm{int}}
    &:= \{F = K^+\cap K^- : K^+,K^-\in\mathcal{T}_h,\; K^+\neq K^-\},\\
\label{eq:Fh_Sigma}
  \mathcal{F}_h^g
    &:= \{F\in\mathcal{F}_h^{\mathrm{int}} :
        K^+\in\mathcal{T}_\Sigma \text{ or } K^-\in\mathcal{T}_\Sigma\}.
\end{align}
Decompose $\Omega$ into two overlapping extended subdomains:
\begin{equation}
\Omega_h^{\pm} 
:= \bigcup \{ K \in \mathcal{T}_h : K \cap \Omega^{\pm} \neq \emptyset \}.
\end{equation}
Since $\Omega_h^+ \cap \Omega_h^- = \bigcup_{K\in\mathcal{T}_\Sigma}K \neq \emptyset$, 
these are referred to as \emph{fictitious domains}. 
We further distinguish the non-interface elements within each subdomain:
\begin{equation}\label{eq:sub_omega}
    \omega_h^{\pm} 
    := \bigcup\{ K \in \mathcal{T}_h : K \subset \Omega^{\pm} \}.
\end{equation}

For each cut element $K$, set $K^\pm := K\cap\Omega^\pm$. We use the cut-dependent weighted average across $\Sigma$ from~\cite{sticko2020high}:
\begin{equation}\label{eq:averange}
\{\bm{v}\}_\theta := \theta_K^+ \bm{v}^+ + \theta_K^- \bm{v}^-,
\qquad
\theta_K^+ := \frac{\alpha^- |K^+|}
{\alpha^- |K^+|+\alpha^+ |K^-|},
\quad
\theta_K^- := \frac{\alpha^+ |K^-|}
{\alpha^- |K^+|+\alpha^+ |K^-|},
\end{equation}
where $\alpha^\pm=1/\varepsilon_\pm$, $|K^\pm|$ denotes the measure of $K^\pm$, and $\theta_K^\pm>0$ with $\theta_K^+ + \theta_K^- = 1$.
Fig.~\ref{fig:domain_decomp} illustrates the two-dimensional decomposition and the fictitious domain.

\begin{figure}[htbp]
    \centering
    \begin{subfigure}[b]{0.32\textwidth}
        \centering
        \includegraphics[width=\textwidth]{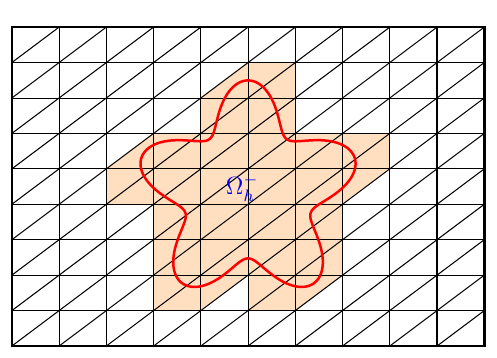}
    \end{subfigure}
    \hfill
    \begin{subfigure}[b]{0.32\textwidth}
        \centering
        \includegraphics[width=\textwidth]{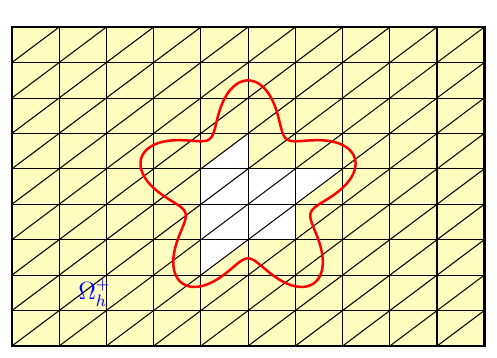}
    \end{subfigure}
    \hfill
    \begin{subfigure}[b]{0.32\textwidth}
        \centering
        \includegraphics[width=\textwidth]{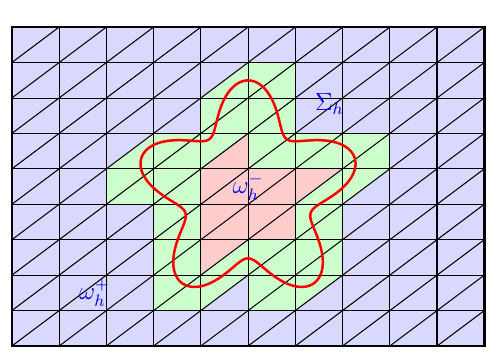}
    \end{subfigure}

    \caption{Visualization of the fictitious domain decomposition.
    (a) Extended subdomain $\Omega_h^-$ (orange);
    (b) Extended subdomain $\Omega_h^+$ (yellow);
    (c) Restricted subdomain $\omega_h^-$ (red), restricted subdomain $\omega_h^+$ (purple), and interface elements $\mathcal{T}_\Sigma$ (green).}
    \label{fig:domain_decomp}
\end{figure}

We employ the second family of N\'{e}d\'{e}lec edge elements~\cite{chen2016finite, nedelec1980mixed,nedelec1986new}. On an element $K \in \mathcal{T}_h$ with barycentric coordinates $\{\lambda_i\}$, the basis functions and degrees of freedom for the $k$-th edge $\bm{e}_k$ connecting vertices $(i,j)$ with unit tangent $\bm{t}$ are
\begin{alignat}{2}
    \boldsymbol{\phi}_k &= \lambda_i \nabla\lambda_j - \lambda_j \nabla\lambda_i,
    &\quad
    l_k(\bm{v}) &= \int_{\bm{e}_k} \bm{v} \cdot \bm{t}\, ds
    \approx \tfrac{1}{2}\bigl[\bm{v}(i) + \bm{v}(j)\bigr] \cdot \bm{e}_k,
    \\
    \boldsymbol{\psi}_k &= \lambda_i \nabla\lambda_j + \lambda_j \nabla\lambda_i,
    &\quad
    l^1_k(\bm{v}) &= 3\int_{\bm{e}_k} \bm{v} \cdot \bm{t}
    (\lambda_i - \lambda_j)\, ds
    \approx \tfrac{1}{2}\bigl[\bm{v}(i) - \bm{v}(j)\bigr] \cdot \bm{e}_k.
\end{alignat}
The second-family N\'{e}d\'{e}lec space on $K$ is
\begin{equation}
    \mathbf{NE}^1(K) = \mathrm{span}\{\boldsymbol{\phi}_k,\, \boldsymbol{\psi}_k 
    : k=1,\ldots,\frac{d(d+1)}{2}\}.
\end{equation}
For each $\bm{k}\in\mathcal{B}$, the modified N\'{e}d\'{e}lec spaces on each fictitious domain are defined as subspaces of $\bm{V}$:
\begin{equation}\label{eq:Vhm}
  \bm{V}_h^m
    := \left\{
         \bm{v}_h^m \in \bm{H}(\mathrm{curl};\Omega_h^m) :
         e^{i\bm{k}\cdot\bm{x}}\bm{v}_h^m\big|_K
         \in \mathbf{NE}^1(K),\;
         \forall\, K\cap\Omega^m\neq\emptyset,\; m=\pm
       \right\}.
\end{equation}
The global extended space and its periodic subspace are
\begin{align}
  \bm{V}_h^{\mathrm{ext}}
    &:= \left\{
         \bm{v}_h = (\bm{v}_h^+,\bm{v}_h^-) :
         \bm{v}_h^m\in\bm{V}_h^m,\; m=\pm
       \right\}
     \subset \bm{V},
  \label{eq:Vhext}\\
  \bm{V}_{h,\mathrm{per}}^{\mathrm{ext}}
    &:= \left\{
         \bm{v}_h\in\bm{V}_h^{\mathrm{ext}} :
         \bm{v}_h^m(\bm{x}+\bm{R})=\bm{v}_h^m(\bm{x})
         \text{ on }\partial\Omega,\;\bm{R}\in A
       \right\}.
  \label{eq:Vhext_per}
\end{align}
For any interface element $K\in\mathcal{T}_\Sigma$, two independent sets of basis functions are associated with $K$: one from $\bm{V}_h^+$ and one from $\bm{V}_h^-$, corresponding to the degrees of freedom on each side of the interface.
\begin{remark}
  The condition $e^{i\bm{k}\cdot\bm{x}}\bm{v}_h^m|_K\in\mathbf{NE}^1(K)$
  in~\eqref{eq:Vhm} means that the basis functions of $\bm{V}_h^m$ are of the form $e^{-i\bm{k}\cdot\bm{x}}\bm{\phi}$, where $\bm{\phi}$ ranges over $\mathbf{NE}^1(K)$.  The degrees of freedom, mesh
  structure, and dimension are identical to those of the standard N\'{e}d\'{e}lec space.
\end{remark}
Similarly, the modified local $P^2$-conforming spaces and their global extended space are
\begin{align}
  P_h^m
    &:= \left\{ q_h^m \in H^1(\Omega_h^m) : e^{i\bm{k}\cdot\bm{x}}q_h^m|_K \in P^2(K),\; \forall\, K\cap\Omega^m\neq\emptyset,\; m=\pm\right\},
  \label{eq:Phm}\\
  \bm{P}_h^{\mathrm{ext}}
    &:= \left\{
         \bm{q}_h = (q_h^+,q_h^-) :
         q_h^m\in P_h^m,\; m=\pm
       \right\}
     \subset P,
  \label{eq:Phext}\\
  \bm{P}_{h,\mathrm{per}}^{\mathrm{ext}}
    &:= \left\{
         \bm{q}_h\in\bm{P}_h^{\mathrm{ext}} :
         q_h^m(\bm{x}+\bm{R})=q_h^m(\bm{x})
         \text{ on }\partial\Omega,\;\bm{R}\in A,\;m=\pm
       \right\}.
  \label{eq:Phext_per}
\end{align}

Recall the standard discrete de Rham complex \cite{brezzi2012mixed, monk2003finite}:
\begin{equation}\label{eq:discrete-derham}
\begin{array}{ccccccccc}
H^1(\Omega)
& \xrightarrow{\nabla} & \bm{H}(\mathrm{curl};\Omega)
& \xrightarrow{\nabla\times} & \bm{H}(\mathrm{div};\Omega)
& \xrightarrow{\nabla\cdot} & L^2(\Omega) \\[0.4em]
\Big\downarrow\tilde{\Pi}_h^{\mathrm{grad}}
&& \Big\downarrow\tilde{\Pi}_h^{\mathrm{curl}}
&& \Big\downarrow\tilde{\Pi}_h^{\mathrm{div}}
&&\Big\downarrow\tilde{\Pi}_h^{0}
\\[0.4em]
P^2
& \xrightarrow{\nabla} & \mathbf{NE}^1
& \xrightarrow{\nabla\times} &\mathbf{RT}
& \xrightarrow{\nabla\cdot} &P^0
\end{array}
\end{equation}
The interpolation operators for the original (unshifted) spaces $\tilde{\bm{V}}_h^{\mathrm{ext}}$ and $\tilde{\bm{P}}_h^{\mathrm{ext}}$ are denoted
\[
  \tilde{\Pi}_h^{\mathrm{curl}}
    : \bm{H}(\mathrm{curl};\Omega_h^\pm) \to \tilde{\bm{V}}_h^{\mathrm{ext}},
  \qquad
  \tilde{\Pi}_h^{\mathrm{grad}}
    : H^1(\Omega_h^\pm) \to \tilde{\bm{P}}_h^{\mathrm{ext}},
\]
and satisfy the standard commuting relation~\eqref{eq:discrete-derham}:
\begin{equation}\label{eq:standard_comm}
  \tilde{\Pi}_h^{\mathrm{curl}}  \nabla
  = \nabla  \tilde{\Pi}_h^{\mathrm{grad}}.
\end{equation}
The corresponding modified interpolants acting on $\bm{V}$ and $P$
are defined by conjugation with $\Phi_{\bm{k}}$:
\begin{equation}\label{eq:bloch_interp}
  \Pi_h^{\mathrm{curl}}
    := \Phi_{\bm{k}}^{-1}  \tilde{\Pi}_h^{\mathrm{curl}}  \Phi_{\bm{k}},
  \qquad
  \Pi_h^{\mathrm{grad}}
    := \Phi_{\bm{k}}^{-1}  \tilde{\Pi}_h^{\mathrm{grad}}  \Phi_{\bm{k}}.
\end{equation}
\begin{remark}
The discrete de Rham complex \eqref{eq:discrete-derham} is stated in three-dimensional form. In the two-dimensional case, the corresponding structure is simpler since $\mathrm{curl}\,\bm v = \partial_x v_2-\partial_y v_1$. Accordingly, for the second-family N\'ed\'elec element of degree one, the relevant local sequence may be viewed as $P^2 \xrightarrow{\nabla} \mathbf{NE}^1 \xrightarrow{\mathrm{curl}} P^0$; see, e.g., \cite{monk2003finite, arnold2018finite} for details.
\end{remark}

The following modified commuting diagram provides a key tool for the subsequent stability analysis of our proposed method in the next section.
\begin{lemma}[Modified commuting diagram]
\label{lem:commuting}
For $\bm{k} \in \mathcal{B}$, the modified spaces and interpolants defined
in~\eqref{eq:Vhext_per}--\eqref{eq:Phext_per} and~\eqref{eq:bloch_interp} satisfy the commuting diagram
\begin{equation}\label{eq:bloch_comm_diag}
\begin{array}{ccccc}
P
& \xrightarrow{\;\nabla_{\bm{k}}\;} & \bm{V} \\[0.7em] \Big\downarrow\,\Pi_h^{\mathrm{grad}}
&& \Big\downarrow\,\Pi_h^{\mathrm{curl}} \\[0.5em] \bm{P}_{h,\mathrm{per}}^{\mathrm{ext}}
& \xrightarrow{\;\nabla_{\bm{k}}\;} & \bm{V}_{h,\mathrm{per}}^{\mathrm{ext}}
\end{array}
\end{equation}
that is,
\begin{equation}\label{eq:bloch_comm}
  \Pi_h^{\mathrm{curl}}  \nabla_{\bm{k}}
  = \nabla_{\bm{k}}  \Pi_h^{\mathrm{grad}}.
\end{equation}
Consequently, for every $q_h \in \bm{P}_{h,\mathrm{per}}^{\mathrm{ext}}$:
\begin{equation}\label{eq:comm_fe}
  \Pi_h^{\mathrm{curl}}(\nabla_{\bm{k}} q_h)
  = \nabla_{\bm{k}} q_h
  \qquad\text{and}\qquad
  \nabla_{\bm{k}} q_h \in \bm{V}_{h,\mathrm{per}}^{\mathrm{ext}}.
\end{equation}
\end{lemma}
\begin{proof}
The proof reduces the shifted statement to the unshifted one via the Bloch identities~\eqref{eq:bloch_grad}--\eqref{eq:bloch_curl}, following \cite[Theorem~4.1]{dobson2001analysis}. From Definition~\ref{def:bloch_map}, for any smooth $q$ and $\bm{v}$:
\begin{align}
  \Phi_{\bm{k}}(\nabla_{\bm{k}} q)
    &= \nabla(\Phi_{\bm{k}} q),
  \label{eq:id1}\\
  \Phi_{\bm{k}}(\nabla_{\bm{k}}\times\bm{v})
    &= \nabla\times(\Phi_{\bm{k}}\bm{v}).
  \label{eq:id2}
\end{align}
Let $q \in P$. Using~\eqref{eq:standard_comm}, \eqref{eq:bloch_interp}, and~\eqref{eq:id1}:
\begin{align*}
  \Pi_h^{\mathrm{curl}}(\nabla_{\bm{k}} q)
  &= \Phi_{\bm{k}}^{-1}\Bigl(
     \tilde{\Pi}_h^{\mathrm{curl}}\bigl(\Phi_{\bm{k}}(\nabla_{\bm{k}} q)\bigr)
     \Bigr)
   = \Phi_{\bm{k}}^{-1}\Bigl(
     \tilde{\Pi}_h^{\mathrm{curl}}\bigl(\nabla(\Phi_{\bm{k}} q)\bigr)
     \Bigr)
    \\
  &= \Phi_{\bm{k}}^{-1}\Bigl(
     \nabla\,\tilde{\Pi}_h^{\mathrm{grad}}(\Phi_{\bm{k}} q)
     \Bigr)
   = \nabla_{\bm{k}}\Bigl(
     \Phi_{\bm{k}}^{-1}\tilde{\Pi}_h^{\mathrm{grad}}(\Phi_{\bm{k}} q)
     \Bigr)
     = \nabla_{\bm{k}}(\Pi_h^{\mathrm{grad}} q).
\end{align*}
This gives~\eqref{eq:bloch_comm}.
For~\eqref{eq:comm_fe}, let $q_h \in \bm{P}_{h,\mathrm{per}}^{\mathrm{ext}}$. By definition~\eqref{eq:Phext_per}, $\Phi_{\bm{k}} q_h \in \tilde{\bm{P}}_{h,\mathrm{per}}^{\mathrm{ext}}$,
so the interpolant fixes it:
$\tilde{\Pi}_h^{\mathrm{grad}}(\Phi_{\bm{k}} q_h) = \Phi_{\bm{k}} q_h$, hence $\Pi_h^{\mathrm{grad}} q_h = q_h$.
And~\eqref{eq:bloch_comm} gives
\[
  \Pi_h^{\mathrm{curl}}(\nabla_{\bm{k}} q_h)
  = \nabla_{\bm{k}}(\Pi_h^{\mathrm{grad}} q_h)
  = \nabla_{\bm{k}} q_h.
\]
Since $\Pi_h^{\mathrm{curl}}$ maps into $\bm{V}_{h,\mathrm{per}}^{\mathrm{ext}}$
by~\eqref{eq:bloch_interp} and~\eqref{eq:Vhext_per},
we conclude $\nabla_{\bm{k}} q_h \in \bm{V}_{h,\mathrm{per}}^{\mathrm{ext}}$.
\end{proof}

\subsection{Mixed Unfitted Approximation}
\label{sec:Mixedmethod}
We define the sesquilinear form
$a_h(\cdot,\cdot):\;
\bm{V}_{h,\mathrm{per}}^{\mathrm{ext}}
\times \bm{V}_{h,\mathrm{per}}^{\mathrm{ext}}
\longrightarrow \mathbb{C}$
as follows:
\begin{equation}\label{eq:nitche_ah}
\begin{aligned}
a_h(\bm u_h,\bm v_h)
=&\sum_{m} \int_{\Omega^m}(\varepsilon_{m} ^{-1} \nabla_{\bm{k}} \times \bm{u}_h) \cdot \overline{(\nabla_{\bm{k}} \times \bm{v}_h)} \, d\Omega 
+ \int_{\Sigma} \{ \varepsilon^{-1} \nabla_{\bm{k}} \times \bm{u}_h \}_{\theta} \cdot [\![ \overline{(\bm{v}_h \times \bm{n})}]\!] \, dS 
\\
&+ \int_{\Sigma} \{ \varepsilon^{-1} \nabla_{\bm{k}} \times \bm{v}_h \}_{\theta} \cdot [\![ \overline{(\bm{u}_h \times \bm{n})}]\!] \, dS
+ \frac{\gamma}{h}\int_{\Sigma}[\![\bm{u}_h \times \textbf{n}]\!]\cdot \overline{[\![\bm{v}_h \times \textbf{n}]\!]} \, dS,
\end{aligned}
\end{equation}
\begin{align}
b_h(p_h, \bm{v}_h) =&
\sum_{m = \pm} \int_{\Omega^m} (\nabla_{\bm{k}} p_h^m) \cdot \overline{\bm{v}_h^m} \, dx 
- \int_{\Sigma} [\![p_h]\!]\,\{\overline{\bm{v}_h\cdot\bm{n}}\}_{\theta}\,ds,
\end{align}
\begin{align}
(\bm{u}_h, \bm{v}_h) =& \int_{\Omega^{m}} \bm{u}_h \cdot \overline{\bm{v}_h} \, dx, \\
J_h(p_h,q_h) =& \frac{\gamma_q}{h}\int_{\Sigma} [\![ p_h ]\!] \cdot \overline{[\![ q_h ]\!]}\,ds.\label{eq:interface-penalty}
\end{align}
Here $\{\cdot\}_{\theta}$ is the weighted average across the interface introduced in \eqref{eq:averange}, \(J_h\) is an interface jump stabilization, and $\mathcal F_h^g$ is the set of ghost faces defined in \eqref{eq:Fh_Sigma}.
For $m=\pm$, let $\mathcal F_h^{g,m}$ consist of the faces in
$\mathcal F_h^g$ whose two adjacent elements both belong to the active
mesh on $\Omega_h^m$. The ghost penalties act separately on the two
material branches. With $h_F=\operatorname{diam}(F)$ and a fixed unit
normal $\bm n_F$ on each face, we define
\begin{align}
g_u(\bm u_h,\bm v_h) 
&:= \sum_{m=\pm} \sum_{F\in\mathcal F_h^{g,m}} \gamma_u h_F \int_F [\![(\nabla_{\bm k}\times \bm u_h^m)]\!]_F
\cdot \overline{ [\![(\nabla_{\bm k}\times \bm v_h^m)]\!]_F }\,dS ,\label{eq:ghost-u}\\
g_p(p_h,q_h)
&:= \sum_{m=\pm}\sum_{j=0}^{1} \sum_{F\in\mathcal F_h^{g,m}}
\gamma_p h_F^{2j+1} \int_F [\![\partial_{n_F}^j(\nabla_{\bm k}p_h^m)]\!]_F
\cdot \overline{ [\![\partial_{n_F}^j(\nabla_{\bm k}q_h^m)]\!]_F }\,dS .\label{eq:ghost-p}
\end{align}
Here $[\![\cdot]\!]_F$ denotes the jump across $F$ within the same
material branch, and $\gamma_u,\gamma_p>0$ are fixed independently of $h$
and the cut configuration. 
Denote
\begin{align}
A_h(\bm{u}_h, \bm{v}_h) =& a_h(\bm u_h,\bm v_h)+ g_u(\bm u_h,\bm v_h), \\
S_h(p_h,q_h) &= J_h(p_h,q_h) + g_p(p_h,q_h),\label{eq:stabilization-p}
\end{align}
Then, for a given quasi-momentum \( \bm{k} \in \mathcal{B} \), the proposed MUFEM discretization, with the interface conditions imposed by Nitsche's method, is defined as follows:
find $(\omega_h ^{2}, \bm{u}_h, p_h) \in \mathbb{R}^{+} \times \bm{V}_{h,\mathrm{per}}^{\mathrm{ext}} \times \bm{P}_{h,\mathrm{per}}^{\mathrm{ext}}$ such that:
\begin{equation}\label{eq:final}
\left\{
\begin{aligned}
A_h(\bm{u}_h, \bm{v}_h) + b_h(p_h, \bm{v}_h) &= \omega_h^2 (\bm{u}_h, \bm{v}_h)
&& \forall \bm{v}_h \in \bm{V}_{h,\mathrm{per}}^{\mathrm{ext}}, \\
\overline{b_h(q_h, \bm{u}_h)} - S_h(p_h,q_h) &= 0
&& \forall q_h \in \bm{P}_{h,\mathrm{per}}^{\mathrm{ext}}.
\end{aligned}
\right.
\end{equation}
Equivalently, we define the global sesquilinear form on $\bigl(\bm V_{h,\mathrm{per}}^{\mathrm{ext}}\times P_{h,\mathrm{per}}^{\mathrm{ext}}\bigr)\times \bigl(\bm V_{h,\mathrm{per}}^{\mathrm{ext}}\times P_{h,\mathrm{per}}^{\mathrm{ext}}\bigr)\rightarrow \mathbb C$ by
\begin{equation}\label{eq:mathcal_Ah}
\mathcal{A}_h[(\bm u_h,p_h),(\bm v_h,q_h)]
:=
A_h(\bm u_h,\bm v_h)
+b_h(p_h,\bm v_h)
+\overline{b_h(q_h,\bm u_h)}
-S_h(p_h,q_h).
\end{equation}
Then the discrete eigenvalue problem \eqref{eq:final} can be written as: find
$(\omega_h^2,\bm u_h,p_h)
\in \mathbb R_+
\times
\bm V_{h,\mathrm{per}}^{\mathrm{ext}}
\times
P_{h,\mathrm{per}}^{\mathrm{ext}}$,
with $\bm u_h\neq 0$, such that
\begin{equation}\label{eq:compact-discrete-eigenproblem}
\mathcal A_h[(\bm u_h,p_h),(\bm v_h,q_h)]
=
\omega_h^2(\bm u_h,\bm v_h)
\qquad
\forall
(\bm v_h,q_h)
\in
\bm V_{h,\mathrm{per}}^{\mathrm{ext}}
\times
P_{h,\mathrm{per}}^{\mathrm{ext}}.
\end{equation}

%% file: section4.tex
\section{Stability}\label{sec:stability}
Throughout this section, the quasi-momentum
$\bm{k}\in\mathcal{B}\setminus\{\bm{0}\}$ is fixed. To show the stability of
the unfitted method, for any $(\bm v_h,q_h)\in \bm
V_{h,\mathrm{per}}^{\mathrm{ext}}\times \bm
P_{h,\mathrm{per}}^{\mathrm{ext}}$, we define the following norms and
seminorms.
\begin{equation}\label{eq:seminorm_V_h}
|\bm v_h|_{\bm V_h}^2
:=
\sum_{m=\pm}\left\| \varepsilon_m^{-1/2}\nabla_{\bm k}\times \bm v_h\right\|_{0,\Omega_h^m}^2  
+ \sum_{K\in \mathcal T_\Sigma} \frac{\gamma}{h}\left\| [\![\bm v_h\times \bm n]\!] \right\|_{0,\Sigma_K}^2 + g_u(\bm v_h,\bm v_h),
\end{equation}
\begin{equation}\label{eq:norm_V_h}
\|\bm v_h\|_{\bm V_h}^2
:= |\bm v_h|_{\bm V_h}^2+\sum_{m=\pm} \left\|\bm v_h^m \right\|_{0,\Omega_h^m }^2 
,
\end{equation}
\begin{equation}\label{eq:all_Vh_Qh}
\|q_h\|_{Q_h}^2 
:=
\sum_{m=\pm}\left\| \nabla_{\bm k} q_h^m \right\|_{0,\Omega_h^m}^2
+ \sum_{K\in \mathcal T_\Sigma}\frac{1}{h}\left\|[\![q_h]\!]\right\|_{0,\Sigma_K}^2 + g_p( q_h, q_h).
\end{equation}
Accordingly, the product norm on
$\bm V_{h,\mathrm{per}}^{\mathrm{ext}}\times \bm P_{h,\mathrm{per}}^{\mathrm{ext}}$
is defined by
\begin{equation}\label{eq:product-norm}
\|(\bm v_h,q_h)\|_h^2
:= \|\bm v_h\|_{\bm V_h}^2+\|q_h\|_{Q_h}^2.
\end{equation}

For a fixed smooth interface and sufficiently small $h$,
the trace inequality holds uniformly in the cut position.

\begin{lemma}\label{lemma:shifted-trace-curl}
For each interface element $K\in\mathcal T_\Sigma$ and any $\bm v_h\in \bm V_h^m$, $m=\pm$, there is $C>0$, independent of $h$ and the cut position, with
\begin{equation}\label{eq:shifted-trace-curl}
\left\| \nabla_{\bm k}\times \bm v_h^m \right\|_{0,\Sigma_K}^2
\le C h^{-1}\left\| \nabla_{\bm k}\times \bm v_h^m \right\|_{0,K}^2 .
\end{equation}
\end{lemma}
\begin{proof}
By~\eqref{eq:Vhm},
$e^{i\bm k\cdot\bm x}\bm v_h^m|_K
\in\mathbf{NE}^1(K)=[P^1(K)]^d$.
Hence, $z:=e^{i\bm k\cdot\bm x}\nabla_{\bm k}\times\bm v_h^m
=\nabla\times\bigl(e^{i\bm k\cdot\bm x}\bm v_h^m\bigr)$ is constant on $K$ by~\eqref{eq:bloch_curl}.
On the uniform mesh, the cut-element trace inequality~\cite{burman2015cutfem}
and $\nabla z=0$ give
\[
\|z\|_{0,\Sigma_K}^2
\le C\left(h^{-1}\|z\|_{0,K}^2+h\|\nabla z\|_{0,K}^2\right)
=Ch^{-1}\|z\|_{0,K}^2.
\]
Since $|e^{i\bm k\cdot\bm x}|=1$, this proves~\eqref{eq:shifted-trace-curl},
with $C$ independent of the cut position and $\bm k$.
\end{proof}

Following~\cite{gurkan2019stabilized}, for each $m=\pm$, we call
$K\in\mathcal T_h$ active if $K\cap\Omega^m\ne\emptyset$.
The union of these elements is $\Omega_h^m$.
We assume that each active element is connected through faces in
$\mathcal F_h^{g,m}$ to an active element $K_*$ satisfying
$|K_*\cap\Omega^m|\ge c_*|K_*|$.
The path lengths and overlap are uniformly bounded, and $c_*>0$ is
independent of $h$ and the cut position.

\begin{lemma}\label{lemma:trace_equ}
Under the above mesh assumptions, for every
$\bm v_h\in\bm V_{h,\mathrm{per}}^{\mathrm{ext}}$, there is
$C_{\mathrm{tr}}>0$, independent of $h$ and the cut position, such that
\begin{equation}\label{eq:trace1}
\sum_{K\in\mathcal T_\Sigma}h
\left\|\left\{\varepsilon^{-1}\nabla_{\bm k}\times\bm v_h
\right\}_\theta\right\|_{0,\Sigma_K}^2
\le C_{\mathrm{tr}}\left(
\sum_{m=\pm}\left\|\varepsilon_m^{-1/2}
\nabla_{\bm k}\times\bm v_h^m\right\|_{0,\Omega^m}^2
+g_u(\bm v_h,\bm v_h)\right).
\end{equation}
\end{lemma}
\begin{proof}
Set $\bm w_h:=\nabla_{\bm k}\times\bm v_h$ and $\bm w_h^m:=\nabla_{\bm k}\times\bm v_h^m$. By Lemma~\ref{lemma:shifted-trace-curl}, $\bm\phi_h^m:=e^{i\bm k\cdot\bm x}\bm w_h^m$ is constant on each element. Since $|e^{i\bm k\cdot\bm x}|=1$, $\bm\phi_h^m$ and $\bm w_h^m$ have the same $L^2$ norms on every element and face. The mesh assumptions give, for each active element, a path of uniformly bounded length to some $K_*$ with $|K_*\cap\Omega^m|\ge c_*|K_*|$, with uniformly bounded overlap between paths. Applying the extension argument of \cite[Lemma~2.19]{gurkan2019stabilized} to $\bm\phi_h^m$ gives
\begin{equation}
\|\bm w_h^m\|_{0,\Omega_h^m}^2
\le C\Big(\|\bm w_h^m\|_{0,\Omega^m}^2
+\sum_{F\in\mathcal F_h^{g,m}} h_F\|[\![\bm w_h^m]\!]_F\|_{0,F}^2\Big).
\end{equation}
Multiplying by $\varepsilon_m^{-1}$, summing over $m=\pm$,
and using~\eqref{eq:ghost-u},
\begin{equation}
\sum_{m=\pm}\varepsilon_m^{-1}\|\bm w_h^m\|_{0,\Omega_h^m}^2
\le C\Big(\sum_{m=\pm}\varepsilon_m^{-1}\|\bm w_h^m\|_{0,\Omega^m}^2
+g_u(\bm v_h,\bm v_h)\Big).
\end{equation}
Using $0\le\theta_K^m\le1$, $\theta_K^++\theta_K^-=1$, and
Lemma~\ref{lemma:shifted-trace-curl}, we obtain
\begin{align*}
\sum_{K\in\mathcal T_\Sigma}h
\left\|\left\{\varepsilon^{-1}\bm w_h\right\}_\theta
\right\|_{0,\Sigma_K}^2
&\le\sum_{K\in\mathcal T_\Sigma}\sum_{m=\pm}
h\theta_K^m\varepsilon_m^{-2}\|\bm w_h^m\|_{0,\Sigma_K}^2
\le C\sum_{m=\pm}\varepsilon_m^{-1}
\|\bm w_h^m\|_{0,\Omega_h^m}^2.
\end{align*}
Combining the two estimates proves~\eqref{eq:trace1}.
\end{proof}

Using the above trace estimate, the interface terms can be controlled uniformly with respect to the cut position. We define the discrete kernel of $\nabla_{\bm k}$ by
\begin{equation}\label{eq:Z_hPer}
\bm Z_{h,\mathrm{per}}^{\mathrm{ext}} := \left\{\bm v_h\in \bm V_{h,\mathrm{per}}^{\mathrm{ext}}:\; b_h(q_h,\bm v_h)=0\quad \forall q_h\in \bm P_{h,\mathrm{per}}^{\mathrm{ext}} \right\}.
\end{equation}
The coercivity of $A_h$ is required only on $\bm Z_{h,\mathrm{per}}^{\mathrm{ext}}$.

\begin{lemma}\label{lemma:ah-cont-coer}
Assume that the penalty parameter $\gamma > 0$ is sufficiently large. Then there exist constants $C_{\mathrm{cont}}, C_{\mathrm{coer}} > 0$,
independent of $h$ and the interface, such that
\begin{align}
|A_h(\bm{u}_h, \bm{v}_h)|
&\leq C_{\mathrm{cont}}\|\bm{u}_h\|_{\bm{V}_h}\|\bm{v}_h\|_{\bm{V}_h}
\qquad \forall\, \bm{u}_h, \bm{v}_h \in \bm{V}_{h,\mathrm{per}}^{\mathrm{ext}},
\label{eq:ah-cont}\\
Re\, A_h(\bm{v}_h, \bm{v}_h)
&\geq C_{\mathrm{coer}}|\bm{v}_h|_{\bm{V}_h}^2
\qquad \forall\, \bm{v}_h \in \bm Z_{h,\mathrm{per}}^{\mathrm{ext}}.
\label{eq:ah-coer}
\end{align}
\end{lemma}
\begin{proof}
The continuity estimate \eqref{eq:ah-cont} follows directly from the definition of $A_h$, the Cauchy--Schwarz inequality, and the trace estimate \eqref{eq:trace1}. We only prove coercivity.
Taking $\bm u_h=\bm v_h$ in
\eqref{eq:nitche_ah} and taking the real part, we get
\begin{align*}
\mathrm{Re}\,A_h(\bm v_h,\bm v_h)
=&\sum_{m=\pm}\left\|\varepsilon_m^{-1/2}
\nabla_{\bm k}\times \bm v_h^m\right\|_{0,\Omega^m}^2+ \sum_{K\in \mathcal T_\Sigma}\frac{\gamma}{h}\left\|[\![\bm v_h\times \bm n]\!]\right\|_{0,\Sigma_K}^2\\
&\quad +2\,\mathrm{Re}\sum_{K\in \mathcal T_\Sigma}
\left( \left\{\varepsilon^{-1}\nabla_{\bm k}\times \bm v_h\right\}_\theta, [\![\bm v_h\times \bm n]\!]\right)_{\Sigma_K} + g_u(\bm v_h,\bm v_h) .
\end{align*}
By Young's inequality with parameter $\eta>0$, we have
\begin{align*}
2\left|\sum_{K\in \mathcal T_\Sigma}\left(\left\{\varepsilon^{-1}\nabla_{\bm k}\times \bm v_h\right\}_\theta,[\![\bm v_h\times \bm n]\!]\right)_{\Sigma_K}\right|
\leq
&\eta\sum_{K\in \mathcal T_\Sigma} h
\left\|\left\{\varepsilon^{-1}\nabla_{\bm k}\times \bm v_h\right\}_\theta\right\|_{0,\Sigma_K}^2
+\frac{1}{\eta}\sum_{K\in \mathcal T_\Sigma}\frac{1}{h}
\left\|[\![\bm v_h\times \bm n]\!]\right\|_{0,\Sigma_K}^2 .
\end{align*}
Therefore,
\begin{align}\label{eq:coer-1-new}
\mathrm{Re}\,A_h(\bm v_h,\bm v_h)
\geq&\sum_{m=\pm}\left\|\varepsilon_m^{-1/2}
\nabla_{\bm k}\times \bm v_h^m\right\|_{0,\Omega^m}^2
-\eta\sum_{K\in \mathcal T_\Sigma} h
\left\|\left\{\varepsilon^{-1}\nabla_{\bm k}\times \bm v_h\right\}_\theta \right\|_{0,\Sigma_K}^2 \nonumber\\
&+\left(\gamma-\frac{1}{\eta}\right) \sum_{K\in \mathcal T_\Sigma}\frac{1}{h} \left\|[\![\bm v_h\times \bm n]\!]\right\|_{0,\Sigma_K}^2 + g_u(\bm v_h,\bm v_h).
\end{align}
Lemma~\ref{lemma:trace_equ} gives
\begin{align}\label{eq:trace-global-new}
\sum_{K\in \mathcal T_\Sigma} h
\left\| \left\{\varepsilon^{-1}\nabla_{\bm k}\times \bm v_h\right\}_\theta \right\|_{0,\Sigma_K}^2
&\leq C_{\mathrm{tr}} (\sum_{m=\pm} \left\|\varepsilon_m^{-1/2}
\nabla_{\bm k}\times \bm v_h^m \right\|_{0,\Omega^m}^2 + g_u(\bm v_h,\bm v_h) ).
\end{align}
Substituting \eqref{eq:trace-global-new} into \eqref{eq:coer-1-new}, we obtain
\begin{align}\label{eq:coer-2-new}
\mathrm{Re}\,A_h(\bm v_h,\bm v_h) \geq& (1-\eta C_{\mathrm{tr}})
(\sum_{m=\pm} \left\|\varepsilon_m^{-1/2} \nabla_{\bm k}\times \bm v_h^m \right\|_{0,\Omega^m}^2 + g_u(\bm v_h,\bm v_h) )
+ \left(\gamma-\frac{1}{\eta}\right) \sum_{K\in \mathcal T_\Sigma}\frac{1}{h} \left\|[\![\bm v_h\times \bm n]\!]\right\|_{0,\Sigma_K}^2 .
\end{align}
We first choose $\eta>0$ sufficiently small such that $1-\eta C_{\mathrm{tr}}>0$, and then choose $\gamma>1/\eta$. Define $c_1:=1-\eta C_{\mathrm{tr}}>0$ and $c_2:=\gamma-\frac{1}{\eta}>0$.
Then \eqref{eq:coer-2-new} gives
\begin{align*}
\mathrm{Re}\,A_h(\bm v_h,\bm v_h)
\geq{}& c_1\sum_{m=\pm}
\left\|\varepsilon_m^{-1/2}\nabla_{\bm k}\times\bm v_h^m
\right\|_{0,\Omega^m}^2
+c_1g_u(\bm v_h,\bm v_h) 
+c_2\sum_{K\in\mathcal T_\Sigma}\frac{1}{h}
\left\|[\![\bm v_h\times\bm n]\!]\right\|_{0,\Sigma_K}^2 \\
\geq{}& C_{\mathrm{coer}}\Bigg(
\sum_{m=\pm}\left\|\varepsilon_m^{-1/2}\nabla_{\bm k}\times\bm v_h^m
\right\|_{0,\Omega^m}^2+g_u(\bm v_h,\bm v_h)
+\sum_{K\in\mathcal T_\Sigma}\frac{\gamma}{h}
\left\|[\![\bm v_h\times\bm n]\!]\right\|_{0,\Sigma_K}^2
\Bigg),
\end{align*}
where $C_{\mathrm{coer}}:=\min\left\{c_1,\frac{c_2}{\gamma}\right\}>0$. Hence, $\mathrm{Re}\,A_h(\bm v_h,\bm v_h) \geq C_{\mathrm{coer}}|\bm v_h|_{\bm V_h}^2$, which proves \eqref{eq:ah-coer}.
\end{proof}

\begin{remark}\label{lem:seminorm-kernel}
$|\cdot|_{\bm V_h}$ is a norm on the discrete
kernel $\bm Z_{h,\mathrm{per}}^{\mathrm{ext}}$, equivalent there to $\|\cdot\|_{\bm V_h}$: there exists a constant $C>0$, independent of $h$, with $\|\bm v_h\|_{\bm V_h}\le C|\bm v_h|_{\bm V_h}$ for all $\bm v_h\in \bm Z_{h,\mathrm{per}}^{\mathrm{ext}}$.
\end{remark}

To complete the stability analysis of the mixed formulation, we also need a uniform stabilized inf-sup condition for the coupling form $b_h$.
\begin{lemma}\label{lem:stabilized-infsup}
For fixed $\bm{k}\in\mathcal{B}\setminus\{\bm{0}\}$, there exists a constant $\beta_0>0$, independent of $h$ and the interface 
position, such that
\begin{equation}\label{eq:stabilized-infsup}
\sup_{\bm v_h\in \bm V_{h,\mathrm{per}}^{\mathrm{ext}}\setminus\{\bm 0\}}
\frac{|b_h(q_h,\bm v_h)|}{\|\bm v_h\|_{\bm V_h}}+
S_h(q_h,q_h)^\frac{1}{2}
\ge \beta_0 \|q_h\|_{Q_h} \qquad \forall\, q_h\in P_{h,\mathrm{per}}^{\mathrm{ext}}.
\end{equation}
\end{lemma}
\begin{proof}
Fix $q_h\in P_{h,\mathrm{per}}^{\mathrm{ext}}$. By the
commuting diagram~\eqref{eq:comm_fe} of Lemma~\ref{lem:commuting},
$\nabla_{\bm k}q_h\in \bm V_{h,\mathrm{per}}^{\mathrm{ext}}$. Define the
admissible test function $\bm w_h$ componentwise by
\begin{equation}\label{eq:choice-wh}
\bm w_h^m|_K :=
\begin{cases}
\nabla_{\bm k}q_h^m|_K, & K\subset \omega_h^m,\\
\bm 0, & K\in\mathcal T_\Sigma,
\end{cases} \qquad m=\pm .
\end{equation}
Since $[\![q_h]\!]$ and $\{\overline{\bm w_h\cdot\bm n}\}_\theta$ cancel on $\Sigma$, by the definition of $b_h$,
\begin{equation}\label{eq:bh-wh}
b_h(q_h, \bm w_h)= \sum_{m=\pm}\left\|\nabla_{\bm k} q_h^m\right\|_{0,\omega_h^m}^2,
\end{equation}
and because $\nabla_{\bm k}\times(\nabla_{\bm k} q_h^m)=\bm 0$ and $\bm w_h$ vanishes on cut elements, there is $C_w>0$ with
\begin{equation}\label{eq:wh-norm}
\|\bm w_h\|_{\bm V_h}^2 = \sum_{m=\pm} \left\| \nabla_{\bm k} q_h^m \right\|_{0,\Omega^m \cap K}^2 \leq C_w^{-1}\sum_{m=\pm}\left\|\nabla_{\bm k} q_h^m\right\|_{0,\omega_h^m}^2 .
\end{equation}
Hence
\begin{equation}\label{eq:sub-control}
\sup_{\bm v_h\ne \bm 0}\frac{|b_h(q_h,\bm v_h)|}{\|\bm v_h\|_{\bm V_h}}
\ge \frac{|b_h(q_h,\bm w_h)|}{\|\bm w_h\|_{\bm V_h}} \ge C_w \Big( \sum_{m=\pm} \|\nabla_{\bm k}q_h^m\|_{0,\omega_h^m}^2 \Big)^{1/2}.
\end{equation}
By the ghost penalty extension property, there is $C_g>0$, independent of $h$ and the interface position, with
\begin{equation}\label{eq:ghost-extension-q}
\sum_{m=\pm}\|\nabla_{\bm k}q_h^m\|_{0,\Omega_h^m}^2
\le C_g \Big( \sum_{m=\pm} \|\nabla_{\bm k}q_h^m\|_{0,\omega_h^m}^2 + g_p(q_h,q_h) \Big),
\end{equation}
and by definition $S_h(q_h,q_h):=\tfrac{\gamma_q}{h}\|[\![q_h]\!]\|_{0,\Sigma}^2 +g_p( q_h, q_h)$ for a parameter $\gamma_q>0$.  Combining \eqref{eq:sub-control} with $S_h$ gives
\begin{align}
&\sup_{\bm v_h\ne \bm0} \frac{|b_h(q_h,\bm v_h)|}{\|\bm v_h\|_{\bm V_h}} +S_h(q_h,q_h)^\frac{1}{2}
\geq   C_w \left( \sum_{m=\pm} \|\nabla_{\bm k}q_h^m\|_{0,\omega_h^m}^2 \right)^\frac{1}{2} + \left(\frac{\gamma_q}{h}\left\|[\![q_h]\!]\right\|_{0,\Sigma}^2 +g_p( q_h, q_h) \right)^\frac{1}{2}
\nonumber\\
&\ge c_0 \left[ \left( \sum_{m=\pm} \|\nabla_{\bm k}q_h^m\|_{0,\omega_h^m}^2 \right)^\frac{1}{2}
+ \left( \sum_{K\in\mathcal T_\Sigma}
\frac1h \|[\![q_h]\!]\|_{0,\Sigma_K}^2 \right)^\frac{1}{2}
+ g_p(q_h,q_h)^\frac{1}{2} \right],
\label{eq:combine}
\end{align}
with $c_0 = \min\left\{ C_w,\, \frac{1}{\sqrt2}\min\{\gamma_q^{1/2},1\} \right\}>0$.
By \eqref{eq:ghost-extension-q}, $\|q_h\|_{Q_h}^2\le (C_g+1)$ times the bracketed sum of squares above, so taking $\beta_0 = \frac{c_0}{\sqrt{C_g+1}}$ yields \eqref{eq:stabilized-infsup}.
\end{proof}

Lemma~\ref{lem:press} follows directly from Lemma~\ref{lem:stabilized-infsup}. It converts the supremum estimate in
\eqref{eq:stabilized-infsup} into the existence of a particular test function
$\bm w_h$, which can be stated as follows.
\begin{lemma}\label{lem:press}
Under the stabilized inf-sup condition \eqref{eq:stabilized-infsup}, for every $p_h\in \bm P_{h,\mathrm{per}}^{\mathrm{ext}}$, there exists
$\bm w_h\in \bm V_{h,\mathrm{per}}^{\mathrm{ext}}$ such that
\begin{equation}\label{eq:wh-lif}
\mathrm{Re}\,b_h(p_h,\bm w_h)+S_h(p_h,p_h)
\ge c_b \|p_h\|_{Q_h}^2,
\qquad
\|\bm w_h\|_{\bm V_h}\le C_b \|p_h\|_{Q_h},
\end{equation}
where the constants $c_b,C_b>0$ are independent of $h$.
\end{lemma}

Combining the coercivity of $A_h$ on the kernel in
Lemma~\ref{lemma:ah-cont-coer} with the stabilized inf-sup condition for $b_h$
in Lemma~\ref{lem:stabilized-infsup}, we obtain our main well-posedness result.
\begin{theorem}\label{thm:uni-stability}
Let $\bm{k}\in\mathcal{B}\setminus\{\bm{0}\}$ be fixed. Under the hypotheses of Lemmas~\ref{lemma:ah-cont-coer} and \ref{lem:stabilized-infsup}, there exists a constant $C>0$, independent of $h$ and the interface position, such that
\begin{equation}\label{eq:uniform-infsup-Ah}
\|\bm u_h\|_{\bm V_h}+\|p_h\|_{Q_h}
\le C \sup_{\substack{(\bm v_h,q_h)\ne(\bm 0,0)}}\frac{ \bigl| \mathcal{A}_h\big((\bm u_h,p_h),(\bm v_h,q_h)\big) \bigr|}{ \|\bm v_h\|_{\bm V_h}+\|q_h\|_{Q_h}}
\end{equation}
for all $(\bm u_h,p_h)\in \bm V_{h,\mathrm{per}}^{\mathrm{ext}}\times P_{h,\mathrm{per}}^{\mathrm{ext}}$. Consequently, by the Babu\v{s}ka--Brezzi theorem \cite[Thm~1.2.6]{dorfler2011photonic}, the discrete mixed problem is well posed.
\end{theorem}
\begin{proof}
Fix $(\bm u_h,p_h)$. By Lemma~\ref{lem:press} choose $\bm w_h$ satisfying \eqref{eq:wh-lif}, and choose the test pair $(\bm v_h,q_h):=(\bm u_h+\delta \bm w_h,-p_h)$ for $\delta\in(0,1]$ to be fixed. Expanding $\mathcal A_h$ and taking real parts,
\begin{align}
\mathrm{Re}\, \mathcal A_h\big((\bm u_h,p_h),(\bm v_h,q_h)\big)
={} \mathrm{Re}\,A_h(\bm u_h,\bm u_h) +\delta\,\mathrm{Re}\,A_h(\bm u_h,\bm w_h) 
+ \delta\,\mathrm{Re}\,b_h(p_h,\bm w_h) +S_h(p_h,p_h),
\label{eq:Ah-real-expanded}
\end{align}
since the two $\mathrm{Re}\,b_h(p_h,\bm u_h)$ terms cancel. Set $c_a:=C_{\mathrm{coer}}$ and $C_a:=C_{\mathrm{cont}}$, where $C_{\mathrm{coer}}$ and $C_{\mathrm{cont}}$ are the constants from Lemma~\ref{lemma:ah-cont-coer}. By the coercivity and continuity estimates of Lemma~\ref{lemma:ah-cont-coer} and \eqref{eq:wh-lif},
\[
\mathrm{Re}\, \mathcal A_h\big((\bm u_h,p_h),(\bm v_h,q_h)\big)
\ge c_a\|\bm u_h\|_{\bm V_h}^2 -\delta C_a C_b\|\bm u_h\|_{\bm V_h}\|p_h\|_{Q_h} + \delta c_b\|p_h\|_{Q_h}^2,
\]
using $\delta\,\mathrm{Re}\,b_h(p_h,\bm w_h)+S_h(p_h,p_h) = \delta(\mathrm{Re}\,b_h(p_h,\bm w_h)+S_h(p_h,p_h))+(1-\delta)S_h(p_h,p_h)\ge \delta c_b\|p_h\|_{Q_h}^2$. Applying Young's inequality to the cross term,
\[
\delta C_a C_b \|\bm u_h\|_{\bm V_h}\|p_h\|_{Q_h} \le \tfrac{c_a}{2}\|\bm u_h\|_{\bm V_h}^2 + \tfrac{\delta^2 C_a^2 C_b^2}{2c_a} \|p_h\|_{Q_h}^2,
\]
so
\[
\mathrm{Re}\, \mathcal A_h\big((\bm u_h,p_h),(\bm v_h,q_h)\big) \ge \tfrac{c_a}{2}\|\bm u_h\|_{\bm V_h}^2 +\Big( \delta c_b - \tfrac{\delta^2 C_a^2 C_b^2}{2c_a} \Big)\|p_h\|_{Q_h}^2.
\]
Choose $\delta$ small enough that $\delta c_b -\frac{\delta^2 C_a^2 C_b^2}{2c_a} >0$. Then with $c:= \min\{\frac{c_a}{2}, \, \delta c_b -\frac{\delta^2 C_a^2 C_b^2}{2c_a}\}>0$,
\begin{equation}\label{eq:Ah-lower-final}
\mathrm{Re}\, \mathcal A_h\big((\bm u_h,p_h),(\bm v_h,q_h)\big) \ge c\big( \|\bm u_h\|_{\bm V_h}^2+\|p_h\|_{Q_h}^2 \big).
\end{equation}
Also, by \eqref{eq:wh-lif}, $\|\bm v_h\|_{\bm V_h}+\|q_h\|_{Q_h} \le C_\delta(\|\bm u_h\|_{\bm V_h}+\|p_h\|_{Q_h})$ with $C_\delta=\max\{1,\delta C_b+1\}$. Since $|\mathcal A_h(\cdot,\cdot)|\ge \mathrm{Re}\,\mathcal A_h(\cdot,\cdot)\ge0$ here, combining this with \eqref{eq:Ah-lower-final} and the inequality $\frac{a^2+b^2}{a+b}\ge\frac12(a+b)$ gives
\[
\frac{\left| \mathcal A_h\big((\bm u_h,p_h),(\bm v_h,q_h)\big) \right|}{\|\bm v_h\|_{\bm V_h}+\|q_h\|_{Q_h}} \ge \frac{c}{2C_\delta} \big( \|\bm u_h\|_{\bm V_h}+\|p_h\|_{Q_h} \big).
\]
Taking the supremum over $(\bm v_h,q_h)\ne(\bm 0,0)$ gives \eqref{eq:uniform-infsup-Ah} with $C=2C_\delta/c$.
\end{proof}

%% file: section5.tex
\section{Error analysis}\label{sec:error-analysis}
In this section, we derive the a priori error estimates for the proposed method. We first establish approximation estimates in the mesh-dependent norms, then analyze
the error for the associated source problem, and finally apply the
Babu\v{s}ka--Osborn spectral approximation theory to obtain the convergence of
eigenvalues and eigenspaces. Throughout this section, the quasi-momentum
$\bm{k}\in\mathcal{B}\setminus\{\bm{0}\}$ is fixed. 

To measure the error between the exact and discrete solutions in a common framework, we extend them to the larger spaces $\bm{V} + \bm{V}^{\mathrm{ext}}_{h,\mathrm{per}}$ and $P+P^{\mathrm{ext}}_{h,\mathrm{per}}$. For $\bm{v}\in \bm{V} + \bm{V}^{\mathrm{ext}}_{h,\mathrm{per}}$ set
\begin{align}\label{eq:Vstar}
  \|\bm{v}\|^2_{\bm{V}^*}
  := \sum_{m=\pm}\Bigl(\bigl\|\varepsilon_m^{-1/2}\nabla_{\bm{k}}\times\bm{v}^m\bigr\|^2_{0,\Omega^m}
       + \|\bm{v}^m\|^2_{0,\Omega^m}
     \Bigr)
   + \sum_{K\in\mathcal T_\Sigma}
     \frac{\gamma}{h}
     \bigl\|[\![\bm{v}\times\mathbf{n}]\!]\bigr\|^2_{0,\Sigma_K}\\
   + h\sum_{K\in\mathcal T_\Sigma}
     \bigl\|
       \bigl\{\varepsilon^{-1}\nabla_{\bm{k}}\times\bm{v}\bigr\}_\theta
     \bigr\|^2_{0,\Sigma_K}
   + h\sum_{K\in\mathcal T_\Sigma}
     \bigl\|\{\bm v\cdot\bm n\}_\theta\bigr\|^2_{0,\Sigma_K}.
\end{align}
For $q\in P + P^{\mathrm{ext}}_{h,\mathrm{per}}$ set
\begin{equation}\label{eq:Qstar}
  \|q\|^2_{Q^*}
  := \sum_{m=\pm}\|\nabla_{\bm{k}} q^m\|^2_{0,\Omega^m}
   + \sum_{K\in\mathcal T_\Sigma}
     \frac{1}{h}\bigl\|[\![q]\!]\bigr\|^2_{0,\Sigma_K}
     .
\end{equation}
The extended product norm on
$(\bm{V}+\bm{V}^{\mathrm{ext}}_{h,\mathrm{per}})\times(P+P^{\mathrm{ext}}_{h,\mathrm{per}})$ is
\begin{equation}\label{eq:star_product}
  \|(\bm{v},q)\|^2_*
  := \|\bm{v}\|^2_{\bm{V}^*} + \|q\|^2_{Q^*}
    .
\end{equation}
A similar argument shows that the continuity of \(\mathcal{B}_h\) extends to these larger spaces: for all $(\mathbf{u},p)\in (\bm{V}+\bm{V}^{\mathrm{ext}}_{h,\mathrm{per}})\times (P+P^{\mathrm{ext}}_{h,\mathrm{per}})$ and $(\mathbf{v}_h,q_h)\in \bm{V}^{\mathrm{ext}}_{h,\mathrm{per}}\times P^{\mathrm{ext}}_{h,\mathrm{per}}$,
\begin{equation}\label{eq:cont_ext}
  \bigl|\mathcal{B}_h\!\bigl((\mathbf{u},p),(\mathbf{v}_h,q_h)\bigr)\bigr|
  \leq C_{\mathrm{cont}}\,\|(\mathbf{u},p)\|_*\,\|(\mathbf{v}_h,q_h)\|_h,
\end{equation}
where $C_{\mathrm{cont}}>0$ is independent of $h$ and the interface position.

For $m = \pm$, recall the bounded extension operators 
introduced in \cite{wang2022convergence},
$\mathcal{E}^m:H^r(\Omega^m)\to H^r(\Omega)$ and
$\bm{\mathcal{E}}^m:\bm{H}^r(\Omega^m)\to \bm{H}^r(\Omega)$, satisfying
\begin{equation}\label{eq:ext_bound}
  \|\mathcal{E}^m q^m\|_{r,\Omega} \leq C\|q^m\|_{r,\Omega^m},
  \qquad
  \|\bm{\mathcal{E}}^m \bm{v}^m\|_{r,\Omega} \leq C\|\bm{v}^m\|_{r,\Omega^m},
\end{equation}
with $C$ independent of $h$ and of the position of the interface relative to the mesh.
We define the unfitted interpolation operators
\begin{align}\label{eq:interp_ops}
  \Pi_h^*\bm{u}:=\bigl(
      (\Pi_h^{\mathrm{curl}} \bm{\mathcal{E}}^+\bm{u}^+)|_{\Omega_h^+},\;
      (\Pi_h^{\mathrm{curl}} \bm{\mathcal{E}}^-\bm{u}^-)|_{\Omega_h^-}
    \bigr),\\
  \Pi_h^* p:=\bigl(
      (\Pi_h^{\mathrm{grad}}\mathcal{E}^+ p^+)|_{\Omega_h^+},\;
      (\Pi_h^{\mathrm{grad}}\mathcal{E}^- p^-)|_{\Omega_h^-}
    \bigr),
\end{align}
where $\Pi_h^{\mathrm{curl}}$ and $\Pi_h^{\mathrm{grad}}$ are the modified interpolation operators of~\eqref{eq:bloch_interp}, so that $\Pi_h^*\bm{u}\in \bm{V}^{\mathrm{ext}}_{h,\mathrm{per}}$
and $\Pi_h^* p\in P^{\mathrm{ext}}_{h,\mathrm{per}}$.

\subsection{Error estimate for the associated source problem}
We assume the interface $\Sigma$ is sufficiently smooth so that $\Omega^\pm$ are Lipschitz domains, and that the exact solution has the piecewise regularity $\bm u \in \bm H^{2}(\Omega^+ \cup \Omega^-)$, $p \in H^{3}(\Omega^+ \cup \Omega^-)$, consistent with the piecewise $P^2$ multiplier space used in our implementation.
\begin{lemma}\label{lem:interp-u-p}
Assume $\bm u\in \bm H^2(\Omega^+\cup\Omega^-)$ and $p\in H^3(\Omega^+\cup\Omega^-)$. Then
\begin{align}
\|\bm u-\bm\Pi_h^*\bm u\|_{\bm V^*}
&\le Ch \sum_{m=\pm} \varepsilon_m^{-1/2} \left\| \bm u^m \right\|_{2,\Omega^m}
\label{eq:interp-u-Vstar} \\
\|p-\Pi_h^*p\|_{Q^*}
&\le Ch^2 \sum_{m=\pm} \|p^m\|_{3,\Omega^m} \label{eq:interp-p-P2} \\
\|(\bm u-\bm\Pi_h^*\bm u,p-\Pi_h^*p)\|_*
&\le C \sum_{m=\pm}
\left( \varepsilon_m^{-1/2} h \|\bm u^m\|_{2,\Omega^m} + h^2\|p^m\|_{3,\Omega^m}
\right)
\label{eq:interp-product}.
\end{align}
\end{lemma}
\begin{proof}
This follows from the standard interpolation estimates for N\'ed\'elec and Lagrange finite elements, with the trace terms in the norms controlled by the local trace inequality on cut elements; see Monk~\cite{monk2003finite}, Brenner--Scott~\cite{brenner2008mathematical}, and Hansbo~\cite{hansbo2002unfitted}.
\end{proof}

Rewrite \eqref{eq:mathcal_Ah} from Section~\ref{chap:chapter3} as
\begin{equation}\label{eq:contiA}
\mathcal A_h[(\bm u,p),(\bm v,q)] = \mathcal B_h[(\bm u,p),(\bm v,q)] + \mathcal C_h[(\bm u,p),(\bm v,q)]\, ,
\end{equation} where
\begin{align}\label{eq:contiB}
\mathcal B_h[(\bm u,p),(\bm v,q)] &= a_h(\bm u,\bm v)+b_h(p,\bm v)+ \overline{b_h(q,\bm u)},\\
\mathcal C_h[(\bm u,p),(\bm v,q)] &= g_u(\bm u,\bm v)-S_h(p,q).\label{eq:contiC}
\end{align}
Define the continuous solution operators $(T\bm f,T_p\bm f)=(\bm u,p) \in \bm V\times P$ by
\begin{equation}\label{eq:continuous-source}
\mathcal B_h[(\bm u,p),(\bm v,q)] = (\bm f,\bm v) \qquad \forall(\bm v,q)\in \bm V\times P.
\end{equation}
Following the argument used in \cite[Lemma~2]{boffi2006modified}, $T$ is compact and self-adjoint on $\bm L_{\mathrm{per}}^2(\Omega)$. Similarly, define the discrete solution operators $(T_h\bm f,T_{p,h}\bm f)=(\bm u_h,p_h) \in \bm V_{h,\mathrm{per}}^{\mathrm{ext}} \times P_{h,\mathrm{per}}^{\mathrm{ext}}$ by
\begin{equation}\label{eq:discrete-source}
\mathcal B_h[(\bm u_h,p_h),(\bm v_h,q_h)] + \mathcal C_h[(\bm u_h,p_h),(\bm v_h,q_h)] = (\bm f,\bm v_h), \,
\forall (\bm v_h,q_h)\in \bm V_{h,\mathrm{per}}^{\mathrm{ext}} \times P_{h,\mathrm{per}}^{\mathrm{ext}}.
\end{equation}
Subtracting the discrete problem from this continuous relation gives the following weak Galerkin orthogonality.
\begin{lemma}[Weak Galerkin orthogonality]
\label{lem:weak-galerkin}
Let $(\bm u,p)$ and $(\bm u_h,p_h)$ be the continuous and discrete source
solutions, respectively. Then
for all $(\bm v_h,q_h)\in \bm V_{h,\mathrm{per}}^{\mathrm{ext}}\times P_{h,\mathrm{per}}^{\mathrm{ext}}$,
\begin{equation}\label{eq:weak-galclean}
\mathcal B_h[(\bm u-\bm u_h,p-p_h),(\bm v_h,q_h)]=g_u(\bm u_h,\bm v_h)-S_h(p_h,q_h).
\end{equation}
\end{lemma}

The stabilization terms satisfy the following weak consistency estimate.
\begin{lemma}
\label{lem:weak-consistency-stab}
Let $(\bm u,p)$ satisfy~\eqref{eqmainc}, with
$\bm u \in \bm H^{2}(\Omega^+ \cup \Omega^-)$ and $p \in H^{3}(\Omega^+ \cup \Omega^-)$,
and let $(\Pi_h^*\bm u,\Pi_h^*p)$ $\in
\bm V_{h,\mathrm{per}}^{\mathrm{ext}}\times P_{h,\mathrm{per}}^{\mathrm{ext}}$
be the unfitted interpolation of $(\bm u,p)$ defined in~\eqref{eq:interp_ops}.
Then for any
$(\bm v_h,q_h)\in \bm V_{h,\mathrm{per}}^{\mathrm{ext}}\times P_{h,\mathrm{per}}^{\mathrm{ext}}$,
\begin{align}\label{eq:weak-consistency-Ch}
&\left| \mathcal C_h[(\Pi_h^*\bm u,\Pi_h^*p),(\bm v_h,q_h)] \right| 
\le C \left[ h\sum_{m=\pm}\varepsilon_m^{-1/2}\|\bm u^m\|_{2,\Omega^m}
+h^2\sum_{m=\pm}\|p^m\|_{3,\Omega^m} \right]
\|(\bm v_h,q_h)\|_h,
\end{align}
where $C>0$ is independent of $h$ and of the interface position.
\end{lemma}
\begin{proof}
Since $\mathcal C_h=g_u-g_p-J_h$, we estimate the three terms separately.
For $m=\pm$, set 
$\bm\eta_u^m:=(\Pi_h^*\bm u)^m-\bm{\mathcal E}^m\bm u^m $, $\eta_p^m:=(\Pi_h^*p)^m-\mathcal E^m p^m$ on $\Omega_h^m$.
The extensions belong to $\bm H^2(\Omega)$ and $H^3(\Omega)$,
respectively, hence for every $F\in\mathcal F_h^{g,m}$:
\[
\begin{aligned}
[\![\nabla_{\bm k}\times\bm{\mathcal E}^m\bm u^m]\!]_F&=0,\\
[\![\partial_{n_F}^j(\nabla_{\bm k}\mathcal E^m p^m)]\!]_F&=0,
\qquad j=0,1.
\end{aligned}
\]
Thus the jumps in the ghost penalties equal those of the interpolation
errors. For each element $K$ sharing $F$, the local interpolation
estimates underlying Lemma~\ref{lem:interp-u-p} give
\begin{align*}
\|\bm\eta_u^m\|_{\ell,K}
&\le C h_F^{2-\ell}\|\bm{\mathcal E}^m\bm u^m\|_{2,K},
&&\ell=1,2,\\
\|\eta_p^m\|_{\ell,K}
&\le C h_F^{3-\ell}\|\mathcal E^m p^m\|_{3,K},
&&\ell=1,2,3.
\end{align*}
Applying the trace inequality $
\|z\|_{0,F}^2\le C\bigl(h_F^{-1}\|z\|_{0,K}^2
+h_F\|\nabla z\|_{0,K}^2\bigr)$, $z\in H^1(K)$,
componentwise to the shifted curl of \(\bm\eta_u^m\) yields
\begin{align*}
h_F\bigl\|[\![\nabla_{\bm k}\times(\Pi_h^*\bm u)^m]\!]_F\bigr\|_{0,F}^2
=h_F\bigl\|[\![\nabla_{\bm k}\times\bm\eta_u^m]\!]_F\bigr\|_{0,F}^2
\le C\sum_{\substack{K\in\mathcal T_h\\F\subset\partial K}}
\left(\|\bm\eta_u^m\|_{1,K}^2+h_F^2\|\bm\eta_u^m\|_{2,K}^2\right)
\le C h_F^2\sum_{\substack{K\in\mathcal T_h\\F\subset\partial K}}
\|\bm{\mathcal E}^m\bm u^m\|_{2,K}^2.
\end{align*}
Similarly, applying the same trace inequality to
$\partial_{n_F}^j(\nabla_{\bm k}\eta_p^m)$ gives, for $j=0,1$,
\begin{align*}
&h_F^{2j+1}\bigl\|[\![\partial_{n_F}^j
(\nabla_{\bm k}(\Pi_h^*p)^m)]\!]_F\bigr\|_{0,F}^2
=h_F^{2j+1}\bigl\|[\![\partial_{n_F}^j
(\nabla_{\bm k}\eta_p^m)]\!]_F\bigr\|_{0,F}^2\\
&\quad\le C\sum_{\substack{K\in\mathcal T_h\\F\subset\partial K}}
\left(h_F^{2j}\|\eta_p^m\|_{j+1,K}^2
+h_F^{2j+2}\|\eta_p^m\|_{j+2,K}^2\right)
\le C h_F^4\sum_{\substack{K\in\mathcal T_h\\F\subset\partial K}}
\|\mathcal E^m p^m\|_{3,K}^2.
\end{align*}
Summing these estimates as in~\eqref{eq:ghost-u}--\eqref{eq:ghost-p},
using bounded element overlap and~\eqref{eq:ext_bound}, and applying
Cauchy--Schwarz with the discrete norm~\eqref{eq:product-norm} gives
\begin{align}
|g_u(\Pi_h^*\bm u,\bm v_h)|
\le g_u(\Pi_h^*\bm u,\Pi_h^*\bm u)^{1/2}
g_u(\bm v_h,\bm v_h)^{1/2}
\le C h\sum_{m=\pm}\varepsilon_m^{-1/2}
\|\bm u^m\|_{2,\Omega^m}\|(\bm v_h,q_h)\|_h,
\label{eq:gu-weak-consistency}\\
|g_p(\Pi_h^*p,q_h)|
\le g_p(\Pi_h^*p,\Pi_h^*p)^{1/2}g_p(q_h,q_h)^{1/2}
\le C h^2\sum_{m=\pm}\|p^m\|_{3,\Omega^m}
\|(\bm v_h,q_h)\|_h.
\label{eq:gp-weak-consistency}
\end{align}
Here $C$ may depend on the fixed material coefficients.
For the interface penalty, $[\![p]\!]=0$ on $\Sigma$ by~\eqref{eqmainc},
and hence $[\![\Pi_h^*p]\!]=[\![\Pi_h^*p-p]\!]$.
Using~\eqref{eq:interp-p-P2}, we obtain
\[
J_h(\Pi_h^*p,\Pi_h^*p)^{1/2}
\le\gamma_q^{1/2}\|\Pi_h^*p-p\|_{Q^*}
\le C h^2\sum_{m=\pm}\|p^m\|_{3,\Omega^m}.
\]
Cauchy--Schwarz and~\eqref{eq:product-norm} therefore give
\begin{equation}\label{eq:Jh-weak-consistency}
|J_h(\Pi_h^*p,q_h)|
\le C h^2\sum_{m=\pm}\|p^m\|_{3,\Omega^m}
\|(\bm v_h,q_h)\|_h.
\end{equation}
Combining~\eqref{eq:gu-weak-consistency}, \eqref{eq:gp-weak-consistency},
and~\eqref{eq:Jh-weak-consistency} proves the claim.
\end{proof}

\begin{remark}
For $p\in H^2(\Omega^+\cup\Omega^-)$,
\eqref{eq:weak-consistency-Ch} holds with
$h^2\sum_{m=\pm}\|p^m\|_{3,\Omega^m}$ replaced by
$h\sum_{m=\pm}\|p^m\|_{2,\Omega^m}$.
The $j=1$ term in $g_p$ is estimated directly 
without requiring  face traces of the exact second derivatives.
\end{remark}

Based on Lemma~\ref{lem:weak-consistency-stab}, we obtain the energy error estimate, one of the main results of this section.
\begin{theorem}\label{thm:source-error}
Let $(\bm u,p)$ solve the continuous source problem and $(\bm u_h,p_h)\in \bm V^{\rm ext}_{h,\mathrm{per}}\times P^{\rm ext}_{h,\mathrm{per}}$ the discrete source problem. Under the assumptions of Theorem~\ref{thm:uni-stability} and Lemmas~\ref{lem:interp-u-p} and~\ref{lem:weak-consistency-stab},
\begin{equation}\label{eq:source-error-final}
\|(\bm u-\bm u_h,p-p_h)\|_* \le C \left[ h\sum_{m=\pm} \varepsilon_m^{-1/2}\|\bm u^m\|_{2,\Omega^m} + h^2\sum_{m=\pm} \|p^m\|_{3,\Omega^m} \right].
\end{equation}
\end{theorem}
\begin{proof}
Write $\bm\eta_u:=\Pi_h^*\bm u-\bm u$, $\eta_p:=\Pi_h^*p-p$, $\bm\xi_u:=\Pi_h^*\bm u-\bm u_h$, $\xi_p:=\Pi_h^*p-p_h$. By the triangle inequality and the discrete trace estimates,
\begin{equation}\label{ineq:triangle}
\|(\bm u-\bm u_h,p-p_h)\|_* \le \|(\bm\eta_u,\eta_p)\|_* +C\|(\bm\xi_u,\xi_p)\|_h,
\end{equation}
and Lemma~\ref{lem:interp-u-p} gives directly
\begin{equation}\label{eq:1st_term}
\|(\bm\eta_u,\eta_p)\|_* \le C \left[ h\sum_{m=\pm} \varepsilon_m^{-1/2}\|\bm u^m\|_{2,\Omega^m} + h^2\sum_{m=\pm} \|p^m\|_{3,\Omega^m} \right].
\end{equation}
It remains to bound $\|(\bm\xi_u,\xi_p)\|_h$. Using bilinearity of $\mathcal B_h$ and weak Galerkin orthogonality \eqref{eq:weak-galclean},
\begin{align*}
\mathcal B_h[(\bm\xi_u,\xi_p),(\bm v_h,q_h)]
=\mathcal B_h[(\bm\eta_u,\eta_p),(\bm v_h,q_h)] + g_u(\bm u_h,\bm v_h)- S_h(p_h,q_h) \\
=\mathcal B_h[(\bm\eta_u,\eta_p),(\bm v_h,q_h)]+g_u(\Pi_h^*\bm u,\bm v_h)-S_h(\Pi_h^*p,q_h)- g_u(\bm\xi_u,\bm v_h)+S_h(\xi_p,q_h), 
\end{align*}
which rearranges to
\begin{equation}\label{eq:xi-error-equation-correct}
\mathcal A_h[(\bm\xi_u,\xi_p),(\bm v_h,q_h)]
= \mathcal B_h[(\bm\eta_u,\eta_p),(\bm v_h,q_h)] + \mathcal C_h[(\Pi_h^*\bm u,\Pi_h^*p),(\bm v_h,q_h)].
\end{equation}
By Theorem~\ref{thm:uni-stability} applied to $(\bm\xi_u,\xi_p)$, together with \eqref{eq:xi-error-equation-correct}, the continuity estimate \eqref{eq:cont_ext}, and the weak consistency estimate of Lemma~\ref{lem:weak-consistency-stab},
\begin{align}\label{eq:2st_term}
\|(\bm\xi_u,\xi_p)\|_h
&\le C \sup_{(\bm v_h,q_h)\neq(0,0)} \frac{\left| \mathcal B_h[(\bm\eta_u,\eta_p),(\bm v_h,q_h)]+ \mathcal C_h[(\Pi_h^*\bm u,\Pi_h^*p),(\bm v_h,q_h)] \right|}{\|(\bm v_h,q_h)\|_h } \nonumber \\
&\le C \left( \|(\bm\eta_u,\eta_p)\|_* + h\sum_{m=\pm} \varepsilon_m^{-1/2}\|\bm u^m\|_{2,\Omega^m} + h^2\sum_{m=\pm} \|p^m\|_{3,\Omega^m} \right).
\end{align}
Combining \eqref{ineq:triangle}, \eqref{eq:1st_term}, and \eqref{eq:2st_term} gives the desired estimate.
\end{proof}

Theorem~\ref{thm:source-error} gives the energy norm estimate. We next derive the $L^2$ estimate.
\begin{theorem}\label{thm:source-L2-error}
Define the dual problem: for $\bm g\in \bm L^2(\Omega)$, let $(\bm z,\tau)$ satisfy
\begin{equation}\label{eq:adjoint-problem-L2}
\mathcal B_h[(\bm v,q),(\bm z,\tau)] = (\bm v,\bm g)_\Omega \qquad \forall (\bm v,q)\in \bm V\times P.
\end{equation}
Assume the regularity
\begin{equation}\label{eq:adjoint-regularity-L2}
\sum_{m=\pm} \left( \varepsilon_m^{-1/2}\|\bm z^m\|_{2,\Omega^m} + \|\tau^m\|_{3,\Omega^m} \right) \le C\|\bm g\|_{0,\Omega},
\end{equation}
and the stability bound
\begin{equation}\label{eq:dual-stab-smallness}
\left| \mathcal C_h[(\bm v_h,q_h),(\bm\Pi_h^*\bm z,\Pi_h^*\tau)] \right| \le C h \|(\bm v_h,q_h)\|_h \|\bm g\|_{0,\Omega}
\qquad \forall (\bm v_h,q_h)\in \bm V^{\rm ext}_{h,\mathrm{per}}\times P^{\rm ext}_{h,\mathrm{per}}.
\end{equation}
Then
\begin{equation}\label{eq:L2-by-energy-final}
\begin{aligned}
\|\bm u-\bm u_h\|_{0,\Omega}
\le C h\Bigl[
\|(\bm u-\bm u_h,p-p_h)\|_*
+h\sum_{m=\pm}\varepsilon_m^{-1/2}\|\bm u^m\|_{2,\Omega^m}
+h^2\sum_{m=\pm}\|p^m\|_{3,\Omega^m}\Bigr].
\end{aligned}
\end{equation}
In particular, for the second-family N\'ed\'elec element of degree one and a piecewise $P_2$ multiplier space,
\begin{equation}\label{eq:L2-source-error-P2}
\|\bm u-\bm u_h\|_{0,\Omega} \le C h^2 \sum_{m=\pm} \left( \varepsilon_m^{-1/2}\|\bm u^m\|_{2,\Omega^m} + \|p^m\|_{3,\Omega^m} \right).
\end{equation}
\end{theorem}
\begin{proof}
We employ the Aubin--Nitsche duality argument \cite[Theorem~3.2.4]{ciarlet1981two}, and account for the ghost penalty term as in \cite[Theorem~2.10]{gurkan2019stabilized}. Let $\bm e_u:=\bm u-\bm u_h$, $e_p:=p-p_h$, and let $(\bm z,\tau)$ solve \eqref{eq:adjoint-problem-L2}. Taking the test function
$(\bm v,q)=(\bm e_u,e_p)$ in \eqref{eq:adjoint-problem-L2} gives
\begin{equation}\label{eq:L2-dual-split-correct}
\begin{aligned}
(\bm e_u,\bm g)_\Omega
= \mathcal B_h[(\bm e_u,e_p),(\bm z,\tau)]  
=\mathcal B_h[(\bm e_u,e_p),(\bm z-\bm\Pi_h^*\bm z,\tau-\Pi_h^*\tau)] +\mathcal B_h[(\bm e_u,e_p),(\bm\Pi_h^*\bm z,\Pi_h^*\tau)] .
\end{aligned}
\end{equation}
For the first term, continuity \eqref{eq:cont_ext}, Lemma~\ref{lem:interp-u-p} applied to $(\bm z,\tau)$, and the regularity \eqref{eq:adjoint-regularity-L2} give $\|(\bm z-\bm\Pi_h^*\bm z,\tau-\Pi_h^*\tau)\|_* \le C h \|\bm g\|_{0,\Omega}$, hence
\begin{equation}\label{eq:L2-first-final-correct}
\left| \mathcal B_h[(\bm e_u,e_p),(\bm z-\bm\Pi_h^*\bm z,\tau-\Pi_h^*\tau)] \right| \le C h \|(\bm e_u,e_p)\|_* \|\bm g\|_{0,\Omega}.
\end{equation}
For the second term, since $(\bm\Pi_h^*\bm z,\Pi_h^*\tau)$ lies in the discrete test space, Lemma~\ref{lem:weak-galerkin} (weak Galerkin orthogonality) gives
\begin{equation}\label{eq:weak-gal-L2-use}
\mathcal B_h[(\bm e_u,e_p),(\bm\Pi_h^*\bm z,\Pi_h^*\tau)] = g_u(\bm u_h,\bm\Pi_h^*\bm z)-S_h(p_h,\Pi_h^*\tau) =\mathcal C_h[(\bm u_h,p_h),(\bm\Pi_h^*\bm z,\Pi_h^*\tau)].
\end{equation}
Writing $(\bm u_h,p_h) =(\bm \Pi_h^*\bm u,\Pi_h^*p) -(\bm \Pi_h^*\bm u-\bm u_h,\Pi_h^*p-p_h)$ and splitting $\mathcal C_h$ accordingly, we use Lemma~\ref{lem:weak-consistency-stab} together with \eqref{eq:dual-stab-smallness} to bound the first part by
\begin{align}\label{eq:C-piu-bound-L2}
\left| \mathcal C_h[(\bm \Pi_h^*\bm u,\Pi_h^*p),
(\bm\Pi_h^*\bm z,\Pi_h^*\tau)] \right| 
&\le C h \left[ h\sum_{m=\pm} \varepsilon_m^{-1/2}\|\bm u^m\|_{2,\Omega^m}
+ h^2\sum_{m=\pm} \|p^m\|_{3,\Omega^m} \right]
\|\bm g\|_{0,\Omega}.
\end{align}
\eqref{eq:dual-stab-smallness} together with Theorem~\ref{thm:source-error} bounds the second part by
\begin{equation}
\begin{aligned}\label{eq:C-xi-bound-L2}
&\left| \mathcal C_h[ (\bm \Pi_h^*\bm u-\bm u_h,\Pi_h^*p-p_h), (\bm\Pi_h^*\bm z,\Pi_h^*\tau)] \right|
\le C h \|(\bm \Pi_h^*\bm u-\bm u_h,\Pi_h^*p-p_h)\|_h \|\bm g\|_{0,\Omega} \nonumber\\
&\le C h \left[ h\sum_{m=\pm} \varepsilon_m^{-1/2}\|\bm u^m\|_{2,\Omega^m} + h^2\sum_{m=\pm} \|p^m\|_{3,\Omega^m} \right] \|\bm g\|_{0,\Omega}.
\end{aligned}    
\end{equation}
Consequently,
\begin{equation}\label{eq:L2-second-term-total}
\left|
\mathcal B_h[(\bm e_u,e_p),(\bm\Pi_h^*\bm z,\Pi_h^*\tau)]
\right| \le C h \left[ h\sum_{m=\pm}
\varepsilon_m^{-1/2}\|\bm u^m\|_{2,\Omega^m} + h^2\sum_{m=\pm} \|p^m\|_{3,\Omega^m} \right]
\|\bm g\|_{0,\Omega}.
\end{equation}
Combining \eqref{eq:L2-second-term-total} with \eqref{eq:L2-first-final-correct} gives
\[
\begin{aligned}
|(\bm e_u,\bm g)_\Omega|
&\le C h\Bigl[\|(\bm e_u,e_p)\|_*
+h\sum_{m=\pm}\varepsilon_m^{-1/2}\|\bm u^m\|_{2,\Omega^m}
+h^2\sum_{m=\pm}\|p^m\|_{3,\Omega^m}\Bigr]
\|\bm g\|_{0,\Omega}.
\end{aligned}
\]
Taking the supremum over $\bm g$ with $\|\bm g\|_{0,\Omega}=1$ gives
\eqref{eq:L2-by-energy-final}. Applying~\eqref{eq:source-error-final}
and using $h\le1$ then yields~\eqref{eq:L2-source-error-P2}.
\end{proof}

\subsection{Convergence of the eigenvalue problem}
\label{subsec:eigen-conver}
This subsection applies the Babu\v{s}ka--Osborn spectral approximation
theory \cite{babuska1991eigenvalue,osborn1975spectral}
to derive the convergence of the discrete eigenpairs.

Recall the continuous solution operator 
$T:\bm L^2_{\mathrm{per}}(\Omega)\to \bm L^2_{\mathrm{per}}(\Omega)$ 
and the discrete solution operator 
$T_h:\bm L^2_{\mathrm{per}}(\Omega)\to \bm V^{\mathrm{ext}}_{h,\mathrm{per}}$ 
defined by \eqref{eq:continuous-source} and \eqref{eq:discrete-source},
respectively.

Theorems~\ref{thm:source-error} and~\ref{thm:source-L2-error} imply there is $C>0$, independent of $h$ and the interface position, with $\|(T-T_h)\bm f\|_{0,\Omega} \le C h^2 \|\bm f\|_{0,\Omega}$ for all $\bm f\in \bm L^2_{\mathrm{per}}(\Omega)$, so
\begin{equation}\label{eq:uniform-convergence}
\|T-T_h\|_{\mathcal L(\bm L^2_{\mathrm{per}}(\Omega))} \le C h^2 \to 0 \qquad \text{as } h\to 0 .
\end{equation}
Since $T_h\to T$ in operator norm, $\{T_h\}$ is collectively compact~\cite{osborn1975spectral}, and the Babu\v{s}ka--Osborn theory applies.

Let $\mu$ be a nonzero eigenvalue of $T$ with algebraic multiplicity $\ell$, and $\Gamma\subset\mathbb{C}$ a circle centered at $\mu$, lying in $\rho(T)$, enclosing no other point of $\sigma(T)$. For $h$ sufficiently small, define the spectral projections $E = \frac{1}{2\pi i}\int_{\Gamma} R_z(T)\,dz$, $E_h = \frac{1}{2\pi i}\int_{\Gamma} R_z(T_h)\,dz$, with $R_z(T)=(z-T)^{-1}$ and $R_z(T_h)=(z-T_h)^{-1}$. The eigenvalues $\mu_1(h),\ldots,\mu_\ell(h)$ of $T_h$ inside $\Gamma$ converge to $\mu$, and $\dim R(E_h)=\dim R(E)=\ell$. 

We first recall the following abstract result, which is the key tool for both the eigenspace and eigenvalue estimates.
\begin{theorem}[Osborn, {\cite[Theorems 1--3]{osborn1975spectral}}]
\label{thm:osborn-abstract}
Let $\phi_1,\ldots,\phi_\ell$ be a basis for $R(E)$. There exist constants $C_1,C_3>0$, independent of $h$, such that for $h$ sufficiently small,
\begin{align}\label{eq:osborn-eigenspace}
\hat{\delta}\bigl(R(E),R(E_h)\bigr)
&\le C_1 \|(T-T_h)|_{R(E)}\|_{\mathcal{L}(\bm L^2_{\mathrm{per}}(\Omega))},
\\
\label{eq:osborn-eigenvalue}
\left|\mu-\frac{1}{\ell}\sum_{r=1}^{\ell}\mu_r(h)\right|
&\le \frac{1}{\ell}\sum_{i,j=1}^{\ell}
\left|\bigl((T-T_h)\phi_i,\phi_j\bigr)\right|
+ C_3
\|(T-T_h)|_{R(E)}\|_{\mathcal{L}(\bm L^2_{\mathrm{per}}(\Omega))}^{2}.
\end{align}
\end{theorem}

Combining Theorem~\ref{thm:osborn-abstract} with the operator error estimate \eqref{eq:uniform-convergence}, we obtain the main theoretical result of this section.
\begin{theorem}\label{thm:main-spectral}
Assume the hypotheses of Theorems~\ref{thm:source-error}
and~\ref{thm:source-L2-error}. Let
$\omega_i^2$ be an eigenvalue of the original problem with multiplicity $\ell_i$
and eigenspace $R(E_i)$. Then, for $h$ small enough, there are exactly $\ell_i$
discrete eigenvalues $\omega_{i1,h}^2,\ldots,\omega_{i\ell_i,h}^2$ converging to
$\omega_i^2$. Setting $\widehat{\omega}_{i,h}^2
=\frac{1}{\ell_i}\sum_{j=1}^{\ell_i}\omega_{ij,h}^2$ and letting $\widehat
R(E_{h,i})$ denote the direct sum of the corresponding discrete eigenspaces,
\begin{equation}\label{eq:eigenvalue-final}
\left|\omega_i^2-\widehat{\omega}_{i,h}^2\right|\le Ch^2, \qquad \hat\delta(R(E_i),\widehat R(E_{h,i}))\le Ch^2,
\end{equation}
with $C>0$ independent of $h$ and of the interface position.
\end{theorem}
\begin{proof}
For the eigenspace estimate, \eqref{eq:osborn-eigenspace} and \eqref{eq:uniform-convergence} give
\[
  \hat\delta(R(E_i),\widehat R(E_{h,i})) \le C_1\|(T-T_h)|_{R(E_i)}\|_{\mathcal{L}(\bm L_{per}^2)}
   = \sup_{\substack{\bm v\in R(E_i)\\ \|\bm v\|_{0,\Omega}=1}} \|(T-T_h)\bm v\|_{\mathcal{L}(\bm L_{per}^2)} \le Ch^2.
\]
For the eigenvalue estimate, since $\omega_i^2 = \mu_i^{-1}$ and $\mu_{ij,h} = \omega_{ij,h}^{-2}$, it suffices to work with $\mu_i,\mu_{ij,h}$, the second term in \eqref{eq:osborn-eigenvalue} is $O(h^4)$ by \eqref{eq:uniform-convergence}, so we only need to bound $|((T-T_h)\bm\phi_j,\bm\phi_k)|$.

Let $(\bm u_j, p_j) = (T\bm\phi_j, T_p\bm\phi_j)$ and $(\bm u_{j,h}, p_{j,h}) = (T_h\bm\phi_j, T_{p,h}\bm\phi_j)$, and similarly for index $k$. Since $T$ is self-adjoint and $(\bm u_k, p_k)$ solves  \eqref{eq:continuous-source} with $\bm\phi_k$, we have
\[
\begin{aligned}
   &\bigl((T-T_h)\bm\phi_j,\,\bm\phi_k\bigr)
  = \bigl(\bm\phi_j,\, (T-T_h)\bm\phi_k\bigr)
  = \bigl(\bm\phi_j,\, T\bm\phi_k\bigr)-\bigl(\bm\phi_j,\, T_h\bm\phi_k\bigr)\\
  =&\mathcal B_h[(T\bm\phi_j,T_p\bm\phi_j),(T\bm\phi_k,T_p\bm\phi_k)]-\mathcal B_h[(T_h\bm\phi_j,T_{p,h}\bm\phi_j),(T_h\bm\phi_k,T_{p,h}\bm\phi_k)] 
  - \mathcal C_h[(T_h\bm\phi_j,T_{p,h}\bm\phi_j),(T_h\bm\phi_k,T_{p,h}\bm\phi_k)] \\
  =&\mathcal B_h[(\bm u_j,p_j),(\bm u_k,p_k)]
  -\mathcal B_h[(\bm u_{j,h},p_{j,h}),(\bm u_{k,h},p_{k,h})]
  -\mathcal C_h[(\bm u_{j,h},p_{j,h}),(\bm u_{k,h},p_{k,h})] 
\end{aligned}
\]
By bilinearity of \(\mathcal B_h\), we have
\[
\begin{aligned}
&\mathcal B_h[(\bm u_j,p_j),(\bm u_k,p_k)]-\mathcal B_h[(\bm u_{j,h},p_{j,h}),(\bm u_{k,h},p_{k,h})]\\
&=\mathcal B_h[(\bm u_j-\bm u_{j,h},p_j-p_{j,h}),(\bm u_{k,h},p_{k,h})]
+\mathcal B_h[(\bm u_{j,h},p_{j,h}),(\bm u_k-\bm u_{k,h},p_k-p_{k,h})]\\
&+\mathcal B_h[(\bm u_j-\bm u_{j,h},p_j-p_{j,h}),
             (\bm u_k-\bm u_{k,h},p_k-p_{k,h})] .
\end{aligned}
\]
Taking $(\bm v_h,q_h)=(\bm u_{k,h},p_{k,h})$
in Lemma~\ref{lem:weak-galerkin}, we get
\[
\mathcal B_h[(\bm u_j-\bm u_{j,h},p_j-p_{j,h}),
(\bm u_{k,h},p_{k,h})]=\mathcal C_h[(\bm u_{j,h},p_{j,h}),(\bm u_{k,h},p_{k,h})].
\]
Using the perturbed weak Galerkin orthogonality for the index \(k\), together with the
Hermitian symmetry of \(\mathcal B_h\) and \(\mathcal C_h\), we obtain
\[
\begin{aligned}
&\mathcal B_h[(\bm u_{j,h},p_{j,h}),(\bm u_k-\bm u_{k,h},p_k-p_{k,h})]  
=\overline{ \mathcal B_h[ (\bm u_k-\bm u_{k,h},p_k-p_{k,h}), (\bm u_{j,h},p_{j,h})]} \\
&=\overline{ \mathcal C_h[ (\bm u_{k,h},p_{k,h}),(\bm u_{j,h},p_{j,h})]} 
=\mathcal C_h[(\bm u_{j,h},p_{j,h}),(\bm u_{k,h},p_{k,h})].
\end{aligned}
\]
Then we get 
\begin{align*}
\bigl((T-T_h)\bm\phi_j,\,\bm\phi_k\bigr)
=\mathcal C_h[(\bm u_{j,h},p_{j,h}),(\bm u_{k,h},p_{k,h})]
+\mathcal B_h[(\bm u_j-\bm u_{j,h},p_j-p_{j,h}),(\bm u_k-\bm u_{k,h},p_k-p_{k,h})].
\end{align*}
By the continuity of $\mathcal B_h$ and the error estimate in Theorem~\ref{thm:source-error},
\begin{align}\label{bh_error}
&\left| \mathcal B_h[(\bm u_j-\bm u_{j,h},p_j-p_{j,h}),(\bm u_k-\bm u_{k,h},p_k-p_{k,h})]\right| \\
&\qquad\le C\|(\bm u_j-\bm u_{j,h},p_j-p_{j,h})\|_*  \|(\bm u_k-\bm u_{k,h},p_k-p_{k,h})\|_*  \le Ch^2
.\nonumber 
\end{align}
By the Cauchy--Schwarz inequality associated with $\mathcal C_h$ and Lemma~\ref{lem:weak-consistency-stab},
\begin{equation}\label{ch_error}
\left| \mathcal C_h[(\bm u_{j,h},p_{j,h}),(\bm u_{k,h},p_{k,h})]
\right| \le |{\mathcal C_h}(\bm u_{j,h},p_{j,h})|\, |{\mathcal C_h}(\bm u_{k,h},p_{k,h})|
\le C h^2 .
\end{equation}
Combining \eqref{bh_error} and \eqref{ch_error}, and inserting them into \eqref{eq:osborn-eigenvalue}
yields $|\mu_i-\widehat\mu_{i,h}|\le Ch^2$, and translating back via
$\omega^2=\mu^{-1}$ gives $\bigl|\omega_i^2-\widehat\omega_{i,h}^2\bigr|\le Ch^2$.

\end{proof}

%% file: section6.tex
\section{Numerical Results}\label{sec:numerical-results}
We assess the convergence and accuracy of the proposed method through
eigenvalue convergence tests and band-structure comparisons.
\subsection{Convergence rates of eigenvalues}
\label{subsec:eig-convergence}
We validate the predicted second-order convergence of the discrete
eigenvalues. We first consider
a homogeneous medium, for which the exact eigenvalues are known, followed by
an inhomogeneous medium containing high-contrast inclusions with complex interfaces.

Let $\omega_{j,h_m}^{2}$ denote the $j$-th eigenvalue computed on a mesh
of size $h_m$, and let $\omega_{j,\mathrm{ref}}^{2}$ denote the corresponding
reference value. We define the absolute and relative eigenvalue errors by $E_j^{(m)}$ and $e_j^{(m)}$, respectively,
\begin{equation}
  \label{eq:eig-errors}
  E_j^{(m)}
  =
  \left|\omega_{j,h_m}^{2}-\omega_{j,\mathrm{ref}}^{2}\right|
  \qquad\text{and}\qquad
  e_j^{(m)}
  =
  \frac{
    \left|\omega_{j,h_m}^{2}-\omega_{j,\mathrm{ref}}^{2}\right|
  }{
    \left|\omega_{j,\mathrm{ref}}^{2}\right|
  }.
\end{equation}
For the homogeneous problems, $\omega_{j,\mathrm{ref}}^2=\omega_j^2$
is the exact eigenvalue. 
For the inhomogeneous problem, $\omega_{j,\mathrm{ref}}^2$ is the numerical eigenvalue computed on a very fine mesh.
The observed convergence order is
\begin{equation}
  \label{eq:rate}
  p_{j}^{(m)}
  =
  \frac{
    \log\left(\mathcal{E}_j^{(m)}/\mathcal{E}_j^{(m+1)}\right)
  }{
    \log\left(h_m/h_{m+1}\right)
  },
  \qquad m\ge 1,
\end{equation}
where $\mathcal{E}_j^{(m)}$ denotes either the absolute error, $E_j^{(m)}$, or relative error, $e_j^{(m)}$.

\subsubsection{Homogeneous medium}
\label{sssec:hom-conv}
For a homogeneous medium of air ($\varepsilon \equiv 1$), the eigenfunctions of~\eqref{eq:Model1} can be represented as plane waves, such that
\begin{equation}
  \label{eq:exact_eig}
  \mathbf{H}(\bm{x}) = \mathbf{C}_{\mathbf{I}}\,
    \mathrm{e}^{\mathrm{i}(\boldsymbol{k} + \mathbf{I}) \cdot \bm{x}},
\end{equation}
where $\mathbf{I} \in 2\pi\,\mathbb{Z}^{d}$ ($d=2,3$) is a reciprocal lattice
vector, and $\mathbb{Z}$ denotes the set of integers. The constant vector $\mathbf{C}_{\mathbf{I}}$ satisfies the divergence
constraint $(\boldsymbol{k} + \mathbf{I}) \cdot \mathbf{C}_{\mathbf{I}} = 0$.
The corresponding analytical eigenvalues~\cite[Section~9.1]{lu2017discontinuous} are
\begin{equation}\label{eq:exact_eigval}
  \omega_{\mathbf{I}}^{2}(\bm{k}) 
  = \lvert \boldsymbol{k} + \mathbf{I} \rvert^{2}.
\end{equation}
For each fixed $\bm{k}$, these eigenvalues are arranged in
nondecreasing order, counting multiplicities, and denoted by
$\omega_j^2(\bm{k})$, $j=1,2,\ldots$.
We suppress the dependence on $\bm{k}$ below.

\paragraph{Two-dimensional case.}
At refinement level $m$, we partition $\Omega = [0,1]^{2}$ into
$N_m^2$ uniform squares and subdivide each square into two triangles
along the same diagonal. The resulting mesh size is $h_m=\sqrt{2}/N_m$.
We fix $\bm{k} = (\pi, 0)$ as a representative example. 
By~\eqref{eq:exact_eigval}, for $\bm k=(\pi,0)$, the first nine exact eigenvalues, counting multiplicities, are
$\omega_1^2=\omega_2^2=\pi^2$,
$\omega_3^2=\omega_4^2=\omega_5^2=\omega_6^2=5\pi^2$,
$\omega_7^2=\omega_8^2=9\pi^2$,
$\omega_9^2=13\pi^2$.
Table~\ref{tab:eigconv_final} reports
the errors and observed convergence rates of the first nine eigenvalues over
four levels of uniform mesh refinement, using second-family N\'ed\'elec elements of degree one.
\begin{table}[htbp]
  \centering
  \caption{Convergence of the first nine eigenvalues of the 2D homogeneous
  case with $\boldsymbol{k} = (\pi, 0)$, using second-family N\'ed\'elec elements of degree one with the MUFEM formulation~\eqref{eq:final}.}
  \label{tab:eigconv_final}
  \renewcommand{\arraystretch}{1.1}
  \fontsize{10}{12}\selectfont
  \setlength{\tabcolsep}{5pt}
  \begin{tabular*}{\textwidth}{@{\extracolsep{\fill}}c|cc|cc|cc|cc@{}}
    \toprule
    \multirow{2}{*}{$\omega_j^{2}$}
      & \multicolumn{2}{c|}{$N_1 = 10$}
      & \multicolumn{2}{c|}{$N_2 = 20$}
      & \multicolumn{2}{c|}{$N_3 = 40$}
      & \multicolumn{2}{c}{$N_4 = 80$} \\
      & $E_j^{(1)}$ & Order
      & $E_j^{(2)}$ & Order
      & $E_j^{(3)}$ & Order
      & $E_j^{(4)}$ & Order \\
    \midrule
    $\omega^{2}_{1}=\pi^2$ & 0.0000 & ---   & 0.0000 & ---   & 0.0000 & ---   & 0.0000 & --- \\
    $\omega^{2}_{2}=\pi^2$ & 0.1617 & ---   & 0.0405 & 1.9963 & 0.0101 & 1.9988 & 0.0025 & 1.9997 \\
    $\omega^{2}_{3}=5\pi^2$ & 0.9129 & ---   & 0.2315 & 1.9791 & 0.0581 & 1.9946 & 0.0145 & 1.9987 \\
    $\omega^{2}_{4}=5\pi^2$ & 0.1266 & ---   & 0.0302 & 2.0651 & 0.0075 & 2.0173 & 0.0019 & 2.0044 \\
    $\omega^{2}_{5}=5\pi^2$ & 0.0792 & ---   & 0.0171 & 2.2112 & 0.0041 & 2.0566 & 0.0010 & 2.0144 \\
    $\omega^{2}_{6}=5\pi^2$ & 2.1052 & ---   & 0.5151 & 2.0310 & 0.1281 & 2.0078 & 0.0320 & 2.0020 \\
    $\omega^{2}_{7}=9\pi^2$ & 1.5941 & ---   & 0.3937 & 2.0174 & 0.0982 & 2.0039 & 0.0245 & 2.0009 \\
    $\omega^{2}_{8}=9\pi^2$ & 0.9288 & ---   & 0.2598 & 1.8379 & 0.0670 & 1.9562 & 0.0169 & 1.9889 \\
    $\omega^{2}_{9}=13\pi^2$ & 3.2420 & ---   & 0.8224 & 1.9790 & 0.2063 & 1.9948 & 0.0516 & 1.9987 \\
    \bottomrule
  \end{tabular*}
\end{table}
Except for the first eigenvalue, whose errors vanish at the reported
precision, the observed orders in Table~\ref{tab:eigconv_final}
approach two.

\paragraph{Three-dimensional case.}
At refinement level $m$, we partition the unit cube
$\Omega=[0,1]^3$ into $N_m^3$ uniform cubic cells and subdivide each cube
into six tetrahedra using the same vertex connectivity. The resulting
mesh size is $h_m=\sqrt{3}/N_m$.
For $\boldsymbol{k}=(\pi,0,0)$, the first two distinct analytical eigenvalues are
$\omega_1^2=\cdots=\omega_4^2=\pi^2$ and $\omega_5^2=\cdots=\omega_{20}^2=5\pi^2$.
Except for the first two modes, which reproduce $\pi^2$ up to machine precision, the numerical
eigenvalues exhibit approximately second-order convergence, as reported in Table~\ref{tab:eigconv_3d_hom}.
\begin{table}[htbp]
  \centering
  \caption{Convergence of the first nine eigenvalues for the 3D
  homogeneous case with $\boldsymbol{k}=(\pi,0,0)$, using second-family N\'ed\'elec elements of degree one with the MUFEM formulation~\eqref{eq:final}.}
  \label{tab:eigconv_3d_hom}
  \renewcommand{\arraystretch}{1.1}
  \fontsize{10}{12}\selectfont
  \setlength{\tabcolsep}{5pt}
  \begin{tabular*}{\textwidth}{@{\extracolsep{\fill}}c|cc|cc|cc|cc@{}}
    \toprule
    \multirow{2}{*}{$\omega_j^2$}
      & \multicolumn{2}{c|}{$N_1=8$} & \multicolumn{2}{c|}{$N_2=11$}
      & \multicolumn{2}{c|}{$N_3=14$}& \multicolumn{2}{c}{$N_4=17$} \\
      & $E_j^{(1)}$ & Order & $E_j^{(2)}$ & Order
      & $E_j^{(3)}$ & Order & $E_j^{(4)}$ & Order \\
    \midrule
    $\omega_1^2=\pi^2$ & 0.0000 & ---& 0.0000 & ---& 0.0000 & ---& 0.0000 & --- \\
    $\omega_2^2=\pi^2$ & 0.0000 & --- & 0.0000 & --- & 0.0000 & ---& 0.0000 & --- \\
    $\omega_3^2=\pi^2$ & 1.2798 & --- & 0.6769 & 2.0001 & 0.4172 & 2.0069  & 0.2832 & 1.9952 \\
    $\omega_4^2=\pi^2$ & 1.3363 & ---& 0.7063 & 2.0025& 0.4356 & 2.0041& 0.2958 & 1.9938 \\
    $\omega_5^2=5\pi^2$& 1.0047 & ---& 0.5393 & 1.9539& 0.3348 & 1.9767 & 0.2281 & 1.9758 \\
    $\omega_6^2=5\pi^2$& 1.0218 & ---& 0.5463 & 1.9663& 0.3379 & 1.9923& 0.2300 & 1.9811 \\
    $\omega_7^2=5\pi^2$& 1.0291 & ---& 0.5480 & 1.9788& 0.3386 & 1.9962& 0.2303 & 1.9862 \\
    $\omega_8^2=5\pi^2$& 1.0471 & ---& 0.5557 & 1.9896& 0.3426 & 2.0051& 0.2325 & 1.9980 \\
    $\omega_9^2=5\pi^2$& 1.2551 & ---& 0.6756 & 1.9450& 0.4197 & 1.9742& 0.2862 & 1.9717 \\
    \bottomrule
  \end{tabular*}
\end{table}

\subsubsection{High-contrast flower-shaped inclusion}
\label{sssec:flower-conv}
We consider a more challenging problem involving a nonconvex curved interface, defined by
\begin{equation}
  \label{eq:flower-levelset}
  \Omega^{-} \;=\; \bigl\{\bm x\in\Omega \;\bigl|\; \varphi_{\star}(\bm x) < 0\bigr\}
  \qquad\text{where}\qquad
  \varphi_{\star}(x,y) = r^{2} - r_{0}(\theta)^{2},
\end{equation}
such that $r = \sqrt{(x-0.5)^{2}+(y-0.5)^{2}}$, $\theta = \operatorname{atan2}(y-0.5,\,x-0.5)$ and $r_{0}(\theta) = \tfrac{1}{5} + \tfrac{s}{7}\sin(5\theta)$ with $s=0.5$, which we refer to as a "flower-shaped" inclusion (see Fig.~\ref{fig2dcell}). Further, we use a high contrast between the inclusion and the background matrix, such that
\begin{equation}
  \label{eq:flower-eps}
  \varepsilon(\bm x) =
  \begin{cases}
    \varepsilon^{-} = 12, & \bm x\in\Omega^{-},\\[2pt]
    \varepsilon^{+} = 1,   & \bm x\in\Omega^{+}.
  \end{cases}
\end{equation}
We use the same uniform background triangulation as in the two-dimensional
homogeneous case, with $h_m=\sqrt{2}/N_m$ and $N_m=5,10,20,40$.
Analytical eigenvalues are unavailable for this geometry, so we use a refined reference computation at $\bm{k}=(\pi,0)$. Table~\ref{tab:eigconv_nonhom_flower} reports the relative errors and observed convergence orders. The errors decrease consistently under mesh refinement, and the observed orders on the finest mesh support an
approximately second-order convergence trend.

\begin{table}[htbp]
  \centering
  \caption{Relative errors of the first nine eigenvalues for the 2D
  inhomogeneous flower-shaped inclusion with $\boldsymbol{k}=(\pi,0)$, computed with the MUFEM formulation~\eqref{eq:final}.
  }
  \label{tab:eigconv_nonhom_flower}
  \renewcommand{\arraystretch}{1.1}
  \fontsize{10}{12}\selectfont
  \setlength{\tabcolsep}{5pt}
  \begin{tabular*}{\textwidth}{@{\extracolsep{\fill}}c|cc|cc|cc|cc@{}}
    \toprule
    \multirow{2}{*}{$\omega_j^{2}$}
      & \multicolumn{2}{c|}{$N_1 = 5$}
      & \multicolumn{2}{c|}{$N_2 = 10$}
      & \multicolumn{2}{c|}{$N_3 = 20$}
      & \multicolumn{2}{c}{$N_4 = 40$} \\
      & $e_j^{(1)}$ & Order
      & $e_j^{(2)}$ & Order
      & $e_j^{(3)}$ & Order
      & $e_j^{(4)}$ & Order \\
    \midrule
    $\omega^{2}_{1}$ & 0.0805 & ---    & 0.0137 & 2.5542 & 0.0051 & 1.4334 & 0.0013 & 1.9162 \\
    $\omega^{2}_{2}$ & 0.0991 & ---    & 0.0219 & 2.1813 & 0.0062 & 1.8166 & 0.0017 & 1.9087 \\
    $\omega^{2}_{3}$ & 0.0591 & ---    & 0.0117 & 2.3354 & 0.0051 & 1.2057 & 0.0009 & 2.4845 \\
    $\omega^{2}_{4}$ & 0.0950 & ---    & 0.0249 & 1.9339 & 0.0073 & 1.7620 & 0.0012 & 2.6472 \\
    $\omega^{2}_{5}$ & 0.0950 & ---    & 0.0157 & 2.5928 & 0.0048 & 1.7101 & 0.0012 & 1.9661 \\
    $\omega^{2}_{6}$ & 0.1005 & ---    & 0.0277 & 1.8595 & 0.0095 & 1.5416 & 0.0015 & 2.6522 \\
    $\omega^{2}_{7}$ & 0.1937 & ---    & 0.0462 & 2.0670 & 0.0100 & 2.2068 & 0.0018 & 2.4813 \\
    $\omega^{2}_{8}$ & 0.0864 & ---    & 0.0282 & 1.6165 & 0.0080 & 1.8129 & 0.0017 & 2.2334 \\
    $\omega^{2}_{9}$ & 0.1836 & ---    & 0.0529 & 1.7944 & 0.0152 & 1.7964 & 0.0025 & 2.6037 \\
    \bottomrule
  \end{tabular*}
\end{table}

\subsection{Validation of band structures}
\label{subsec:band-structures}
We assess MUFEM on the unit cells $\Omega=[0,1]^d$ ($d=2,3$)
by comparing the computed band structures~\eqref{eq:bandstructure}
with PWE results for low-contrast media~\cite{sukhoivanov2009photonic}
and discontinuous Galerkin (DG) benchmark results for high-contrast
media~\cite{lu2017discontinuous}.
For these unit cells, the corresponding Brillouin zones
\eqref{eq:Brillouinzone} are shown in Fig.~\ref{fig:Brillouin-zones}.
As is standard, symmetry arguments can be used to restrict the computations to the \textit{irreducible} Brillouin zones (yellow regions in Fig.~\ref{fig:Brillouin-zones}) without loss of generality,
and we study the \textit{band diagrams}, obtained by restricting the band structures to selected high-symmetry paths along the boundaries of these regions. For brevity, we omit the subscript $h$ from the computed bands. Agreement is quantified by the mean absolute difference in
normalized frequency, averaged over the reported bands and
sampled quasi-momenta.

The two-dimensional examples comprise a flower-shaped inclusion
with a nonconvex interface, rotated elliptic inclusions to assess
robustness with respect to interface orientation, and a circular
inclusion as a high-contrast benchmark. We extend the validation to three-dimensional photonic crystals with popcorn-shaped and scaffold inclusions. 

\begin{figure}[htbp]
  \centering
  \begin{subfigure}{0.5\linewidth}
    \centering
    \includegraphics[width=0.72\linewidth]{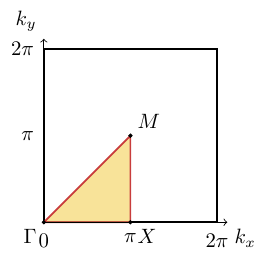}
    \caption{Two-dimensional square lattice}
    \label{fig:2dBrillouin}
  \end{subfigure}
  \begin{subfigure}{0.42\linewidth}
    \centering
    \includegraphics[width=0.72\linewidth]{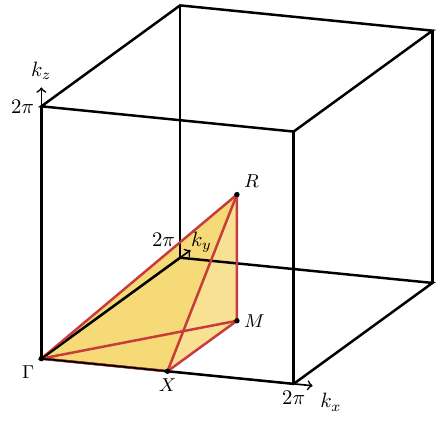}
    \caption{Three-dimensional cubic lattice}
    \label{fig:3dBrillouin}
  \end{subfigure}
  \caption{\small Brillouin zones and selected high-symmetry
  paths: (a) $\Gamma$--$X$--$M$--$\Gamma$ for the two-dimensional square
  lattice, where $\Gamma=(0,0)$, $X=(\pi,0)$, and $M=(\pi,\pi)$;
  (b) $\Gamma$--$X$--$M$--$R$--$\Gamma$ for the three-dimensional simple
  cubic lattice, where $\Gamma=(0,0,0)$, $X=(\pi,0,0)$,
  $M=(\pi,\pi,0)$, and $R=(\pi,\pi,\pi)$.}
  \label{fig:Brillouin-zones}
\end{figure}

\subsubsection{Two-dimensional crystals}
\label{sssec:2d-band}
For the two-dimensional cases, we consider dielectric inclusions
$\Omega^-$ embedded in air ($\varepsilon^{+}=1$), $\Omega^+=\Omega\setminus\overline{\Omega^-}$. 
The lowest bands ($\omega_{1}(\bm{k})$, $\omega_{2}(\bm{k})$,\ldots, $\omega_{n}(\bm{k})$) are computed along the high-symmetry path $\Gamma\to X\to M\to\Gamma$ (Fig.~\ref{fig:2dBrillouin}). 
For each geometry, we use the same uniform triangulation as in Section~\ref{sssec:hom-conv}, with $N=64$, resulting in $2N^2=8192$ triangular cells and a mesh size of $h=\sqrt{2}/64\approx 0.0221$. We use second-family N\'ed\'elec elements of degree one and the same quasi-momentum sampling. Each segment of the high-symmetry path is divided into 20 equal intervals, yielding 61 sampling positions along the full path. The frequencies are reported in the normalized form $\omega a/(2\pi c)$.

\paragraph{Flower-shaped inclusion.}
We use the flower-shaped geometry in~\eqref{eq:flower-levelset} with
$\varepsilon^{-}=2$ to enable comparison with the PWE
method of~\cite[Section~5.15]{sukhoivanov2009photonic}.
Fig.~\ref{fig:flower-band-result} compares the MUFEM and PWE bands
at the same quasi-momentum samples along $\Gamma$--$X$--$M$--$\Gamma$.
The first nine bands agree closely with the PWE results,
with a mean absolute difference of $2.2561\times10^{-3}$
in the normalized frequency $\omega a/(2\pi c)$.
\begin{figure}[htbp]
  \centering
  \begin{subfigure}{0.60\linewidth}
    \centering
    \includegraphics[width=\linewidth]{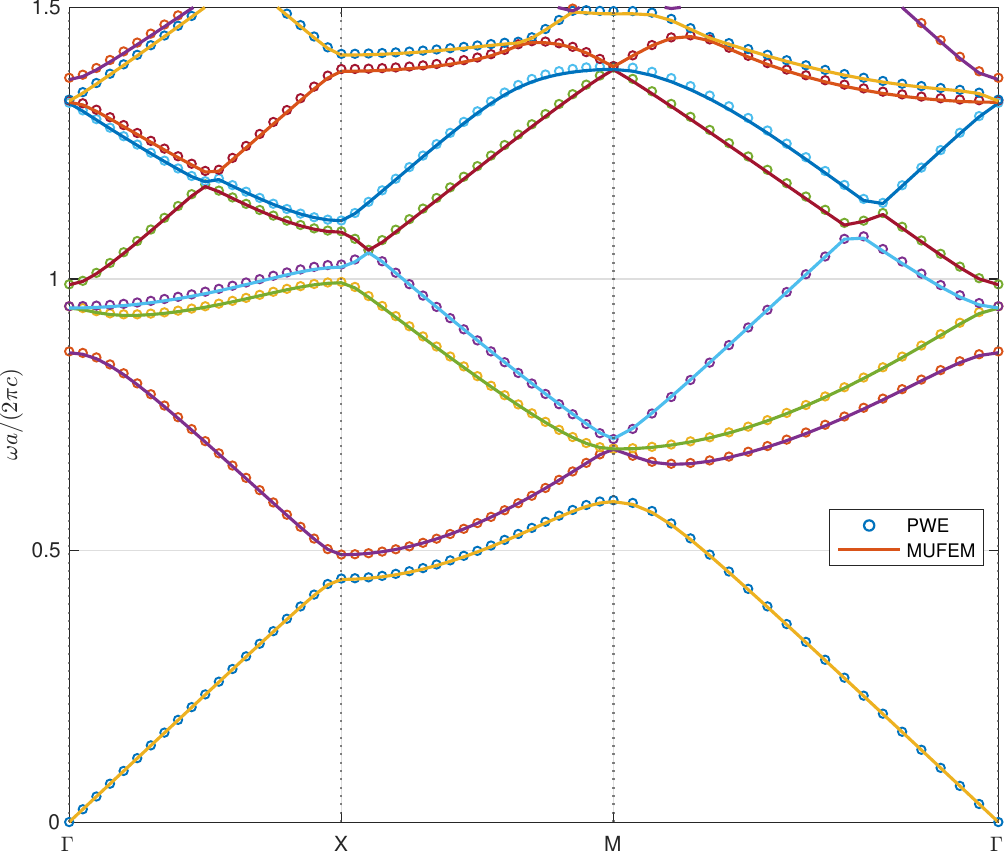}
    \caption{Band structure.}
    \label{fig:flower-band-result}
  \end{subfigure}
  \hfill
  \begin{subfigure}{0.35\linewidth}
    \centering
    \includegraphics[width=\linewidth]{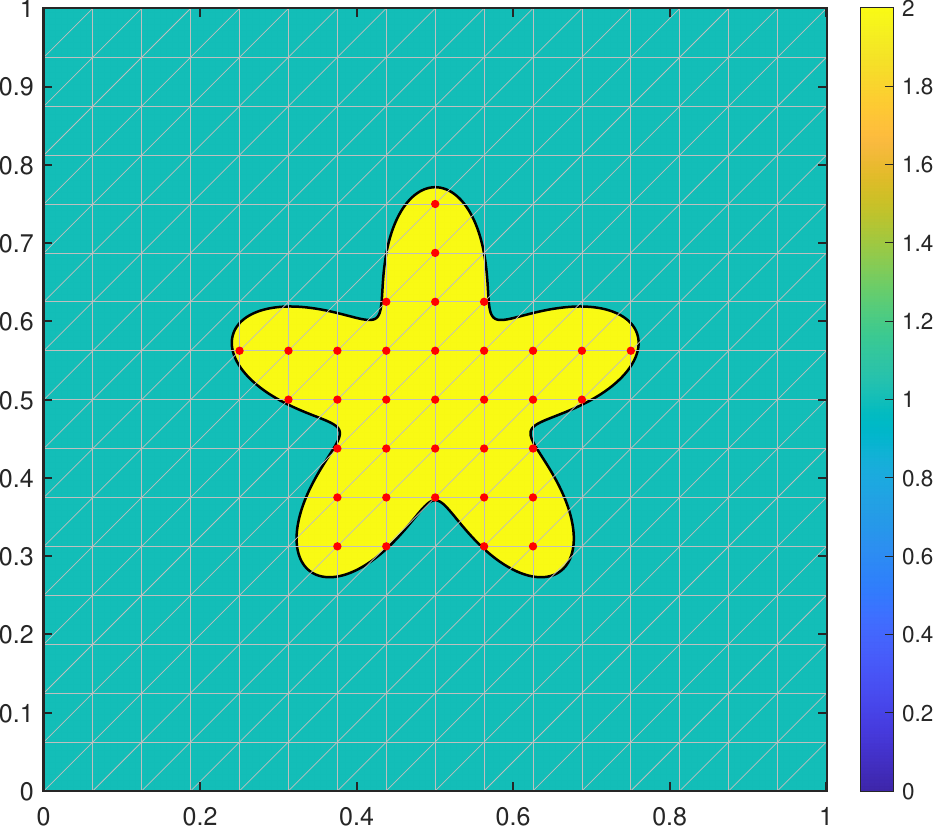}
    \caption{Unit-cell geometry.}
    \label{fig:flower-band-geometry}
  \end{subfigure}
  \caption{Photonic band structure for a square lattice with a low-contrast flower-shaped dielectric inclusion and $\varepsilon^-=2$ in air: (a) the MUFEM results (solid curves) compared with the PWE reference results (dots)~\cite{sukhoivanov2009photonic}; (b) the unit-cell geometry.}
  \label{fig:band2d_flower}
\end{figure}

\paragraph{Elliptic inclusion.}
We next consider a rotated elliptic inclusion centered at $\bm x_{c} = (0.5,0.5)$ with semi-major axis $0.25$ and semi-minor axis $0.16$. For a rotation angle $\alpha$, we define
\begin{equation}
  \Omega^{-} \,=\, \bigl\{\bm x\in\Omega \;\bigl|\; \varphi_{\alpha}(\bm x) < 0\bigr\},
  \qquad
  \varphi_{\alpha}(x,y) = \left(\frac{x_{\alpha}}{0.25}\right)^{2} + \left(\frac{y_{\alpha}}{0.16}\right)^{2} - 1,
\end{equation}
where
\begin{equation}
  x_{\alpha} = \cos\alpha\,(x-0.5) + \sin\alpha\,(y-0.5), \qquad
  y_{\alpha} = -\sin\alpha\,(x-0.5) + \cos\alpha\,(y-0.5),
\end{equation}
and we take $\alpha\in\{0^{\circ},45^{\circ},90^{\circ}\}$ with low contrast $\varepsilon^{-}=2$ and $\varepsilon^{+}=1$.
Fig.~\ref{fig:ellipse-band-result} compares MUFEM and PWE
results~\cite{sukhoivanov2009photonic} for these three orientations
on a fixed background mesh.
For the first nine bands, the two sets of band curves are visually indistinguishable, the mean absolute differences are
$1.2145\times10^{-3}$,
$2.3046\times10^{-3}$, and $1.2663\times10^{-3}$ for
$\alpha=0^{\circ}$, $45^{\circ}$, and $90^{\circ}$, respectively.

\begin{figure}[htbp]
  \centering
  \begin{subfigure}{0.96\linewidth}
    \centering
    \includegraphics[width=\linewidth]{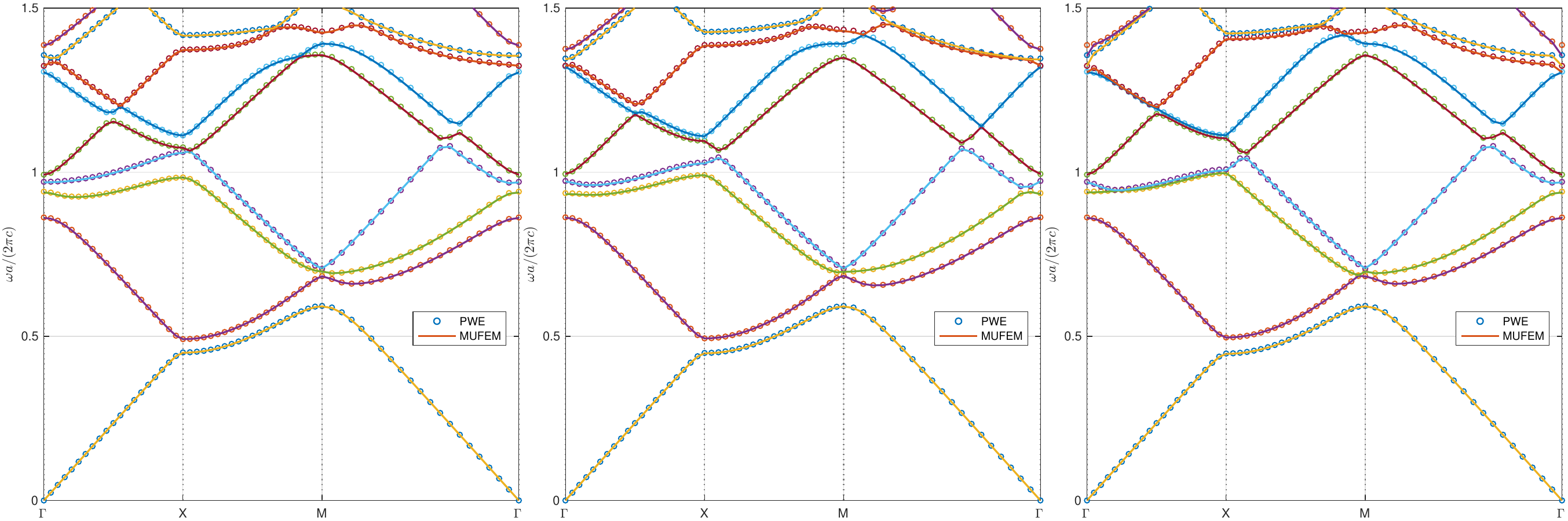}
    \caption{Band structures at $0^{\circ}$, $45^{\circ}$, and $90^{\circ}$.}
    \label{fig:ellipse-band-result}
  \end{subfigure}
  \begin{subfigure}{0.96\linewidth}
    \centering
    \includegraphics[width=\linewidth]{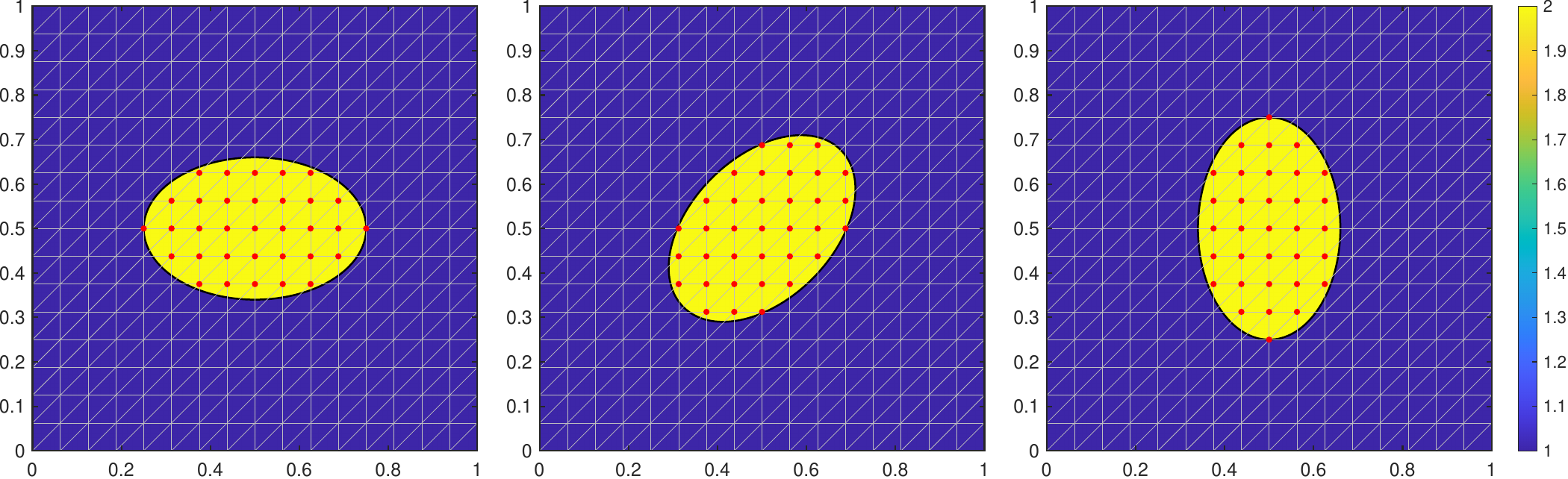}
    \caption{Physical regions for the same rotations.}
    \label{fig:ellipse-band-geometry}
  \end{subfigure}
  \caption{Photonic band structure for a square lattice with a low-contrast elliptic inclusion and $\varepsilon^-=2$ in air. The upper panel compares the MUFEM results (solid curves) with the PWE reference results (dots)~\cite{sukhoivanov2009photonic} for inclusion orientations of $0^{\circ}$, $45^{\circ}$, and $90^{\circ}$ on a fixed background mesh, and the lower panel shows the corresponding physical regions.}
  \label{fig:band2d_ellipse}
\end{figure}

\paragraph{Circular inclusion.}
We consider the circular dielectric rod shown in
Fig.~\ref{fig:circle-band-geometry}, with radius $r_0=0.2$ and center
$(0.5,0.5)$, such that
\begin{equation}
  \Omega^{-} \,=\, \bigl\{\bm x\in\Omega \;\bigl|\; \varphi_{\circ}(\bm x) < 0\bigr\},
  \qquad
  \varphi_{\circ}(x,y) = \sqrt{(x-0.5)^{2}+(y-0.5)^{2}} - r_{0},
\end{equation}
with permittivity
\begin{equation}
  \varepsilon(\bm x) =
  \begin{cases}
    \varepsilon^{-} = 8.9, & \bm x\in\Omega^{-},\\[2pt]
    \varepsilon^{+} = 1,   & \bm x\in\Omega^{+}.
  \end{cases}
\end{equation}
Fig.~\ref{fig:circle-band-result} compares the bands computed
by MUFEM with those obtained using the DG method
in~\cite[Fig.~7]{lu2017discontinuous}.
The first six bands agree closely with the DG results,
including the band-gap location, with a mean absolute
difference of $1.7925\times10^{-3}$.

\begin{figure}[htbp]
  \centering
  \begin{subfigure}{0.6\linewidth}
    \centering
    \includegraphics[width=\linewidth]
    {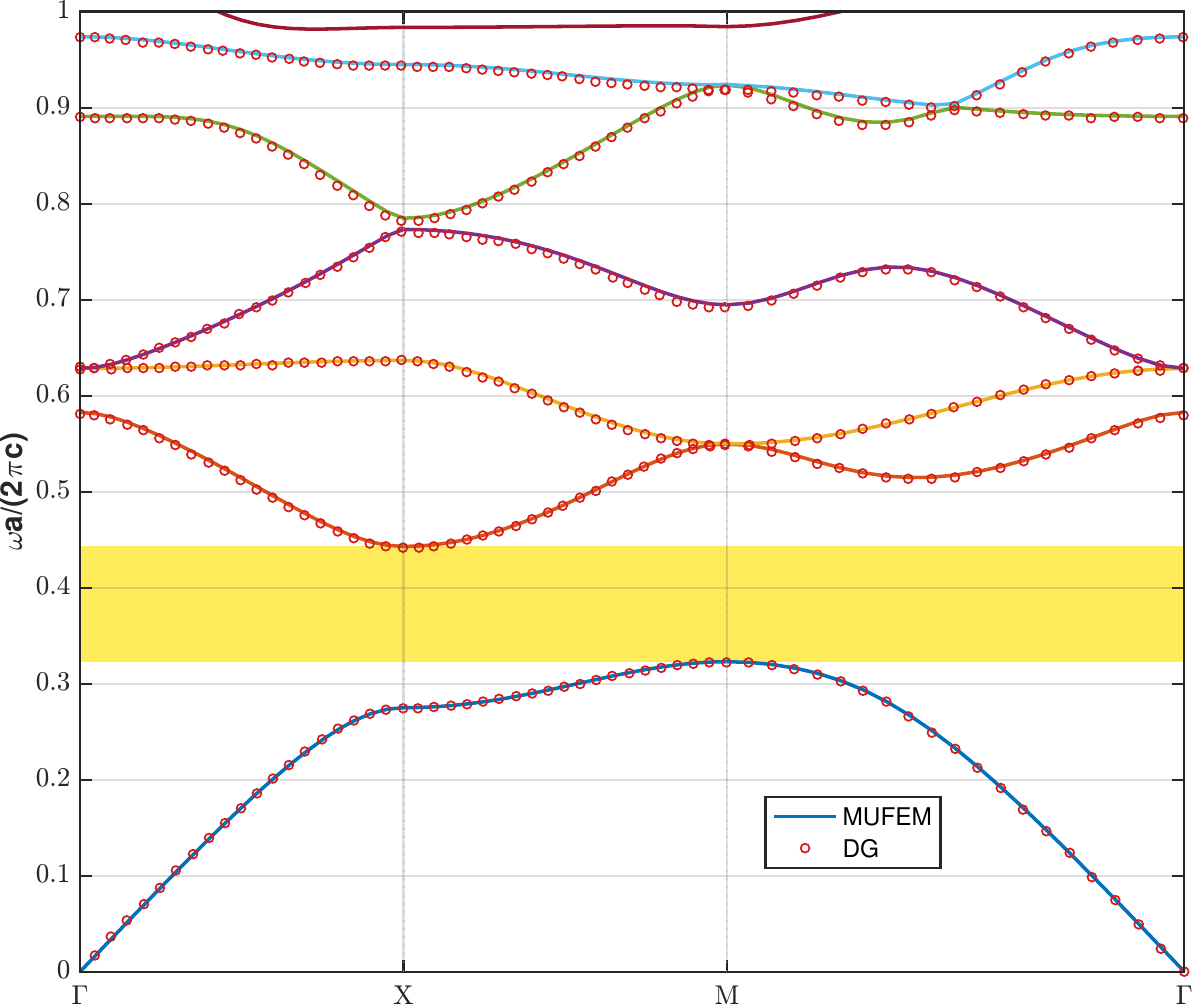}
    \caption{Band structure comparison.}
    \label{fig:circle-band-result}
  \end{subfigure}
  \hfill
  \begin{subfigure}{0.35\linewidth}
    \centering
    \includegraphics[width=\linewidth]{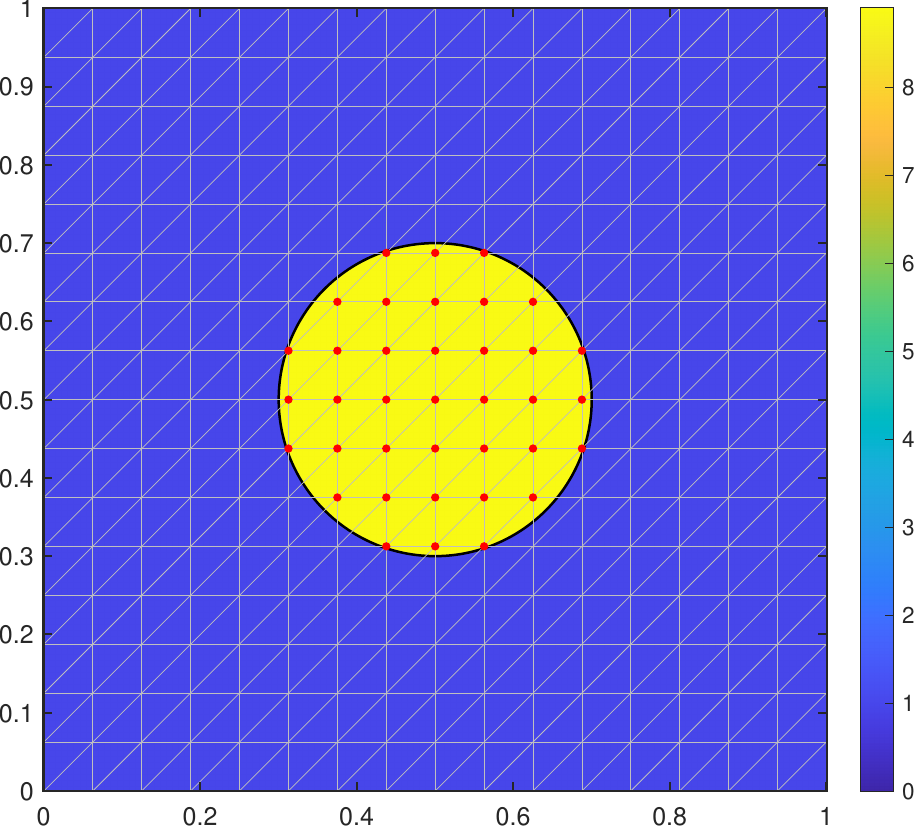}
    \caption{Unit-cell geometry.}
    \label{fig:circle-band-geometry}
  \end{subfigure}
  \caption{Photonic band structure for a square lattice with a circular
  dielectric inclusion of radius $r_0=0.2$ and permittivity
  $\varepsilon^-=8.9$ in air ($\varepsilon^{+}=1$): (a) bands computed by MUFEM (solid curves), compared with the DG results from Fig.~7 of~\cite{lu2017discontinuous} (dots); (b) unit-cell geometry.
  Band gaps are highlighted in (a).}
  \label{fig:band2d_circle}
\end{figure}

\subsubsection{Three-dimensional crystals}
\label{sssec:3d-band}
For the three-dimensional cases, we use the same uniform tetrahedral mesh as in the homogeneous case. The unit cell $\Omega=[0,1]^3$ is partitioned into
$N^3$ uniform cubic cells with $N=36$, and each cubic cell is subdivided into
six tetrahedra, resulting in a tetrahedral mesh with
$6N^3=279936$ elements. The frequencies are reported in the
normalized form $\omega a/(2\pi c)$, and the lowest bands are computed along
the high-symmetry path
$\Gamma\to X\to M\to R\to\Gamma$ shown in Fig.~\ref{fig:3dBrillouin}. Each segment of the high-symmetry path is divided into 12 equal intervals, yielding 49 sampling points along the full path.

\paragraph{Popcorn-shaped inclusion.}
We consider a popcorn-shaped inclusion~\cite{chern2007coupling}
to assess the feasibility of the unfitted discretization for a nonconvex three-dimensional interface, whose zero level set $\psi(x,y,z)=0$
is given by
\begin{equation}
\begin{aligned}
  \psi(x,y,z)
  ={}& \sqrt{(x-c_x)^2+(y-c_y)^2+(z-c_z)^2}-r_0 \\
  &-\sum_{k=0}^{11}A\exp\!\left(
  -\frac{(x-c_x-x_k)^2+(y-c_y-y_k)^2+(z-c_z-z_k)^2}
  {\sigma^2}\right).
\end{aligned}
\label{eq:popcorn-level-set}
\end{equation}
Here, the twelve bump offsets are defined by
\begin{align}
  (x_k,y_k,z_k)
  &= \frac{r_0}{\sqrt{5}}
  \left(2\cos\frac{2k\pi}{5},\,
        2\sin\frac{2k\pi}{5},\,1\right),
  && 0\leq k\leq4, \\
  (x_k,y_k,z_k)
  &= \frac{r_0}{\sqrt{5}}
  \left(2\cos\frac{(2(k-5)-1)\pi}{5},\,
        2\sin\frac{(2(k-5)-1)\pi}{5},\,-1\right),
  && 5\leq k\leq9, \\
  (x_{10},y_{10},z_{10})&=(0,0,r_0),\,(x_{11},y_{11},z_{11})=(0,0,-r_0).
\end{align}
We take $\boldsymbol{c}=(0.5,0.5,0.5)$, $r_0=0.3$, $A=1$, and
$\sigma=0.1$, with low-contrast permittivities $\varepsilon^-=1$ and $\varepsilon^+=2$ inside and outside the popcorn inclusions, respectively. Fig.~\ref{fig:popcorn-band3d-result} compares the MUFEM results with the PWE
reference computation. The mean absolute difference from the PWE
results is $3.2489\times10^{-3}$ over the first six bands.
\begin{figure}[htbp]
  \centering
  \begin{subfigure}{0.58\linewidth}
    \centering
    \includegraphics[width=\linewidth]
    {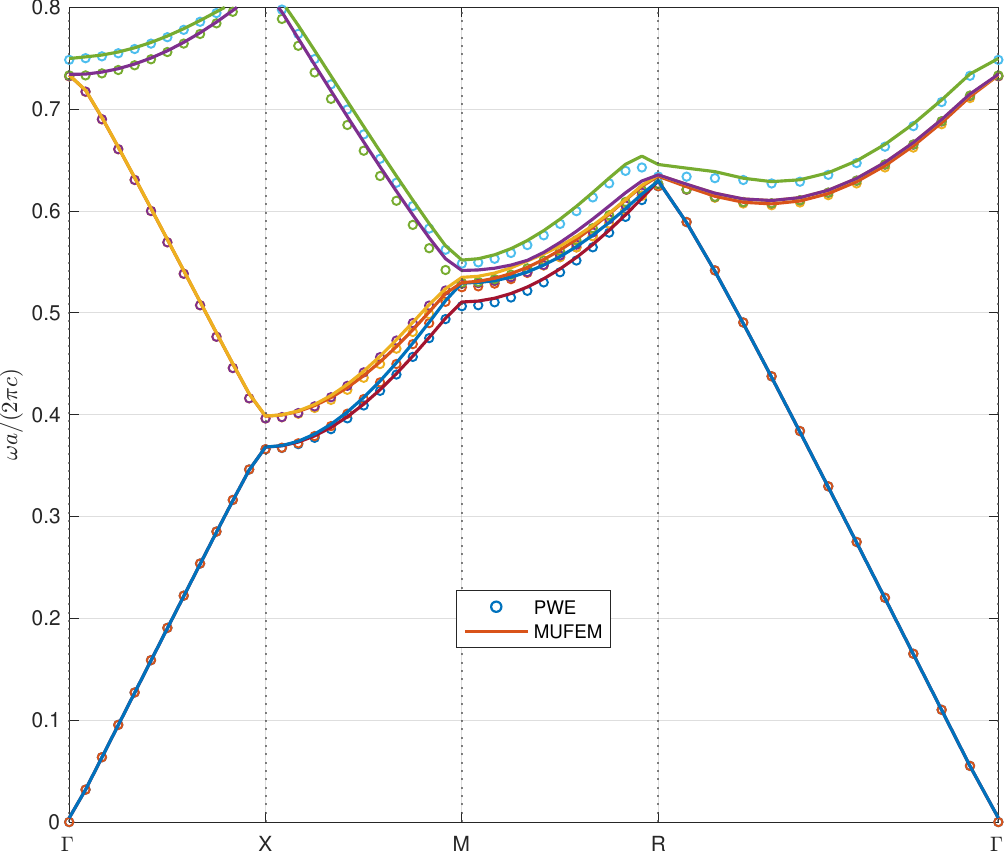}
    \caption{Band structure comparison.}
    \label{fig:popcorn-band3d-result}
  \end{subfigure}
  \begin{subfigure}{0.41\linewidth}
    \centering
    \includegraphics[width=\linewidth]{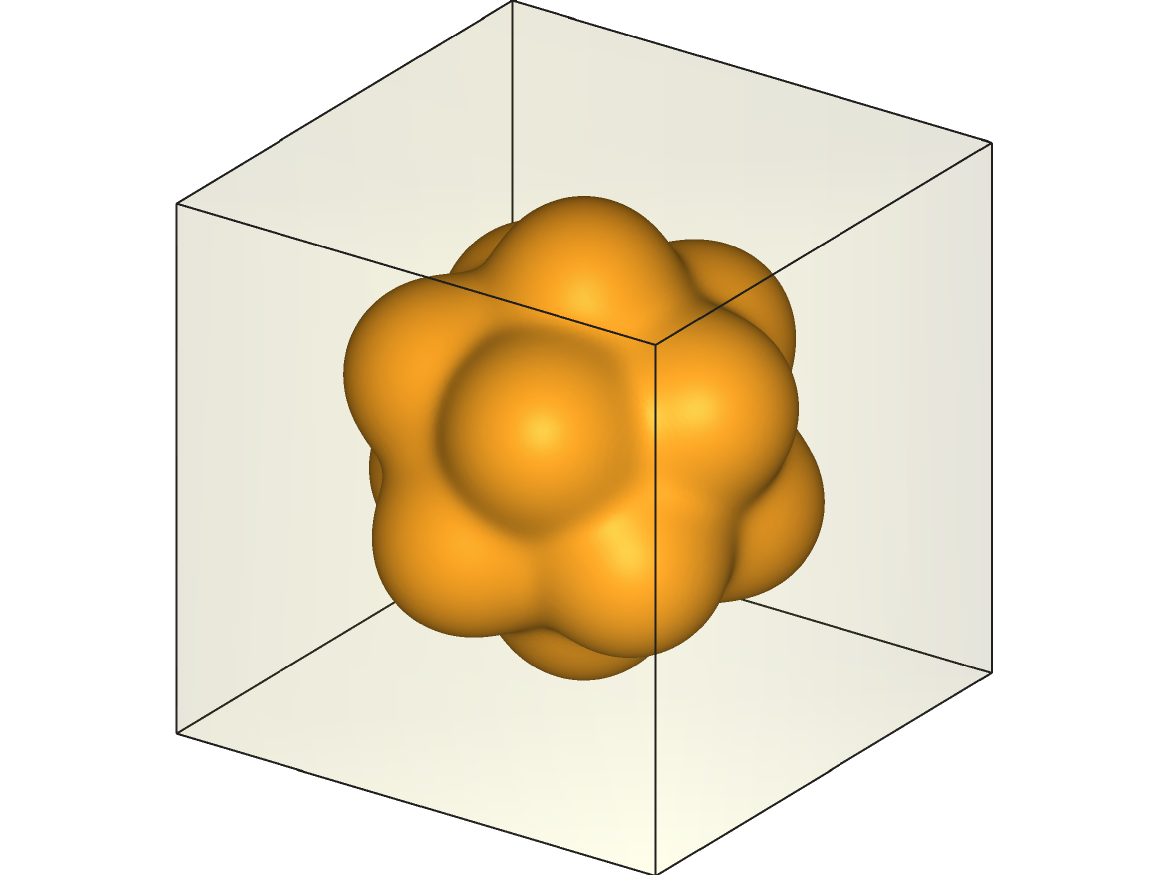}
    \caption{Unit-cell geometry.}
    \label{fig:popcorn-band3d-geometry}
  \end{subfigure}
  \caption{Photonic band structure for a three-dimensional cubic crystal
  with a low-contrast popcorn-shaped inclusion, $\varepsilon^-=1$ and
  $\varepsilon^+=2$: (a) the MUFEM results (solid curves) compared with the
  PWE reference results (dots); (b) the unit-cell geometry.}
  \label{fig:popcorn-band3d}
\end{figure}

\paragraph{Scaffold inclusion.}
We consider the three-dimensional scaffold structure shown in
Fig.~\ref{fig:scaffold-band3d-geometry}, presented by~\cite{lu2017discontinuous,sozuer1993photonic}.
The structure consists of silicon inclusions with $\varepsilon^-=13$
embedded in an air matrix with $\varepsilon^+=1$.
Following~\cite{lu2017discontinuous},
the scaffold frame has thickness $d=0.25$.
Fig.~\ref{fig:scaffold-band3d-result} compares the first seven MUFEM
bands with the reference results in~\cite[Fig.~12]{lu2017discontinuous}.
The MUFEM bands show some deviations from the DG results, while the overall band structure and band-gap locations remain in good agreement, with a mean absolute difference of $4.5534\times10^{-3}$.
\begin{figure}[htbp]
  \centering
  \begin{subfigure}{0.58\linewidth}
    \centering
    \includegraphics[width=\linewidth]{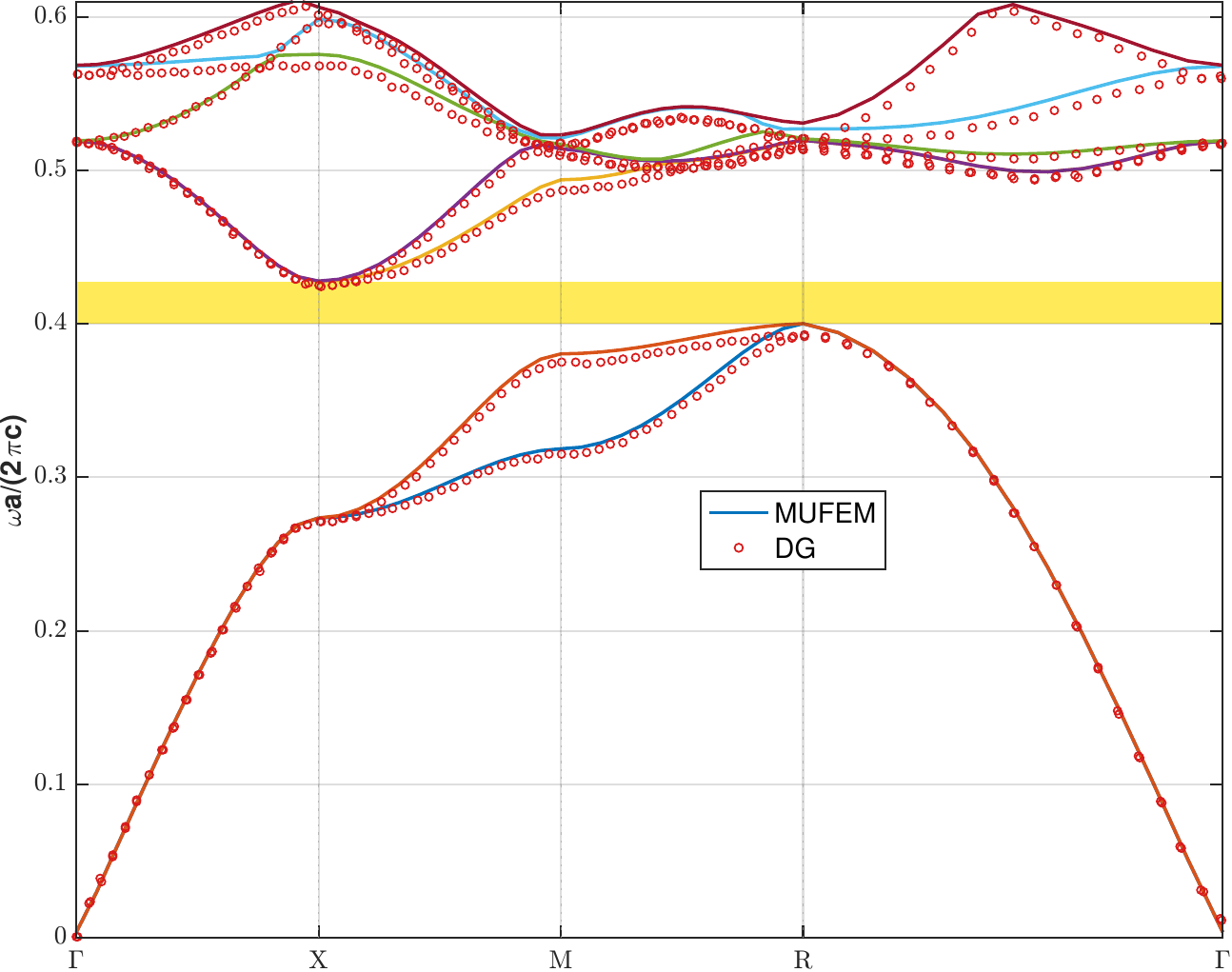}
    \caption{Band structure comparison.}
    \label{fig:scaffold-band3d-result}
  \end{subfigure}
  \begin{subfigure}{0.35\linewidth}
    \centering
    \includegraphics[width=\linewidth]{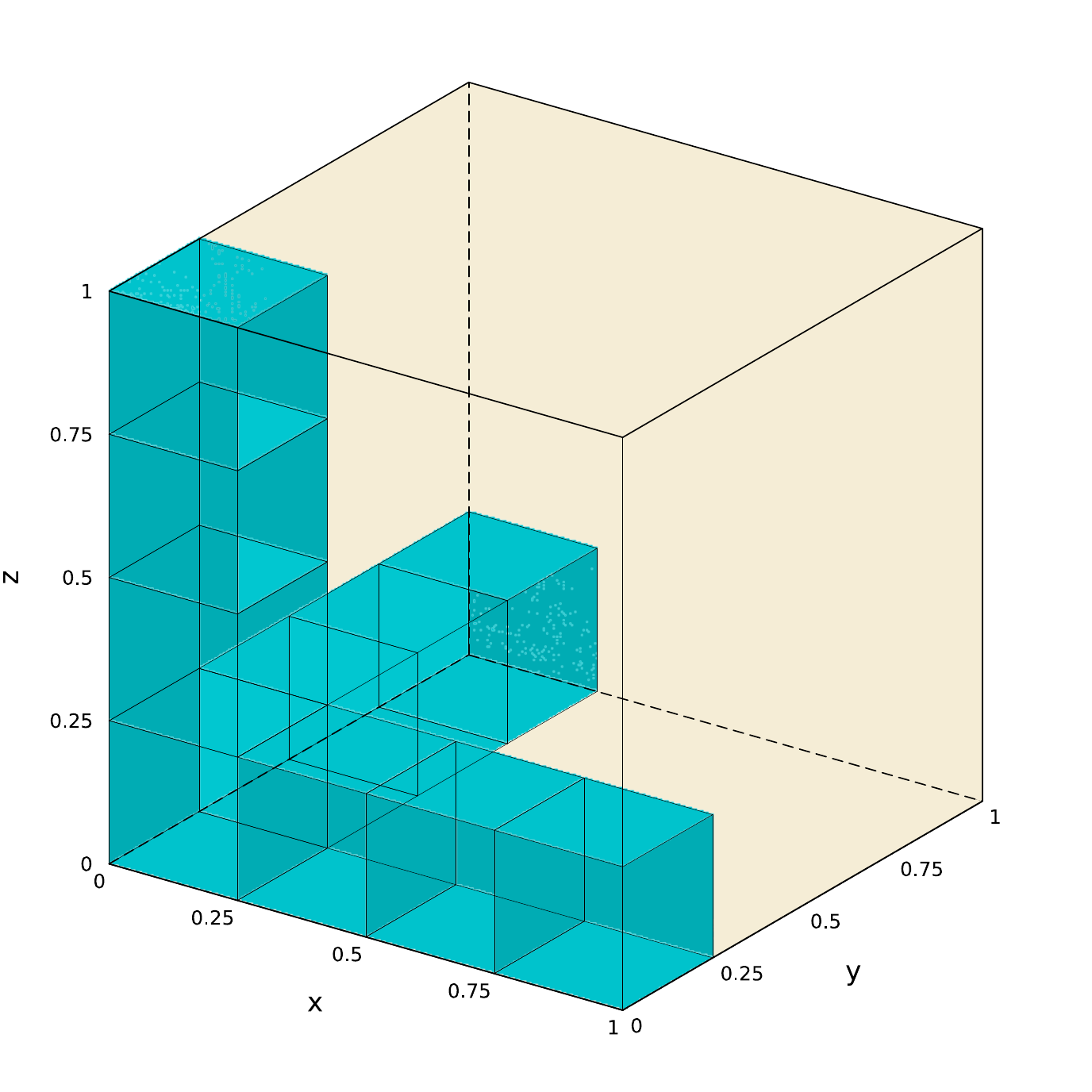}
    \caption{Unit-cell geometry.}
    \label{fig:scaffold-band3d-geometry}
  \end{subfigure}
  \caption{Photonic band structure for the three-dimensional scaffold
  structure~\cite{lu2017discontinuous}, with frame thickness
  $d=0.25$ and $\varepsilon^-=13$ in air ($\varepsilon^+=1$). (a)~The first seven bands
  computed by MUFEM (solid curves), compared with the
  reference results (dots). 
  (b)~The unit-cell geometry. Band gaps are
  highlighted in (a).}
  \label{fig:scaffold-band3d}
\end{figure}

%% file: section7.tex
\section{Conclusion}
\label{sec:conclusion}
In this paper, we have developed a mixed unfitted finite element method for computing photonic band structures with complex interface geometries. By applying the Floquet–Bloch transform, we have reduced the quasi-periodic Maxwell eigenvalue problem to a periodic problem on a single unit cell. Our method employs a uniform background mesh independent of the inclusion geometry, imposes the interface conditions weakly through Nitsche’s method, and stabilizes small cut elements using ghost penalties. A key feature of our construction is that the modified finite element spaces preserve the discrete de~Rham structure associated with the shifted operator. On the theoretical side, we have established the well-posedness of the discrete mixed formulation, derived optimal error estimates for the associated source problem, and proved the convergence of the eigenvalues and eigenspaces using the Babu\v{s}ka–Osborn spectral approximation theory. Numerical experiments in two and three dimensions have confirmed the predicted convergence behavior and shown close agreement with reference band structures for complex inclusion geometries.

%% file: references.bib
@article{yablonovitch1987inhibited,
  title={Inhibited spontaneous emission in solid-state physics and electronics},
  author={Yablonovitch, Eli},
  journal={Physical review letters},
  volume={58},
  number={20},
  pages={2059},
  year={1987},
  publisher={APS}
}

@article{john1987strong,
  title={Strong localization of photons in certain disordered dielectric superlattices},
  author={John, Sajeev},
  journal={Physical review letters},
  volume={58},
  number={23},
  pages={2486},
  year={1987},
  publisher={APS}
}

@article{joannopoulos2008molding,
  title={Molding the flow of light},
  author={Joannopoulos, John D and Johnson, Steven G and Winn, Joshua N and Meade, Robert D},
  journal={Princet. Univ. Press. Princeton, NJ [ua]},
  volume={12},
  pages={33},
  year={2008}
}

@article{russell2003photonic,
  title={Photonic crystal fibers},
  author={Russell, Philip},
  journal={science},
  volume={299},
  number={5605},
  pages={358--362},
  year={2003},
  publisher={American Association for the Advancement of Science}
}

@book{sakoda2005optical,
  title={Optical properties of photonic crystals},
  author={Sakoda, Kazuaki},
  year={2005},
  publisher={Springer}
}

@article{cangellaris1991analysis,
  title={Analysis of the numerical error caused by the stair-stepped approximation of a conducting boundary in FDTD simulations of electromagnetic phenomena},
  author={Cangellaris, Andreas C and Wright, Diana B},
  journal={IEEE transactions on antennas and propagation},
  volume={39},
  number={10},
  pages={1518--1525},
  year={1991},
  publisher={IEEE}
}

@article{giani2012adaptive,
  title={Adaptive finite element methods for computing band gaps in photonic crystals},
  author={Giani, Stefano and Graham, Ivan G},
  journal={Numerische Mathematik},
  volume={121},
  number={1},
  pages={31--64},
  year={2012},
  publisher={Springer}
}

@article{lin1998experimental,
  title={Experimental demonstration of guiding and bending of electromagnetic waves in a photonic crystal},
  author={Lin, Shawn-Yu and Chow, Edmund and Hietala, Vince and Villeneuve, Pierre R and Joannopoulos, JD},
  journal={Science},
  volume={282},
  number={5387},
  pages={274--276},
  year={1998},
  publisher={American Association for the Advancement of Science}
}

@article{krauss1996two,
  title={Two-dimensional photonic-bandgap structures operating at near-infrared wavelengths},
  author={Krauss, Thomas F and Rue, Richard M De La and Brand, Stuart},
  journal={Nature},
  volume={383},
  number={6602},
  pages={699--702},
  year={1996},
  publisher={Nature Publishing Group UK London}
}

@article{ho1990existence,
  title={Existence of a photonic gap in periodic dielectric structures},
  author={Ho, KM and Chan, Che Ting and Soukoulis, Costas M},
  journal={Physical review letters},
  volume={65},
  number={25},
  pages={3152},
  year={1990},
  publisher={APS}
}

@article{inguva1992photonic,
  title={Photonic bands: Convergence problems with the plane-wave method},
  author={Inguva, R},
  journal={PHYSICAL REVIEW B},
  volume={45},
  number={24},
  year={1992}
}

@inproceedings{nitsche1971variationsprinzip,
  title={{\"U}ber ein Variationsprinzip zur L{\"o}sung von Dirichlet-Problemen bei Verwendung von Teilr{\"a}umen, die keinen Randbedingungen unterworfen sind},
  author={Nitsche, Joachim},
  booktitle={Abhandlungen aus dem mathematischen Seminar der Universit{\"a}t Hamburg},
  volume={36},
  pages={9--15},
  year={1971},
  organization={Springer}
}

@article{hansbo2002unfitted,
  title={An unfitted finite element method, based on Nitsche’s method, for elliptic interface problems},
  author={Hansbo, Anita and Hansbo, Peter},
  journal={Computer methods in applied mechanics and engineering},
  volume={191},
  number={47-48},
  pages={5537--5552},
  year={2002},
  publisher={Elsevier}
}

@book{brenner2008mathematical,
  title={The mathematical theory of finite element methods},
  author={Brenner, Susanne C and Scott, L Ridgway},
  year={2008},
  publisher={Springer}
}

@article{burman2010ghost,
  title={Ghost penalty},
  author={Burman, Erik},
  journal={Comptes Rendus Mathematique},
  volume={348},
  number={21-22},
  pages={1217--1220},
  year={2010},
  publisher={Elsevier}
}

@article{burman2012fictitious,
  title={Fictitious domain finite element methods using cut elements: II. A stabilized Nitsche method},
  author={Burman, Erik and Hansbo, Peter},
  journal={Applied Numerical Mathematics},
  volume={62},
  number={4},
  pages={328--341},
  year={2012},
  publisher={Elsevier}
}

@article{burman2015cutfem,
  title={CutFEM: discretizing geometry and partial differential equations},
  author={Burman, Erik and Claus, Susanne and Hansbo, Peter and Larson, Mats G and Massing, Andr{\'e}},
  journal={International Journal for Numerical Methods in Engineering},
  volume={104},
  number={7},
  pages={472--501},
  year={2015},
  publisher={Wiley Online Library}
}

@book{kuchment2012floquet,
  title={Floquet theory for partial differential equations},
  author={Kuchment, Peter A},
  volume={60},
  year={2012},
  publisher={Birkh{\"a}user}
}

@article{painter1999two,
  title={Two-dimensional photonic band-gap defect mode laser},
  author={Painter, Oskar and Lee, RK and Scherer, Axel and Yariv, A and O'brien, JD and Dapkus, PD and Kim, I},
  journal={Science},
  volume={284},
  number={5421},
  pages={1819--1821},
  year={1999},
  publisher={American Association for the Advancement of Science}
}

@book{evans2022partial,
  title={Partial differential equations},
  author={Evans, Lawrence C},
  volume={19},
  year={2022},
  publisher={American mathematical society}
}

@article{joannopoulos1995photonic,
  title={Photonic crystals: molding the flow of light Princeton University},
  author={Joannopoulos, JD and Meads, RD and Winn, JN},
  journal={Princeton, NJ},
  year={1995}
}

@book{kittel2018introduction,
  title={Introduction to solid state physics},
  author={Kittel, Charles and McEuen, Paul},
  year={2018},
  publisher={John Wiley \& Sons}
}

@article{nedelec1980mixed,
  title={Mixed finite elements in R3},
  author={N{\'e}d{\'e}lec, Jean-Claude},
  journal={Numerische Mathematik},
  volume={35},
  number={3},
  pages={315--341},
  year={1980},
  publisher={Springer}
}

@article{nedelec1986new,
  title={A new family of mixed finite elements in R3},
  author={N{\'e}d{\'e}lec, Jean-Claude},
  journal={Numerische Mathematik},
  volume={50},
  number={1},
  pages={57--81},
  year={1986},
  publisher={Springer}
}

@book{dorfler2011photonic,
  title={Photonic crystals: Mathematical analysis and numerical approximation},
  author={D{\"o}rfler, Willy and Lechleiter, Armin and Plum, Michael and Schneider, Guido and Wieners, Christian},
  year={2011},
  publisher={Springer Science \& Business Media}
}

@article{sticko2020high,
  title={High-order cut finite elements for the elastic wave equation},
  author={Sticko, Simon and Ludvigsson, Gustav and Kreiss, Gunilla},
  journal={Advances in Computational Mathematics},
  volume={46},
  number={3},
  pages={45},
  year={2020},
  publisher={Springer}
}

@book{monk2003finite,
  title={Finite element methods for Maxwell's equations},
  author={Monk, Peter},
  year={2003},
  publisher={Oxford university press}
}

@article{lu2017discontinuous,
  title={Discontinuous Galerkin approximations for computing electromagnetic Bloch modes in photonic crystals},
  author={Lu, Zhongjie and {\c{C}}esmelioglu, Ay{\c{c}}il and Van der Vegt, Jaap JW and Xu, Yan},
  journal={Journal of scientific computing},
  volume={70},
  number={2},
  pages={922--964},
  year={2017},
  publisher={Springer}
}

@article{boffi2006modified,
  title={Modified edge finite elements for photonic crystals},
  author={Boffi, Daniele and Conforti, Matteo and Gastaldi, Lucia},
  journal={Numerische Mathematik},
  volume={105},
  pages={249--266},
  year={2006},
  publisher={Springer}
}

@article{wang2022convergence,
  title={Convergence analysis of Nitsche extended finite element methods for H (curl)-elliptic interface problems},
  author={Wang, Nan and Chen, Jinru},
  journal={International Journal of Numerical Analysis and Modeling},
  volume={19},
  number={4},
  pages={487--510},
  year={2022}
}

@article{gurkan2019stabilized,
  title={A stabilized cut discontinuous Galerkin framework for elliptic boundary value and interface problems},
  author={G{\"u}rkan, Ceren and Massing, Andr{\'e}},
  journal={Computer Methods in Applied Mechanics and Engineering},
  volume={348},
  pages={466--499},
  year={2019},
  publisher={Elsevier}
}

@misc{jackson1999classical,
  title={Classical electrodynamics},
  author={Jackson, John David and Fox, Ronald F},
  year={1999},
  publisher={American Association of Physics Teachers}
}

@inproceedings{floquet1883equations,
  title={Sur les {\'e}quations diff{\'e}rentielles lin{\'e}aires {\`a} coefficients p{\'e}riodiques},
  author={Floquet, Gaston},
  booktitle={Annales scientifiques de l'{\'E}cole normale sup{\'e}rieure},
  volume={12},
  pages={47--88},
  year={1883}
}

@article{bloch1929quantenmechanik,
  title={{\"U}ber die quantenmechanik der elektronen in kristallgittern},
  author={Bloch, Felix},
  journal={Zeitschrift f{\"u}r physik},
  volume={52},
  number={7},
  pages={555--600},
  year={1929},
  publisher={Springer}
}

@article{shi2004plane,
  title={Plane-wave expansion method for calculating band structure of photonic crystal slabs with perfectly matched layers},
  author={Shi, Shouyuan and Chen, Caihua and Prather, Dennis W},
  journal={Journal of the Optical Society of America A},
  volume={21},
  number={9},
  pages={1769--1775},
  year={2004},
  publisher={OSA}
}

@article{li2003photonic,
  title={Photonic band structures solved by a plane-wave-based transfer-matrix method},
  author={Li, Zhi-Yuan and Lin, Lan-Lan},
  journal={Physical Review E},
  volume={67},
  number={4},
  pages={046607},
  year={2003},
  publisher={APS}
}

@inproceedings{de2001finite,
  title={Finite-difference time-domain modelling of photonic crystal structures},
  author={de Ridder, Rent5 M and Stoffer, Remco},
  booktitle={Proceedings of 2001 3rd International Conference on Transparent Optical Networks (IEEE Cat. No. 01EX488)},
  pages={22--25},
  year={2001},
  organization={IEEE}
}

@article{wu2000analysis,
  title={Analysis of photonic crystal filters by the finite-difference time-domain technique},
  author={Wu, Bae-Ian and Yang, Eric and Kong, Jin Au and Oswald, JA and McIntosh, KA and Mahoney, L and Verghese, S},
  journal={microwave and optical technology letters},
  volume={27},
  number={2},
  pages={81--87},
  year={2000},
  publisher={Wiley Online Library}
}

@article{andonegui2013finite,
  title={The finite element method applied to the study of two-dimensional photonic crystals and resonant cavities},
  author={Andonegui, Imanol and Garcia-Adeva, Angel J},
  journal={Optics express},
  volume={21},
  number={4},
  pages={4072--4092},
  year={2013},
  publisher={Optical Society of America}
}

@article{degirmenci2013finite,
  title={Finite element method analysis of band gap and transmission of two-dimensional metallic photonic crystals at terahertz frequencies},
  author={Degirmenci, Elif and Landais, Pascal},
  journal={Applied optics},
  volume={52},
  number={30},
  pages={7367--7375},
  year={2013},
  publisher={Optical Society of America}
}

@article{crocco2007generalized,
  title={Generalized scattering-matrix method for the analysis of two-dimensional photonic bandgap devices},
  author={Crocco, Lorenzo and Cuomo, Fabrizio and Isernia, Tommaso},
  journal={Journal of the Optical Society of America A},
  volume={24},
  number={10},
  pages={A12--A22},
  year={2007},
  publisher={OSA}
}

@article{yuan2008modeling,
  title={Modeling photonic crystals by boundary integral equations and Dirichlet-to-Neumann maps},
  author={Yuan, Jianhua and Lu, Ya Yan and Antoine, Xavier},
  journal={Journal of Computational Physics},
  volume={227},
  number={9},
  pages={4617--4629},
  year={2008},
  publisher={Elsevier}
}

@article{mescia2011particle,
  title={Particle swarm optimization for the design and characterization of silica-based photonic crystal fiber amplifiers},
  author={Mescia, L and Giaquinto, A and Fornarelli, G and Acciani, G and De Sario, M and Prudenzano, F},
  journal={Journal of non-crystalline solids},
  volume={357},
  number={8-9},
  pages={1851--1855},
  year={2011},
  publisher={Elsevier}
}

@article{wang2021finite,
  title={A finite element method for the band structure computation of photonic crystals with complex scatterer geometry},
  author={Wang, Liqun and Zhao, Meiling and Zhang, Yifan and Shi, Liwei},
  journal={Computer Physics Communications},
  volume={263},
  pages={107869},
  year={2021},
  publisher={Elsevier}
}

@article{norton2013planewave,
  title={Planewave expansion methods for photonic crystal fibres},
  author={Norton, Richard A and Scheichl, Robert},
  journal={Applied numerical mathematics},
  volume={63},
  pages={88--104},
  year={2013},
  publisher={Elsevier}
}

@article{guo2021unfitted,
  title={Unfitted Nitsche’s method for computing band structures of phononic crystals with periodic inclusions},
  author={Guo, Hailong and Yang, Xu and Zhu, Yi},
  journal={Computer Methods in Applied Mechanics and Engineering},
  volume={380},
  pages={113743},
  year={2021},
  publisher={Elsevier}
}

@article {guo2021topmaterials,
    AUTHOR = {Guo, Hailong and Yang, Xu and Zhu, Yi},
     TITLE = {Unfitted {N}itsche's method for computing wave modes in
              topological materials},
   JOURNAL = {J. Sci. Comput.},
  FJOURNAL = {Journal of Scientific Computing},
    VOLUME = {88},
      YEAR = {2021},
    NUMBER = {1},
     PAGES = {Paper No. 24, 28},
      ISSN = {0885-7474,1573-7691},
   MRCLASS = {65N30 (74A50 74J05)},
  MRNUMBER = {4270808},
       //DOI = {10.1007/s10915-021-01540-w},
}

@article {chen2009adaptiveimmersed,
    AUTHOR = {Chen, Zhiming and Xiao, Yuanming and Zhang, Linbo},
     TITLE = {The adaptive immersed interface finite element method for
              elliptic and {M}axwell interface problems},
   JOURNAL = {J. Comput. Phys.},
  FJOURNAL = {Journal of Computational Physics},
    VOLUME = {228},
      YEAR = {2009},
    NUMBER = {14},
     PAGES = {5000--5019},
      ISSN = {0021-9991,1090-2716},
   MRCLASS = {65N30 (78M10)},
  MRNUMBER = {2537841},
MRREVIEWER = {Nilima\ Nigam},
       //DOI = {10.1016/j.jcp.2009.03.044},
       //URL = {https://doi.org/10.1016/j.jcp.2009.03.044},
}

@article {chen2024highorder,
    AUTHOR = {Chen, Zhiming and Li, Ke and Lyu, Maohui and Xiang, Xueshaung},
     TITLE = {A high order unfitted finite element method for time-harmonic
              {M}axwell interface problems},
   JOURNAL = {Int. J. Numer. Anal. Model.},
  FJOURNAL = {International Journal of Numerical Analysis and Modeling},
    VOLUME = {21},
      YEAR = {2024},
    NUMBER = {6},
     PAGES = {822--849},
      ISSN = {1705-5105,2617-8710},
   MRCLASS = {65N30 (65N12 78M10)},
  MRNUMBER = {4809778},
MRREVIEWER = {Xia\ Ji},
       //DOI = {10.4208/ijnam2024-1033},
       //URL = {https://doi.org/10.4208/ijnam2024-1033},
}

@book{arnold2018finite,
  title={Finite element exterior calculus},
  author={Arnold, Douglas N},
  year={2018},
  publisher={SIAM}
}

@article{wang2025mixed,
  title={Mixed Nitsche extended finite element method for solving three-dimensional H (curl)-elliptic interface problems},
  author={Wang, Nan and Chu, Hanyu and Chen, Jinru and Cai, Ying},
  journal={Computers \& Mathematics with Applications},
  volume={199},
  pages={22--44},
  year={2025},
  publisher={Elsevier}
}

@article{cao2021virtual,
  title={A virtual finite element method for two-dimensional Maxwell interface problems with a background unfitted mesh},
  author={Cao, Shuhao and Chen, Long and Guo, Ruchi},
  journal={Mathematical Models and Methods in Applied Sciences},
  volume={31},
  number={14},
  pages={2907--2936},
  year={2021},
  publisher={World Scientific}
}

@article{liu2020interface,
  title={Interface-penalty finite element methods for interface problems in H1, H (curl), and H (div)},
  author={Liu, Huaqing and Zhang, Linbo and Zhang, Xiaodi and Zheng, Weiying},
  journal={Computer Methods in Applied Mechanics and Engineering},
  volume={367},
  pages={113137},
  year={2020},
  publisher={Elsevier}
}

@article{luo2009spectral,
  title={Spectral element method for band structures of two-dimensional anisotropic photonic crystals},
  author={Luo, Ma and Liu, Qing Huo and Li, Zhibing},
  journal={Physical Review E—Statistical, Nonlinear, and Soft Matter Physics},
  volume={79},
  number={2},
  pages={026705},
  year={2009},
  publisher={APS}
}

@article{busch2011discontinuous,
  title={Discontinuous Galerkin methods in nanophotonics},
  author={Busch, Kurt and K{\"o}nig, Michael and Niegemann, Jens},
  journal={Laser \& Photonics Reviews},
  volume={5},
  number={6},
  pages={773--809},
  year={2011},
  publisher={Wiley Online Library}
}

@article{chen2020multiphysics,
  title={Multiphysics simulation of plasmonic photoconductive devices using discontinuous Galerkin methods},
  author={Chen, Liang and Bagci, Hakan},
  journal={IEEE Journal on Multiscale and Multiphysics Computational Techniques},
  volume={5},
  pages={188--200},
  year={2020},
  publisher={IEEE}
}

@article{chen2016finite,
  title={Finite element methods for Maxwell equations},
  author={Chen, LONG},
  journal={Lecture Notes},
  year={2016}
}

@book{brezzi2012mixed,
  title={Mixed and hybrid finite element methods},
  author={Brezzi, Franco and Fortin, Michel},
  volume={15},
  year={2012},
  publisher={Springer Science \& Business Media}
}

@article{dobson2001analysis,
  title={Analysis of an algorithm for computing electromagnetic Bloch modes using Nedelec spaces},
  author={Dobson, David C and Pasciak, Joseph E},
  journal={Computational methods in applied mathematics},
  volume={1},
  number={2},
  pages={138--153},
  year={2001},
  publisher={De Gruyter Berlin}
}

@incollection{babuska1991eigenvalue,
  author    = {Babu\v{s}ka, I. and Osborn, J.},
  title     = {Eigenvalue Problems},
  booktitle = {Handbook of Numerical Analysis, Vol.\ II},
  editor    = {Ciarlet, P.~G. and Lions, J.-L.},
  publisher = {North-Holland, Amsterdam},
  year      = {1991},
  pages     = {641--787}
}

@article{osborn1975spectral,
  title={Spectral approximation for compact operators},
  author={Osborn, John E},
  journal={Mathematics of computation},
  volume={29},
  number={131},
  pages={712--725},
  year={1975}
}

@book{sukhoivanov2009photonic,
  title={Photonic crystals: physics and practical modeling},
  author={Sukhoivanov, Igor A and Guryev, Igor V},
  volume={152},
  year={2009},
  publisher={Springer}
}

@article{sozuer1993photonic,
  title={Photonic bands: simple-cubic lattice},
  author={S{\"o}z{\"u}er, H Sami and Haus, Joseph W},
  journal={Journal of the Optical Society of America B},
  volume={10},
  number={2},
  pages={296--302},
  year={1993},
  publisher={Optical Society of America}
}

@article{ciarlet1981two,
  title={Two-dimensional approximations of three-dimensional eigenvalue problems in plate theory},
  author={Ciarlet, Philippe Gaston and Kesavan, S},
  journal={Computer Methods in Applied Mechanics and Engineering},
  volume={26},
  number={2},
  pages={145--172},
  year={1981},
  publisher={Elsevier}
}

@article{chern2007coupling,
  title={A coupling interface method for elliptic interface problems},
  author={Chern, I-Liang and Shu, Yu-Chen},
  journal={Journal of Computational Physics},
  volume={225},
  number={2},
  pages={2138--2174},
  year={2007},
  publisher={Elsevier}
}

@article{boffi2010finite,
  title={Finite element approximation of eigenvalue problems},
  author={Boffi, Daniele},
  journal={Acta numerica},
  volume={19},
  pages={1--120},
  year={2010},
  publisher={Cambridge University Press}
}

@article{takahashi2009direct,
  title={Direct creation of three-dimensional photonic crystals by a top-down approach},
  author={Takahashi, Shigeki and Suzuki, Katsuyoshi and Okano, Makoto and Imada, Masahiro and Nakamori, Takeshi and Ota, Yuji and Ishizaki, Kenji and Noda, Susumu},
  journal={Nature materials},
  volume={8},
  number={9},
  pages={721--725},
  year={2009},
  publisher={Nature Publishing Group UK London}
}

@article{serbin2004fabrication,
  title={Fabrication of woodpile structures by two-photon polymerization and investigation of their optical properties},
  author={Serbin, Jesper and Ovsianikov, Aleksandr and Chichkov, Boris},
  journal={Optics Express},
  volume={12},
  number={21},
  pages={5221--5228},
  year={2004},
  publisher={Optical Society of America}
}

@article{feng2005imaging,
  title={Imaging properties of an elliptical-rod photonic-crystal slab lens},
  author={Feng, Shuai and Li, Zhi-Yuan and Feng, Zhi-Fang and Cheng, Bing-Ying and Zhang, Dao-Zhong},
  journal={Physical Review B—Condensed Matter and Materials Physics},
  volume={72},
  number={7},
  pages={075101},
  year={2005},
  publisher={APS}
}

@article{jiang1996origin,
  title={The origin of spurious solutions in computational electromagnetics},
  author={Jiang, Bo-nan and Wu, Jie and Povinelli, Louis A},
  journal={Journal of computational physics},
  volume={125},
  number={1},
  pages={104--123},
  year={1996},
  publisher={Elsevier}
}

@article{boffi1999computational,
  title={Computational models of electromagnetic resonators: analysis of edge element approximation},
  author={Boffi, Daniele and Fernandes, Paolo and Gastaldi, Lucia and Perugia, Ilaria},
  journal={SIAM journal on numerical analysis},
  volume={36},
  number={4},
  pages={1264--1290},
  year={1999},
  publisher={SIAM}
}

@book{ammari2018mathematical,
  title={Mathematical and computational methods in photonics and phononics},
  author={Ammari, Habib and Fitzpatrick, Brian and Kang, Hyeonbae and Ruiz, Matias and Yu, Sanghyeon and Zhang, Hai},
  volume={235},
  year={2018},
  publisher={American Mathematical Soc.}
}
